\documentclass{article}
\RequirePackage[OT1]{fontenc}
\RequirePackage{amsthm,amsmath}
\RequirePackage{hyperref}

\usepackage{graphicx}
\usepackage{multirow}
\usepackage{subfigure}
\usepackage{lipsum}
\usepackage{setspace}
\usepackage{geometry}
\usepackage{xcolor}
\usepackage[style = ext-authoryear, sorting = nty, isbn = false, doi = false, url = false, eprint = false, maxbibnames = 99, maxcitenames = 2, uniquename=false, firstinits=true]{biblatex}
\DeclareNameAlias{sortname}{family-given}
\renewbibmacro{in:}{}
\DeclareInnerCiteDelims{cite}{(}{)}

\newtheoremstyle{break}
  {\topsep}{\topsep}%
  {\itshape}{}%
  {\bfseries}{}%
  {\newline}{}%
\theoremstyle{break}

\newtheorem{theorem}{Theorem}
\newtheorem{remark}{Remark}
\newtheorem{definition}{Definition}

\newtheorem{lemma}{Lemma}
\newtheorem{example}{Example}
\newtheorem{corollary}{Corollary}

\newtheorem{algorithm}{Algorithm}
\title{Simultaneous Statistical Inference for Second  Order Parameters of Time Series under Weak Conditions}
\author{Yunyi Zhang \and Efstathios  Paparoditis \and Dimitris N. Politis}
\date{\begin{small}
School of Data Science, The Chinese University of Hong Kong, Shenzhen, Guangdong, China zhangyunyi@cuhk.edu.cn\\
\hspace*{-0.15cm} Department of Mathematics and Statistics, University of Cyprus, Nicosia,   CYPRUS,  stathisp@ucy.ac.cy\\
Department of Mathematics and Halicioglu Data Science Institute, University of California--San Diego,
La Jolla, US., dpolitis@ucsd.edu \end{small}
}

\begin{document}
\maketitle
\begin{abstract}
Strict stationarity is a   common assumption used in the time series literature in order to derive
 asymptotic distributional  results for   second-order statistics, like sample autocovariances and  sample autocorrelations. Focusing on  weak stationarity,
  this paper derives the asymptotic distribution of  the maximum of sample autocovariances and  sample autocorrelations under weak conditions by using  Gaussian approximation techniques.
The asymptotic theory for parameter estimation obtained by fitting a  (linear) autoregressive model to a general weakly stationary time series is revisited and  a Gaussian approximation theorem for the maximum of the estimators of the  autoregressive  coefficients  is derived.
To perform statistical inference for the second order parameters  considered, a  bootstrap algorithm, the so-called second-order wild bootstrap,  is applied.  Consistency of this bootstrap procedure is proven.
In contrast to existing bootstrap alternatives,   validity  of  the second-order  wild bootstrap does not require the imposition of strict stationary conditions  or structural process assumptions,  like linearity.
The good finite sample performance of the second-order wild bootstrap is demonstrated by means of simulations.
\end{abstract}
\section{Introduction}
Stationarity is a fundamental assumption in time series analysis which enables  statistical  inference  for stochastic processes based on  observations  obtained over  time.
Assuming that the observed time series stems  from a strict stationary process, researchers have developed different techniques to analyze  time series  data, like fitting    linear or nonlinear models,
performing statistical inference in the time
or in the frequency domain,
and forecasting  future observations of the process;
see the monographs of \cite{MR1093459}, \cite{MR1884963}, \cite{10.1093/acprof:oso/9780199641178.001.0001}, and \cite{time_series} for an overview. One way to relax the strict stationarity assumption  is  by allowing for   local  (strict) stationarity, that is by  enabling a time varying dependence structure,  which however, can locally (in time) be well approximated by that of a (strictly) stationary process; see \cite{MR1429916}, \cite{MR2190207}, \cite{MR3097614}, \cite{MR3920364}, \cite{MR4270034}.

Strict stationary, which requires that the joint distributions of  two segments $(X_{t_1},...,X_{t_k})^T$ and  $(X_{t_1 + h},...,X_{t_k + h})^T$ are the same for any $k = 1,2,...,\ t_1,...,t_k\in\mathbf{Z}$ and  any $h\geq 0$,  is quite restrictive and hard to verify in practice. In contrast, weak stationary  is more tractable since it only imposes restrictions   on  the time behavior  of the first and of the second order moments of the process, i.e., it requires that    $\mathbf{E}X_i = \mathbf{E}X_0$ and $\mathbf{E}X_iX_{i + j} = \mathbf{E}X_0X_j$ for any $i\in\mathbf{Z},\ j = 0,1,...$. Weak stationarity is easier to verify in practice and several proposals exists in the literature; see  among others, \cite{MR269062}, \cite{MR2724865}, \cite{MR3124795}, and \cite{MR3439530}.
Although  the definition of  second-order characteristics, like  autocovariances,  autocorrelations and  spectral densities (see \cite{MR1093459}), only requires  weak stationarity,  to  derive  distributional results for  sample estimators of the aforementioned  quantities,  strict stationary assumptions are
   commonly imposed  (e.g., \cite{MR2036395}, \cite{MR3235390}, and \cite{MR4325664}).
Furthermore, and as   Example~\ref{exp1} below demonstrates,  two time series may be  weakly stationary  having  identical
second order properties,  but the asymptotic distribution of   sample  quantities, like autocovariances and autocorrelations,   may be  quite different.

\begin{example}
Consider a first order autoregressive process, (AR(1)),
$X_i = \rho X_{i - 1} + \epsilon_i$.
Suppose the random variables $e_i, i\in\mathbf{Z}$,  are i.i.d. Gaussian  with mean $0$ and variance $1$ and consider the following three cases for the  white noise innovations $\epsilon_i$ driving the AR(1) process: 1. Linear: $\epsilon_i = e_i$; 2. Nonlinear: $\epsilon_i  =  e_i\cdot e_{i - 1}$; 3. Non-stationary: $\epsilon_{2i } = e_{2i}$ and $\epsilon_{2i + 1} = e_{2i+1}e_{2i}$. It holds true that
$\epsilon_i$ is white noise, i.e., that,  $\mathbf{E}\epsilon_i = 0$, $\mathbf{E}\epsilon_i\epsilon_j = 0$ for $i\neq j$ and $\mathbf{E}\epsilon^2_i = 1$. However, since $\mathbf{E}\epsilon^2_i\epsilon^2_{i - 1} = \mathbf{E}e^2_i\times \mathbf{E}e^2_{i - 2}\times \mathbf{E}e^4_{i - 1} = 3 \neq 1$, the innovations in Case 2 are dependent while those  in Case 3 are non-stationary.
Assume  $\vert\rho\vert < 1$ and observe that  $\mathbf{E}X_i = 0$ and $\mathbf{E}X_iX_{i+k} = \rho^k/(1 - \rho^2)$, $k \geq 0$, for all types of innovations considered. Table \ref{simu1} shows  for the case $\rho=0.7$, the variance of the Yule-Walker estimator $\widehat{\rho}=\widehat{\gamma}_h/\widehat{\gamma}_0$ as well as that of the sample autocovariance $\widehat{\gamma}_h = n^{-1}\sum_{i = h+1}^n X_iX_{i-h}$ for $ h =1$,   using  $n=10000$ and $n=100000$, observations.


\begin{table}[htbp]
\centering
\small
\caption{Estimates of the variance  of the  Yule-Walker estimator $\widehat{\rho}$ and of the  sample autocovariance $\widehat{\gamma}_1$ for the different types of white noise innovations considered.}
\begin{tabular}{l c c c c c}
\hline\hline
& & & & & \\
Innovations & $n$ & $\widehat{\rho}$ & $Var(\sqrt{n}(\widehat{\rho} - \rho))$ & $\widehat{\gamma}_1$ & $Var(\sqrt{n}(\widehat{\gamma}_1 - \gamma_1))$ \\
& & & & & \\
\hline
\textit{Linear} & 10000 & 0.70 & 0.52 & 1.45 & 20.74\\
                         & 100000 & 0.70 & 0.52 & 1.37  & 21.00\\
                         & & & & & \\
\textit{Nonlinear} & 10000 & 0.69 & 1.03 & 1.31 & 57.34\\
                                     & 100000 & 0.70 & 1.04 & 1.40 & 57.54\\
                                     & & & & & \\
\textit{Non-stationary} & 10000 & 0.71 & 1.58 & 1.49 & 70.51\\
                                    & 100000 & 0.70 & 1.53& 1.40 & 70.68\\
 \hline\hline
\end{tabular}
\label{simu1}
\end{table}

\label{exp1}
\end{example}

It is  a challenging issue  to  investigate the question whether it is possible to derive distributional results  in a time series context by  avoiding strict stationary assumptions, that is by   solely assuming  that the
observed time series is  generated by a weakly (but not necessarily strictly) stationary stochastic process. There have been several attempts in this direction in the literature; see among others  \cite{MR2893863}, \cite{MR2827528}, \cite{MR3931381}, \cite{MR4134800}, and  \cite{MR4206676}. However, to the best of our knowledge, there is relatively little research on deriving the asymptotic  distribution of
second-order statistics (like sample  autocovariances and sample autocorrelations)  and on bootstrapping a weakly but not necessarily strictly stationary time series. One of the contributions of this paper is  to show that for particular classes of statistics, like for the distribution of the maximum  of the aforementioned statistics, it is possible to derive distributional results by only invoking  weak stationarity assumptions.

Related to the above  inference problem is that  of deriving the   asymptotic properties of   parameter estimators when  an autoregressive model of order $p$,  ($AR(p)$ for simplicity),  that is,  the model,
\begin{equation}
X_t = \sum_{j = 1}^p a_jX_{t - j} + \epsilon_t,\ t\in\mathbf{Z},
\label{eq.def_AR_introduction}
\end{equation}
is fitted to a time series at hand.
Here the $a_j$'s  are fixed coefficients and the $\epsilon_t$ are  assumed to be white noise innovations, i.e., they are uncorrelated (not necessarily independent), zero mean  random variables with constant and positive  variance.
 The AR model (\ref{eq.def_AR_introduction}) reveals a linear regression relation between an observation $X_t$ and the previous observations $X_{t - j}, j = 1,...,p$, and is easy to fit, e.g., through solving the well-known system of Yule-Walker equations;
 see   Section 8.1 in \cite{MR1093459} and    \cite{MR2755014}. Since model (\ref{eq.def_AR_introduction})  is informative and easy to analyze, there has been an abundant literature on fitting  autoregressive models.
 However, it is common in the time series literature
 to assume  that  the observed time series  indeed  stems from  a linear process or   that equation  \eqref{eq.def_AR_introduction} holds true with the innovations  $\epsilon_t$ being i.i.d.. This  allows for the  use of the results obtained in  \cite{MR1093459} in order to infer properties of the autoregressive parameters. However, imposing structural  assumptions may be quite misleading  while   linearity assumptions are difficult to verify in practice and, as demonstrated in  Example~\ref{exp1},
 the asymptotic distribution of   Yule-Walker estimators $\widehat{\rho}$   may be quite different depending on the stochastic properties  of the white noise innovations $ \epsilon_t$  driving equation (\ref{eq.def_AR_introduction}).

Notably, one  can think of equation  (\ref{eq.def_AR_introduction}) as a truncated version of an autoregressive representation of the underlying process which has a rich background beyond the linear process class.
Similar to the Wold decomposition, \cite{MR2893863}
used the fact that stationary process  which possess a strictly positive and bounded spectral density,   obey  an autoregressive representation, i.e., they satisfy
eq.\eqref{eq.def_AR_introduction} with $p  =\infty$ and  $\epsilon_t$ being  white noise  instead of
i.i.d.   Notice that such an autoregressive representation of a weakly stationary process is also  useful for prediction. To elaborate,  let $\mathcal{H}_{t-1,t-p}$ be the Hilbert space
spanned by $X_{t-1}, X_{t-2},...,X_{t-p}$ with the definition of inner product $< a, b > = \mathbf{E}a\times b$. Then it is well known that the (mean square) optimal linear predictor of $X_t$ is
given by $\sum_{j = 1}^p \beta_jX_{t-j}$, where the coefficients  $\beta_j, j = 1,...,p$, satisfy the Yule-Walker equations. Hence, fitting an $AR(p)$ model to a (non-necessarily linear) weakly stationary time series may be  justifiable and seems to be helpful in
understanding the linear dependence structure of the underlying process or  for making prediction.  However, since  we do not want to assume that the underlying  time series is strictly stationary and linear, the asymptotic distribution of the corresponding Yule-Walker estimator is  different from the one valid  in the linear process case.  In this  paper we will  focus on  the asymptotic distribution of the maximum of
Yule-Walker estimators under weak stationarity assumptions and will discuss some important applications to inference. Moreover, we will apply a bootstrap procedure to automatically obtain   consistent, simultaneous confidence intervals for the autoregressive  parameters in (\ref{eq.def_AR_introduction}).

The  paper is organized as follows. Section \ref{section.medium_range} lays some foundations and introduces a new kind of non-stationary random variables, called the $(m,\alpha,\beta)$-medium range dependent random variables. Some useful properties of such sequences of  random variables are discussed. Section
\ref{section.autocovariance} derives  the asymptotic distribution of the maximum of sample autocovariances and  sample autocorrelations using Gaussian approximation techniques. Section \ref{section.AR_coef}
uses the \mbox results \mbox  obtained in the previous section to obtain the  asymptotic distribution of the  maximum of Yule-Walker estimators when an AR(p) model is fitted to a general stationary time series. Section \ref{section.bootstrap} considers  simultaneous statistical inference  for
autocovariances,  autocorrelations, and   coefficients of  a linear  autoregressive model fitted to a general stationary time series. Furthermore,
a bootstrap algorithm,
called  the `second-order wild bootstrap', is proposed,  which is  used to  perform statistical inference for the aforementioned second order parameters.
Section \ref{section.numerical_experiment} presents some  numerical results for simulated and real-life data,  which demonstrate   the finite sample performance
of the proposed bootstrap procedure. Conclusions are given in Section \ref{section.conclusion}  while all proofs and technical calculations are deferred  to the Supplementary  Material.

\textbf{Notation: } This paper uses the  standard order notation $O(\cdot),\ o(\cdot),\ O_p(\cdot),\ o_p(\cdot)$. That is, for two numerical sequences $a_n, b_n, n = 1,2,...$, we write $a_n = O(b_n)$ if there exists a constant $C > 0$ such that $\vert a_n\vert\leq C\vert b_n\vert$ for all $n$; and $a_n = o(b_n)$ if $\lim_{n\to\infty}a_n/b_n = 0$. For two sequences of random variables $X_n, Y_n, n = 1,2,...$, we write $X_n = O_p(Y_n)$ if for any $\varepsilon>0$,  there exists a constant $C_\varepsilon$ such that $Prob\left(\vert X_n\vert\leq C_\varepsilon \vert Y_n\vert\right)\geq 1 - \varepsilon$ for any $n$; and $X_n = o_p(Y_n)$ if $X_n / Y_n\to_p 0$ where the latter denotes convergence in probability.
All order notations as well as  convergence results are understood to hold true as the sample size $n\to\infty$. The symbol $\exists $ and $\forall$ respectively represents `there exists' and `for all'. For a vector $a = (a_1,...,a_p)^T\in\mathbf{R}^p$, we use  the norm $\vert a\vert_q = (\sum_{i = 1}^p \vert a_i\vert^q)^{1/q}$ with $q\geq 1$ and we define  $\vert a\vert_\infty = \max_{i = 1,...,p}\vert a_i\vert$. For a matrix $M\in\mathbf{R}^{p\times p}$,  $\vert T\vert_2 = \max_{\vert a\vert_2 = 1}\vert Ta\vert_2$ denotes   the   operator norm. For a finite set $A$,  $\vert A\vert$ denotes the number of elements of  $A$ and for  a random variable $X$, we define its $m$ norm as $\Vert X\Vert_m = (\mathbf{E}\vert X\vert^m)^{1/m}$, where  $m\geq 1$. Furthermore,   $a\vee b = \max(a,b)$ and $a\wedge b = \min(a,b)$. Finally, we  use the notation $C$, $C^\prime$, $C^{\prime\prime}$,  to represent generic constants, i.e., the value of  $C$, $C^\prime$, $C^{\prime\prime}$ may be different  in different places.

\section{Medium range dependent random sequences}
\label{section.medium_range}
This section introduces a new kind of dependence for possible   non-stationary random variables,  called the $(m,\alpha,\beta)-$ medium range dependence  and  derives some useful properties of  sequences of  random variables satisfying this kind of dependence. In particular, we will establish  a Gaussian approximation theorem  and a consistent estimation result for  the covariance matrix of  linear combinations of $(m,\alpha,\beta)-$ medium range dependent random variables. Deriving the asymptotic distribution of linear combinations of random variables is an important  step in  investigating distributional properties of  many types of statistics not only for time series but also for i.i.d. data, for instance like those appearing in  linear and nonlinear regression.
Therefore, the results presented in this section are  of interest on their own.

Let  $e_i, i\in\mathbf{Z}$ be  independent (but non-necessarily identically distributed) random variables and consider  random variables $X_{i,j}$ generated as
\begin{equation}
X_{i,j} = g_{i,j}(..., e_{i - 2}, e_{i - 1}, e_i),\ \text{for $i\in\mathbf{Z}$ and } j = 1, 2,...,d.
\label{eq.defEps}
\end{equation}
Suppose the observed data consists of the set  $X_{i,j}$ for  $i = 1,2,...,T$ and $j = 1,2,...,d$. Since the functions $g_{i,j}$ can vary with respect to the indices  $i$ and $j$ and the $e_i$ may have different distributions for different $i$,  equation (\ref{eq.defEps}) can be used to  describe a variety of time series, also including    non-stationary time series.
In fact, our approach  allows the measurable functions $g_{i,j}$ to change with respect to the sample size $T$, i.e., \textit{$g_{i,j}$,  $ i\in\mathbf{Z}$, $ j = 1,2, ...,d$,  may vary with respect to different sample size $T$}. To simplify notation, in the following we will not explicitly stress the dependency of $g_{i,j}$ on $T$  and
if $d = 1$, then we will omit the subscript $j$ and simple write  $X_i$ for $X_{i, 1}$.

\begin{remark}
The generating equation  \eqref{eq.defEps} has been  used in  \cite{MR2172215}, and has become a useful framework especially when dealing  with nonlinear time series, see e.g.,
\cite{MR2351105}, \cite{MR3319142}, \cite{MR3310530}, and \cite{MR3819347}. In a previous form  with $g$ does not depending on $i$ and $j$, such an  equation   was  used  to model univariate, strictly stationary time series.  Recently there have been generalizations to vector or high dimensional time series, e.g, \cite{MR3718156}; and non-stationary time series,  where $g$ is allowed to vary with $i$ and $j$. In this section
we also generalize results given in \cite{MR2172215} to  non-stationary time series.
\end{remark}

\begin{example}[Locally stationary process]
Recently, \cite{MR4206676} introduced a framework for  locally stationary process that can be considered as a special case of \eqref{eq.defEps}. To elaborate, consider  $d$ measurable functions $f_j(x, ..., y_2, y_1)$,   $ j = 1,...,d$, and    assume  that $X_{i,j} = f_j(i/T, ..., e_{i - 1}, e_i)$. Setting   $g_{i,j}(\cdot) = f_j(i/T, \cdot)$, such a  locally stationary process  is a  special case of \eqref{eq.defEps}.
\end{example}

Define a  sequence of random variables  $e_i^\dagger,\ i\in\mathbf{Z}$,  such that $e_i^\dagger$ are independent from  each other, $e_i^\dagger$ has the same distribution as $e_i$, and $e_i, e^\dagger_j$ are mutually independent for any $i,j\in\mathbf{Z}$. Define the filter $\mathcal{F}_i$ as the $\sigma$-field generated by $..., e_{i - 2}, e_{i - 1}, e_i$ and $\mathcal{F}_{i,k}$ as the $\sigma$-field generated by $e_{i - k}, e_{i - k + 1},...,e_i$, where $i\in\mathbf{Z}$ and $k\geq 0$. Using  these definitions,     $X_{i,j}$ is $\mathcal{F}_i$ measurable. For any $k\in\mathbf{Z}$, define
\begin{equation}
\begin{aligned}
X_{i,j}(k) =
\begin{cases}
g_{i,j}(..., e_{i - k - 1}, e^\dagger_{i - k}, e_{i - k + 1},..., e_i)\ \text{if $k\geq 0$}\\
X_{i,j}\ \text{if $k < 0$.}
\end{cases}
\end{aligned}
\label{eq.defX_ik}
\end{equation}
Furthermore, for any $m\geq 1$, define
\begin{equation}
\delta_{i,j,m}(k) = \Vert X_{i,j} - X_{i,j}(k)\Vert_m\ \text{and } \delta_{m}(k) = \sup_{i\in\mathbf{Z}, j = 1,...,d}\delta_{i,j,m}(k)
\end{equation}
Similar to \eqref{eq.defEps}, we omit the subscript $j$ if $d = 1$, i.e., we write  $\delta_{i,m}(k) = \Vert X_i - X_i(k)\Vert_m$, (recall that $X_i = X_{i,1}$ if $d = 1$). According to \eqref{eq.defX_ik}, $\delta_{i,j,m}(k) = \delta_m(k) = 0$ if $k < 0$. Definition \ref{def.medium_range} introduces the kind of dependent random variables considered  in this paper.

\begin{definition}[$(m,\alpha,\beta)$-medium range dependent random variables]
Suppose that the  random variables $X_{i,j}, i\in\mathbf{Z}, j = 1,...,d$ satisfy \eqref{eq.defEps}. Let $X_{i,j}, i = 1,2,...,T$, $j =  1,...,d$ be  the observed time series and $m,\alpha,\beta$   constants such that $m\geq 4,\ \alpha > 1,\ \beta\geq 0$. We say that the random variable $X_{i,j}$ are $(m,\alpha,\beta)$-medium range dependent  if
\begin{equation}
\begin{aligned}
\mathbf{E}X_{i,j} = 0\ \text{for all $i,j$}\\
\sup_{i\in\mathbf{Z}, j = 1,...,d}\Vert X_{i, j}\Vert_m = O(1)\\
\text{and } \sup_{k = 0,1,...}(1 + k)^\alpha\sum_{l = k}^\infty \delta_m(l) = O(T^\beta)
\end{aligned}
\label{eq.def_medium_range}
\end{equation}
\label{def.medium_range}
\end{definition}

\begin{remark}[$(m,\alpha)$-short range dependent random variables]
\label{remark.short_range}
A special case of  Definition \ref{def.medium_range} is when $\beta = 0$, i.e., $\sup_{k = 0,1,...}(1 + k)^\alpha\sum_{l = k}^\infty \delta_m(k) = O(1)$. If this happens, we call $X_{i,j}$,  $(m,\alpha)$-short range dependent random variables.
\end{remark}

\begin{remark}
    Assume that $e_i, i\in\mathbf{Z}$ are independent and identically distributed, and $g_{i,j} = g_j$, i.e., $g_{i,j}$ does not change with respect to the index $i$. In that case,
    \begin{align*}
        \delta_{i,j,m}(k) = \Vert X_{k, j, m} - X_{k,j,m}(k)\Vert_m\ \text{is independent of } i,
    \end{align*}
    and the third line in eq.\eqref{eq.def_medium_range} coincides with \cite{MR2172215}.
\end{remark}

Suppose that $i_1 < i_2$. Then,  from Cauchy-Schwarz's inequality, we have,
\begin{equation}
\begin{aligned}
\vert\mathbf{E}X_{i_1, j_1}X_{i_2, j_2}\vert = \vert\mathbf{E}X_{i_1, j_1}(X_{i_2, j_2} - \mathbf{E}X_{i_2, j_2}|\mathcal{F}_{i_2, i_2 - i_1 - 1})\vert\\
\leq C\Vert X_{i_2, j_2} - \mathbf{E}X_{i_2, j_2}|\mathcal{F}_{i_2, i_2 - i_1 - 1}\Vert_m\leq C\sum_{s = i_2 - i_1}^\infty\delta_{i_2, j_2, m}(s)\leq \frac{C^\prime T^\beta}{(1 + i_2 - i_1)^\alpha}.
\end{aligned}
\label{eq.cov}
\end{equation}
That is,  Definition \ref{def.medium_range} implies that the autocovariances $\gamma_{i_1,i_2}$ have a polynomial decay with respect to the time lag  $i_2-i_1$.

We next  derive some useful properties of $(m,\alpha,\beta)$-medium range dependent random variables. The first result gives a bound for the moments of linear combinations of such random variables. This result is similar to that of \cite{MR0133849} where the corresponding bounds  assume that the random variables involved are
 independent.

\begin{lemma} Suppose $X_{i,j}, i\in\mathbf{Z}, j = 1,...,d$ are $(m,\alpha,\beta)$- medium range dependent random variables with $m\geq 4$, $\alpha > 1$ and $\beta\geq 0$. Let   $X_{1,j},...,X_{T,j}, j = 1,...,d$, be  the time series observed and  let $a_1,...,a_T\in\mathbf{R}$. Then,
\begin{equation}
\begin{aligned}
(i)\  \ \max_{j = 1,...,d}\Vert\sum_{i = 1}^T a_iX_{i,j}\Vert_m = O\left(T^\beta\times \sqrt{\sum_{i = 1}^T a^2_i}\right)\\
(ii)\ \  \sup_{j = 1,...,d, s \geq 0}(1 + s)^\alpha\Vert\sum_{i = 1}^T a_i(X_{i,j} - \mathbf{E}X_{i,j}|\mathcal{F}_{i,s})\Vert_m = O\left(T^\beta\times \sqrt{\sum_{i = 1}^T a^2_i}\right),
\end{aligned}
\label{eq.linear_comb}
\end{equation}
where  $s\geq 0$ is an integer.
\label{lemma.linear_comb}
\end{lemma}
For the  special case of (\ref{lemma.linear_comb})  where the  random variables $X_{i,j}$ are $(m,\alpha)-$ short range dependent,  we have,
\begin{equation}
\begin{aligned}
\max_{j = 1,...,d}\Vert\sum_{i = 1}^T a_iX_{i,j}\Vert_m = O\left(\sqrt{\sum_{i = 1}^T a^2_i}\right)\\
\text{and } \sup_{j = 1,...,d, s \geq 0}(1 + s)^\alpha\Vert\sum_{i = 1}^T a_i(X_{i,j} - \mathbf{E}X_{i,j}|\mathcal{F}_{i,s})\Vert_m = O\left(\sqrt{\sum_{i = 1}^T a^2_i}\right)
\end{aligned}
\end{equation}

If $X_i, i\in\mathbf{Z}$ are $(m,\alpha)-$ short range dependent random variables, then the following lemma  shows that the centered product $(X_iX_{i - j} - \mathbf{E}X_iX_{i - j})$, for some given $j$,  will also  be a sequence of medium range dependent random variables. This property allows for using
results obtained in this section
to discuss asymptotic properties  of  second-order statistics, like  sample autocovariances and sample autocorrelations.

\begin{lemma}
\label{lemma.recognize}
Suppose $X_i, i\in\mathbf{Z}$ are $(m,\alpha)$-short range dependent,  random variables
with $m\geq 8$ and $\alpha > 2$ and let $X_1,...,X_T$ be the observations available.

(i) \ \  Suppose $d = O(T^\mathcal{B})$ is an non-negative integer where $0\leq \mathcal{B} < 1$. Let $\{a_{ij}\}_{i = 1,...,p_1,}$\\
\noindent$\ _{j = 0,...,d}\in\mathbf{R}^{p_1\times (d + 1)}$ satisfy $\max_{i = 1,...,p_1}\sum_{j = 0}^d a_{ij}^2 = O(1)$, where $p_1$ is an integer (it can be arbitrarily large). Define $Z_{i,k} = \sum_{j = 0}^d a_{kj}(X_iX_{i - j} - \mathbf{E}X_iX_{i - j})$. Then $Z_{i,k}, i \in\mathbf{Z}, k = 1,...,p_1$,  are $(m/2, \alpha - 1, \alpha\mathcal{B})$-medium range dependent random variables.

(ii) \ \  Suppose $X_i, i\in\mathbf{Z}$ are white noise random variables, i.e., $\mathbf{E}X^2_i = \sigma^2$,  $0<\sigma^2 < \infty$,  and $\mathbf{E}X_iX_j = 0$ for $i\neq j$. Furthermore,  suppose that $p = O(1)$ is an non-negative integer and let $a_1,...,a_p\in\mathbf{R}$ such that  $P(x) = 1 - \sum_{j = 1}^p a_j x^j\neq 0\ \text{for all } x\in\mathbf{C}, \vert x\vert\leq 1$. Define the autoregressive random variables $Y_i, i\in\mathbf{Z}$,  as
\begin{equation}
Y_i = X_i + \sum_{j = 1}^p a_j Y_{i - j}.
\end{equation}
Then the random variables $Y_i, i\in\mathbf{Z}$ are $(m,\alpha)$-short range dependent.
\end{lemma}

If $\mathcal{B} = 0$ in (i), that is if $d$ is constant, then the random variables $Z_{i,k}$ in Lemma \ref{lemma.recognize} will be $(m/2, \alpha - 1)$-short range dependent.
However, if $\mathcal{B} > 0$, i.e., if the dimension  $d$ grows as the sample size $T$ increases,  then the products $Z_{i, k}$ will no longer consist  of  short range
dependent random variables. In this case we can adopt the medium range dependent concept.
As it will be demonstrated in Section~\ref{section.autocovariance} below, Lemma~\ref{lemma.recognize} is useful for allowing  the maximum lag $d$ considered to grow as the sample size $T$ increases,  when we deal with   the asymptotic distribution of  sample autocovariances and sample
autocorrelations.

The third result of this section presents  a Gaussian approximation theorem for the sample autocovariances of $(m,\alpha)$-short range dependent random variables. Notice that  various central limit theorems have been established regarding  the asymptotic distribution of  sample  autocovariances; see, for instance,  \cite{MR1700749} and \cite{MR2002723}. However, the stationarity conditions of these central limit theorems do not cover non-stationary time series. Recently, \cite{MR3161448} proposed a Gaussian approximation  theorem which is  suitable for more complex situations; also see \cite{MR3718156} and \cite{ZHANG}. Theorem \ref{lemma.Gaussian} bellow allows for the random variables considered to be non-stationary (i.e., they may have different marginal distributions) and it is  more suitable for our setting.

\begin{theorem}
Suppose $X_i, i\in\mathbf{Z}$,  are $(m,\alpha)$-short range dependent random variables with $m\geq 8$ and $\alpha > 2$, $X_1,...,X_T$ is the observed time series
and $d = O(T^\mathcal{B})$ is a positive integer with $\mathcal{B}\geq 0$. Let  $a_{ij}, i = 1,...,p_1, j = 0,1,...,d$ be real numbers satisfying
\begin{equation}
 \max_{i = 1,...,p_1}\sum_{j = 0}^d  a_{ij}^2 = O(1)\ \text{and } p_1 = O(T^{\alpha_{p_1}}),
\end{equation}
where $\alpha_{p_1}\geq 0$ is a constant. Define $Z_{i, k} = \sum_{j = 0}^d a_{kj}(X_iX_{i - j} - \mathbf{E}X_iX_{i  -j})$ for
$i\in\mathbf{Z}$ and $k =  1,...,p_1$. Assume  that
\begin{equation}
\Vert\frac{1}{\sqrt{T}}\sum_{i = 1}^T Z_{i,k}\Vert_2  > c\ \text{for } \ \  k = 1,...,p_1,
\end{equation}
where $c > 0$ is a constant. If there exist  two constants $\alpha_s, \alpha_l$ such that $0 < \alpha_s < \alpha_l < 1$ and
\begin{equation}
\begin{aligned}
\frac{2\alpha_{p_1}}{m} + 2\alpha\mathcal{B} < (\alpha - 1)\alpha_s,\ \frac{4\alpha_{p_1}}{m} + 4\alpha\mathcal{B} + \alpha_s < \alpha_l,\
\frac{12\alpha_{p_1}}{m} + 12\alpha\mathcal{B} + \alpha_l < 1\\
\text{and } 4\alpha\mathcal{B} < (\alpha - 1)\alpha_s,\ 8\alpha\mathcal{B} + \alpha_s < \alpha_l,
\end{aligned}
\label{eq.cond_alpha}
\end{equation}
then,
\begin{equation}
\sup_{x\in\mathbf{R}}\vert Prob\left(\max_{k = 1,...,p_1}\vert\frac{1}{\sqrt{T}}\sum_{i = 1}^T Z_{i, k}\vert\leq x\right)
 -  Prob\left(\max_{k = 1,...,p_1}\vert\xi_k\vert\leq x\right)\vert = o(1),
\label{eq.delta_prob}
\end{equation}
where $\xi_k, k = 1,...,p_1$,  are joint normal distributed random variables with $\mathbf{E}\xi_k = 0$ and
\noindent
$\mathbf{E}\xi_{k_1}\xi_{k_2} = T^{-1}\sum_{i_1 = 1}^T\sum_{i_2 = 1}^T \mathbf{E}Z_{i_1, k_1}Z_{i_2, k_2}$.
\label{lemma.Gaussian}
\end{theorem}

\begin{remark}
    As demonstrated in the supplementary material, $\alpha_s$ and $\alpha_l$   control the sizes of the `small blocks' and  `big blocks', respectively, used  in the method of proof of Theorem \ref{lemma.Gaussian}.
    The reason why   eq.\eqref{eq.cond_alpha} looks quite  involved   is that  the number $d$, which will later coincide   with the total number of autocovariances, respectively, autocorrelations  considered, is allowed  to increase  to infinity as the sample size $T$  increases to infinite.  Recall that $Z_{i, k} = \sum_{j = 0}^d a_{kj}(X_iX_{i - j} - \mathbf{E}X_iX_{i  -j})$. Therefore, even if $X_i$ is short-range dependent, the  dependence between $Z_{i_1,k}$ and $Z_{i_2, k}$ can be strong  for $\vert i_1 - i_2\vert$  large, and this  affects the asymptotic distribution of the sum $ T^{-1/2}\sum_{i = 1}^T Z_{i, k}$. Regarding the parameter $ p_1$,
    this parameter  represents the number of simultaneous linear combinations taken into consideration, so $p_1$, correspondingly $\alpha_{p_1}$, does not have to depend on $d$. However, in Section \ref{section.autocovariance} and Section \ref{section.AR_coef}, we estimate the autocovariances, autocorrelations, and AR coefficients with maximum lag $d$, respectively, $p$. Therefore, the number of simultaneous linear combinations considered in these sections are smaller  or equal to $d$. Notice  that  condition \eqref{eq.cond_alpha} essentially imposes some restriction on how fast $ d$  and $p_1$ are  allowed to  increase to infinity with $T$ and that   \eqref{eq.cond_alpha} is feasible. For  instance, for  $m = 20, \alpha = 2.5$, $\alpha_{p_1} = 1/25$ and $\mathcal{B} = 1/80$, $\alpha_s $ and $ \alpha_l $   can  be  chosen as  $\alpha_s = 1 / 6$ and $\alpha_l = 1/2$ so that  \eqref{eq.cond_alpha}  is fulfilled.

    However, if   $\mathcal{B} = 0$, which implies $d = O(1)$, that is, if only a fixed number of autocovariances, respectively autocorrelations is considered, then  eq.\eqref{eq.cond_alpha} simplifies considerably. It becomes
    \begin{align*}
    \frac{2\alpha_{p_1}}{m} < (\alpha - 1)\alpha_s,\ \frac{4\alpha_{p_1}}{m} + \alpha_s < \alpha_l,\
\frac{12\alpha_{p_1}}{m}  + \alpha_l < 1
    \end{align*}
    which  implies that the conditions of Theorem~\ref{lemma.Gaussian} are fulfilled if
    \[ \frac{2\alpha_{p_1}}{m}\frac{1}{(\alpha-1)} < \alpha_s < \alpha_l - \frac{4\alpha_{p_1}}{m}<  \alpha_l < 1 - \frac{12\alpha_{p_1}}{m}. \]
    Finally, if also $ \alpha_{p_1}=0$, that is if the number of simultaneous linear combinations considered is fixed, then \eqref{eq.cond_alpha} degenerates to the simple requirement $ 0<\alpha_s < \alpha_l <1$.
\end{remark}


Consider next the problem of estimating the covariances of the random variables $Z_{i,k}$, defined in Theorem \ref{lemma.Gaussian}. In order to construct an estimator,
we introduce a kernel function $K$  which has  the following properties.

\begin{definition}[Kernel function]
$K(\cdot): \mathbf{R}\to [0,\infty)$ is a symmetric, continuously di-fferentiable  function, with $K(0) = 1$, $\int_{\mathbf{R}}K(x)dx < \infty$ and $K(x)$  decreasing on $[0,\infty)$. Define the Fourier transformation of $K$ as $\mathcal{F}K(x) = \int_{\mathbf{R}}K(t)\exp(-2\pi \mathrm{i} tx)dt$, where $\mathrm{i}= \sqrt{-1}$. Assume $\mathcal{F}K(x)\geq 0$ for all $  x\in\mathbf{R}$ and $\int_{\mathbf{R}}\mathcal{F}K(x)dx<\infty$.
\label{def.kernel}
\end{definition}
 \noindent

\begin{remark}
Definition \ref{def.kernel} is a  bit 
stronger than the usual definition of a kernel function (see e.g., \cite{MR1865334}) but ensures  some desirable properties.
According to \cite{MR2656050} and the Fourier inversion theorem (e.g., Theorem 8.26 in \cite{MR1681462}), $\forall x = (x_1,...,x_n)^T\in\mathbf{R}^n$ and any positive number $k$,
\begin{equation}
\begin{aligned}
\sum_{s = 1}^n\sum_{j = 1}^n x_sx_jK\left(\frac{s - j}{k}\right) = \int_{\mathbf{R}}\sum_{s = 1}^n\sum_{j = 1}^nx_sx_j\mathcal{F}K(z)\exp\left(2\pi\mathrm{i}z\frac{s - j}{k}\right)dz\\
=\int_{\mathbf{R}}\mathcal{F}K(z)\Big|\sum_{s = 1}^n x_s\exp\left(\frac{2\pi\mathrm{i}zs}{k}\right)\Big|^2 dz\geq 0,
\end{aligned}
\end{equation}
that is,  the matrix $\left\{ K\left(\frac{s - j}{k}\right)\right\}_{s,j  = 1,2,...,n}$ is positive semi-definite. One possible kernel function is $K(x)  = \exp(-x^2/2)$, whose Fourier transform is $\mathcal{F}K(x) = \sqrt{2\pi}\exp(-2\pi^2x^2)$.
\end{remark}

\begin{lemma}
Let $X_i,i\in\mathbf{Z}$, be  random variables and $a_{ij}, i = 1,...,p_1, j = 0,1,...,d$ be real numbers. The integers $d, p_1$ and the random variables $Z_{i,k}, i\in\mathbf{Z}, k = 1,...,p_1$
satisfy the conditions of Theorem \ref{lemma.Gaussian}. In addition, suppose $K\left(\cdot\right): \mathbf{R}\to [0,\infty)$ is a
kernel function satisfying  Definition \ref{def.kernel}, $k_T > 0$ is a bandwidth and $k_T\to\infty $ as $T\to\infty$. Then,
\begin{equation}
\begin{aligned}
\max_{j_1,j_2 = 1,...,p_1}\Big|\frac{1}{T}\sum_{i_1 = 1}^T\sum_{i_2 = 1}^T Z_{i_1, j_1}Z_{i_2, j_2}K\left(\frac{i_1 - i_2}{k_T}\right)
 - \frac{1}{T}\sum_{i_1 = 1}^T\sum_{i_2 = 1}^T\mathbf{E}Z_{i_1, j_1}Z_{i_2, j_2}\Big|\\
= O_p(v_T\times T^{\alpha\mathcal{B}} + k_T\times T^{\frac{8\alpha_{p_1}}{m} + \alpha\mathcal{B} - \frac{1}{2}})
\end{aligned}
\label{eq.cov_converge}
\end{equation}
where  $$
v_T = \begin{cases}
k_T^{2 -\alpha}\ \text{if } 2 < \alpha < 3\\
 \log(k_T)/ k_T \ \text{if }  \alpha = 3\\
1/k_T\ \text{if }  \alpha > 3.
\end{cases}
$$
\label{lemma.covariance}
\end{lemma}
According to equation \eqref{eq.cond_alpha}, $16\alpha_{p_1}/m + 16\alpha\mathcal{B} < 1$. Therefore, if $\alpha > 9/4$, then we may choose $k_T = T^{4\alpha\mathcal{B}}$
and  we have
$$\max_{j_1,j_2 = 1,...,p_1}\Big|\frac{1}{T}\sum_{i_1 = 1}^T\sum_{i_2 = 1}^T Z_{i_1, j_1}Z_{i_2, j_2}K\left(\frac{i_1 - i_2}{k_T}\right) - \frac{1}{T}\sum_{i_1 = 1}^T\sum_{i_2 = 1}^T\mathbf{E}Z_{i_1, j_1}Z_{i_2, j_2}\Big|\\
 = o_p(1).$$

\section{Asymptotic distribution of the maximum of  sample autocovariances and autocorrelations}
\label{section.autocovariance}
This section goes back to the original problem, namely that  of estimating the autocovariances of a weakly, but not necessarily strictly stationary
time series. Suppose that $\{X_i, i\in\mathbf{Z}\}$ form a stochastic process with
$\mathbf{E}X_i = 0$ and
$X_1,X_2,...,X_T$ is the time series observed, (similar to $X_i$ forming a triangular array, the underlying process  may be different for  different sample sizes $T$).
Define $\sigma_j = \mathbf{E}X_iX_{i - j} = \mathbf{E}X_0X_{-j}$. Our purpose is to generate a (simultaneous) confidence interval or to perform  hypothesis tests for the autocovariances $\sigma_j$
as well as for the  autocorrelations  $\rho_j = \sigma_j/\sigma_0$. Toward this goal, we use the sample autocovariances and the sample autocorrelations
as estimators of $ \sigma_j$ and $\rho_j$, where
\begin{equation}
\widehat{\sigma}_j   = \frac{1}{T}\sum_{i = j+1}^T X_iX_{i  - j}\ \text{and  }  \widehat{\rho}_j = \widehat{\sigma}_j/\widehat{\sigma}_0.
\label{eq.def_sigma}
\end{equation}
Theorem \ref{theorem.auto_covariance} derives a Gaussian approximation result  for  these  estimators.
\begin{theorem} \label{th.th2}
Suppose $\{X_i, i\in\mathbf{Z}\}$ are $(m,\alpha_X)-$short range dependent random variables with $m\geq 8,\ \alpha_X > 2$.
Suppose further that $d = O(T^{\beta_X})$ is a positive integer and the sets $\mathcal{H}\subset\{0,1,...,d\},\ \mathcal{I}\subset\{1,2,...,d\}$ are not empty. Assume that
there exists constants $0 < \alpha_s < \alpha_l < 1$ such that
\begin{equation}
\begin{aligned}
\frac{2\beta_X}{m} + 2\alpha_X\beta_X < (\alpha_X - 1)\alpha_s,\ \frac{4\beta_{X}}{m} + 4\alpha_X\beta_X + \alpha_s < \alpha_l,\
\frac{12\beta_{X}}{m} + 12\alpha_X\beta_X + \alpha_l < 1\\
\text{and } 4\alpha_X\beta_X < (\alpha_X - 1)\alpha_s,\ 8\alpha_X\beta_X + \alpha_s < \alpha_l
\end{aligned}
\label{eq.condH}
\end{equation}

\noindent (i) If  a constant $c>0$ exists such that
\begin{equation}
\Vert\frac{1}{\sqrt{T}}\sum_{i = 1}^T (X_iX_{i - j} - \sigma_{j})\Vert_2 > c\ \text{for $j\in\mathcal{H}$,}
\label{eq.minVARAUTO}
\end{equation}
then,
\begin{equation}
\sup_{x\in\mathbf{R}}\Big| Prob\left(\max_{j\in\mathcal{H}}\sqrt{T}\vert\widehat{\sigma}_j - \sigma_j\vert\leq x\right)
 -  Prob\left(\max_{j\in\mathcal{H}}\vert \xi_j\vert\leq x\right)\Big| = o(1),
 \label{eq.auto_covariance_gaussian}
\end{equation}
where $\xi_j,j\in\mathcal{H}$ are joint Gaussian  random variables with  $\mathbf{E}\xi_j = 0$
and

\noindent $\mathbf{E}\xi_{j_1}\xi_{j_2} = T^{-1}\sum_{i_1 = 1}^T\sum_{i_2 = 1}^T\mathbf{E} (X_{i_1}X_{i_1 - j} - \sigma_{j}) (X_{i_2}X_{i_2 - j} - \sigma_{j})$.

\noindent (ii) Define
\[Z_{i,j} = - \frac{\sigma_j}{\sigma^2_0}(X_i^2 - \sigma_0) + \frac{1}{\sigma_0}(X_iX_{i - j} - \sigma_j)\]
and suppose there exists a constant $c > 0$ such that
\begin{equation}
\Vert\frac{1}{\sqrt{T}}\sum_{i = 1}^T Z_{i,j}\Vert_2 > c\ \text{for $j\in\mathcal{I}$}\ \text{and $\sigma_0 > c$.}
\label{eq.minVARCOEF}
\end{equation}
Then
\begin{equation}
\sup_{x\in\mathbf{R}}\Big| Prob\left(\max_{j\in\mathcal{I}}\sqrt{T}\vert\widehat{\rho}_j - \rho_j\vert\leq x\right)
 -  Prob\left(\max_{j\in\mathcal{I}}\vert\zeta_j\vert\leq x\right)\Big| = o(1),
\label{eq.rho_central}
\end{equation}
where  $\zeta_j,j\in\mathcal{I}$ are joint Gaussian  random variables with  $\mathbf{E}\zeta_j = 0$
and

\noindent $\mathbf{E}\zeta_{j_1}\zeta_{j_2} = T^{-1}\sum_{i_1 = 1}^T\sum_{i_2 = 1}^T \mathbf{E}Z_{i_1,j_1}Z_{i_2,j_2}$.
\label{theorem.auto_covariance}
\end{theorem}

\begin{remark}
    For $d$  growing to infinity as the sample size $T$ increases, the maximum $\max_{j\in\mathcal{H}}\vert\xi_j\vert$ and $\max_{j\in\mathcal{I}}\vert\zeta_j\vert$ may  not necessarily converge to a fixed distribution; see \cite{MR2234156}. However, Theorem \ref{theorem.auto_covariance} says that the difference between the two distributions considered is  asymptotically negligible. Therefore, and for a  data set given, statisticians can still use theorem \ref{theorem.auto_covariance} to obtain the desired confidence bands.
\end{remark}

Theorem \ref{theorem.auto_covariance} approximates the distribution of the estimation roots $\max_{j\in\mathcal{H}}\sqrt{T}\vert\widehat{\sigma}_j - \sigma_j\vert$ and $\max_{j\in\mathcal{I}}\sqrt{T}\vert \widehat{\rho}_j - \rho_j\vert$ by the distribution of the maximum of joint normal random variables with appropriate mean and covariance structure.
Notice that  $d$ (and accordingly, the number of elements in $\mathcal{H}$ and $\mathcal{I}$) can grow at a polynomial rate $O(T^{\beta_X})$ as the sample size $T$ increases to infinity. Furthermore,  \eqref{eq.minVARAUTO} and \eqref{eq.minVARCOEF} are introduced to ensure  that the estimation roots $\sqrt{T}(\widehat{\sigma}_j - \sigma_j)$ and $\sqrt{T}(\widehat{\rho}_j - \rho_j)$ do not degenerate asymptotically, that is,  the corresponding sequences do not converge to a constant for some  $j$ given.

Suppose we know the covariances of the estimation roots
$\sqrt{T}(\widehat{\sigma}_j - \sigma_j),\ j\in\mathcal{H}$ and $\sqrt{T}(\widehat{\rho}_j - \rho_j), j\in\mathcal{I}$. In that case, we can use Monte-Carlo, that is, we can generate joint normal pseudo  random variables  and calculate  the $1 - \alpha$ (e.g., $95\%$) quantile of $\max_{j\in\mathcal{H}}\vert\xi_j\vert$ and $\max_{j\in\mathcal{I}}\vert\zeta_j\vert$. Once these
quantiles are available,  simultaneous confidence intervals for $\sigma_j$ and $\rho_j$ can be constructed.  However, since $X_i, i\in\mathbf{Z}$ are allowed to be  non stationary,
estimating the covariances of the corresponding estimation roots is not straightforward. We postpone this topic to Section \ref{section.bootstrap}, where we will propose a bootstrap algorithm that automatically generates the desired confidence bands.

\section{Simultaneous inference for  autoregressive coefficients}
\label{section.AR_coef}
In this section we derive a Gaussian approximation result   for   the estimated autoregressive coefficients obtained by fitting a  $p$th order autoregression (AR(p)) to a (possibly nonlinear) weakly stationary time series. To elaborate, consider  the AR(p) model,
\begin{equation}
\begin{aligned}
X_i = \epsilon_i + \sum_{j = 1}^p a_jX_{i - j},\
\end{aligned}
\label{eq.def_AR}
\end{equation}
where  $\epsilon_i\sim WN(0,\sigma^2)$,
is fitted  by means of minimizing the mean square one-step ahead prediction error $ E(X_i-\sum_{j=1}^pa_jX_{i-j})^2$.  We call $(a_1,...,a_p)$ the auto-regressive coefficients (or AR coefficients for simplicity).

Multiplying by $X_{i - k}$, $ k = 1,2,...,p$, on both sides of eq.\eqref{eq.def_AR} and taking expectations, we arrive at  the well-known Yule-Walker equations (see Section 8.1, \cite{MR1093459}), i.e.,
\begin{equation}
\sigma_k = \sum_{j = 1}^p a_j \sigma_{\vert j - k\vert},\ \text{where $\sigma_k = \mathbf{E}X_0X_{-k}$}
\label{eq.Yule_walker}
\end{equation}
for $k=1,2, \ldots, p$. Replacing the autocovariances $\sigma_j$ by sample estimators and
 solving the corresponding  linear system, i.e.,
 \begin{equation}
     \widehat{\sigma}_k = \sum_{j = 1}^p \widehat{a}_j\widehat{\sigma}_{\vert j - k\vert},\ \text{where } \widehat{\sigma}_k \text{ satisfies equation  \eqref{eq.def_sigma}},
     \label{eq.def_SAMPLEAR}
 \end{equation}
leads to the  set of moment estimators of $a_{j}, j=1,2, \ldots, p$, known as Yule-Walker estimators.

Since it is more reasonable to assume that  the  observed time series does not necessarily stem from a linear  AR process, i.e., it satisfies (\ref{eq.def_AR}) with i.i.d. instead of white noise innovations,  it is important to derive the limiting distribution  of the  autoregressive estimators under such an  assumption.

Notice that different motivations may lead to fitting AR($p$) models to a time series at hand.
For instance,  \cite{MR2893863} used the fact that, under certain conditions,     a general strictly stationary process has  a general  AR($\infty$) representation, i.e., $X_i$ can be expressed as
  $X_i = \epsilon_i+ \sum_{j = 1}^\infty a_jX_{i - j}$. Here the  innovations $\epsilon_i$ are   white noises and not necessarily   i.i.d..
Furthermore,  according to \cite{MR3331856} and Section 2.7 in \cite{MR1093459}, the solution of the $p$th order system of Yule-Walker equation, determines  the
coefficients of the best linear predictor of $X_i$ in the Hilbert space spanned by $X_{i - 1},...,X_{i - p}$. More precisely, the best linear
predictor of $X_i$  is given by $\sum_{j = 1}^p \beta_j X_{i - j}$, which by  (2.7.15) in  \cite{MR1093459}, satisfies
\begin{equation}
\mathbf{E}X_iX_{i - k} = \sum_{j = 1}^p \beta_j \mathbf{E}X_{i - j}X_{i - k}\ \text{for } k  = 1,2,...,p
\label{eq.def_YULE_WALKER}
\end{equation}
and, therefore,  coincides with \eqref{eq.Yule_walker}.


\begin{definition}[AR coefficients of  order $p$ for a general time series]
Suppose  that the stochastic process  $\{X_i, i\in\mathbf{Z}\}$ is  weakly stationary (i.e., satisfies Definition 1.3.2 in \cite{MR1093459}), and
for a given positive integer $p$, suppose that the   solution $a_1,...,a_p$ of the  Yule-Walker system
\begin{equation}
\sigma_k = \sum_{j = 1}^p a_j \sigma_{\vert k - j\vert},\text{where }  k = 1,2,...,p.
\label{eq.def_NEWAR}
\end{equation}
exists. Then,  $a_1,...,a_p$  are called the AR coefficients of order  $p$.
\label{def.AR_coef}
\end{definition}
If $\{X_i\}$ possesses  a strictly  positive spectral density (see Corollary 4.3.2 and Proposition 4.5.3 in \cite{MR1093459}), then the
covariance matrix $\{\sigma_{\vert k - j\vert}\}, k,j = 1,...,p$ is
positive definite for any $p > 0$,  and, therefore, the AR coefficients are well-defined.

\begin{corollary}
Suppose $\{X_i, i\in\mathbf{Z}\}$ is weakly stationary, $\mathbf{E}X_i =  0$ and the matrix $\{\sigma_{\vert j - k\vert}\}_{j,k = 1,...,p}$ is
positive definite.
Then the coefficients
$a_1,...,a_p$ in \eqref{eq.def_AR} satisfy Definition \ref{def.AR_coef}.
\end{corollary}
\begin{proof}
The corollary follows  by comparing \eqref{eq.def_NEWAR} and \eqref{eq.Yule_walker}.
\end{proof}

After adopting Definition \ref{def.AR_coef} we  can apply the results of Section \ref{section.medium_range}
and establish a Gaussian approximation result  for the AR coefficients. We state this result as Theorem
\ref{theorem.Gaussian_ARCOEF}. To formulate this theorem, we need to introduce some additional notation. Define  the matrix
$\Sigma = \{\sigma_{\vert j - k\vert}\}_{j,k = 1,...,p}$ and assume that  it is non-singular for every $p\in\mathbf{N}$. Let  $\mathbf{e}_i = (\underbrace{0,0,...,0}_{i - 1}, 1, 0,...,0)^T\in\mathbf{R}^p$ be the vector with the one appearing in the $i$th position,
$\gamma = (\sigma_1, \sigma_2,...,\sigma_p)^T$ and $ T_i$ the $p\times p$ matrix,  $T_i = \{t_{\vert j - k \vert}^{(i)}\}_{j,k = 1,...,p},\ i = 0,1,...,p-1$ such   that $t_{s}^{(i)} = 1$ if $s = i$ and $0$ otherwise.
Define the matrix $B = \{b_{jk}\}_{j = 1,...,p, k = 0,1,...,p} = (\mathbf{b}_0,\mathbf{b}_1,...,\mathbf{b}_p)\in\mathbf{R}^{p\times (p + 1)}$ such that
$\mathbf{b}_0 = -\Sigma^{-2}\gamma$; $\mathbf{b}_i = \Sigma^{-1}\mathbf{e}_i - \Sigma^{-1}T_i\Sigma^{-1}\gamma$ for $i = 1,...,p-1$;
and $\mathbf{b}_p = \Sigma^{-1}\mathbf{e}_p$. Finally, for $j=1,2, \ldots, p$, let
\begin{equation}
Z_{i,j} = \sum_{k = 0}^p b_{jk}(X_iX_{i - k} - \sigma_{k}^{(i)}).
\label{eq.def_ARcoef}
\end{equation}
We can then establish the following result.

\begin{theorem}[Gaussian Approximation for AR Coefficients]
Suppose $\{X_i,i\in\mathbf{Z}\}$ are $(m,\alpha)$-short range dependent random variables with $m\geq 8$ and $\alpha > 2$.
Let  $p$ be  a positive integer such
that $p = O(1)$.  In addition suppose that $\{X_i\}$ is  weakly stationary  and there exists a constant $c > 0$ such that the smallest eigenvalue of $\Sigma$ is greater than $c$.
Define $Z_{i,j}$ as in \eqref{eq.def_ARcoef} and assume
\begin{equation}
\Vert\frac{1}{\sqrt{T}}\sum_{i = 1}^TZ_{i, j}\Vert_2 > C\ \text{for a constant $C>0$ and any $j = 1,...,p$.}
\label{eq.cond_AUTOCORR}
\end{equation}
\label{theorem.Gaussian_ARCOEF}
Then,
\begin{equation}
\sup_{x\in\mathbf{R}}\Big| Prob\left(\max_{j = 1,...,p}\vert\sqrt{T}(\widehat{a}_j - a_j)\vert\leq x\right) - Prob\left(\max_{j = 1,...,p}\vert\xi_j\vert\leq x\right)\Big|
= o(1),
\label{eq.auto_coeff}
\end{equation}
where  $a_1,...,a_p$ are the AR coefficients satisfying Definition \ref{def.AR_coef} and $\xi_1,...,\xi_p$ are joint normal random variables with  $\mathbf{E}\xi_j = 0$ and
\noindent $\mathbf{E}\xi_{j_1}\xi_{j_2} = T^{-1}\sum_{i_1 = 1}^T\sum_{i_2 = 1}^T \mathbf{E}Z_{i_1, j_1}Z_{i_2, j_2}$;
\end{theorem}

As in Section \ref{section.autocovariance}, the main problem in implementing the result  of Theorem \ref{theorem.Gaussian_ARCOEF} is that the covariances of
the  random  variables  $\sqrt{T}(\widehat{a}_j - a_j)$, $j=1,2, \ldots, p$,  are unknown and hard to estimate. The next section deals with this problem.

\section{Bootstrap based inference}
\label{section.bootstrap}
So far we have established  Gaussian approximation results  for  sample
autocovariances,  sample autocorrelations, and  estimates of the  AR coefficients. In this section,   we discuss how to implement these results for the construction of confidence bands or tests regarding the corresponding population parameters. As already mentioned, one important problem in this context is the estimation  of  the
covariances of the random variables  considered which is affected by  the fourth order moment structure of the underlying process. We  use the notation $Prob^*$ and $\mathbf{E}^*$, defined as
$Prob^*(\cdot) = Prob(\cdot |X_1,...,X_T)$ and
$\mathbf{E}^*\cdot = \mathbf{E}(\cdot|X_1,...,X_T)$, to represent probability and expectation in the bootstrap world.


We begin with Lemma \ref{lemma.consistent_variance}  which  derives  consistent estimators of the covariances of interest.

\begin{lemma}
Suppose $K(\cdot): \mathbf{R}\to [0,\infty)$ is a kernel function satisfying Definition \ref{def.kernel} and $k_T\to\infty$ as  $T\to\infty$.
 The random variables $\{X_i, i\in\mathbf{Z}\}$ are $(m,\alpha_X)-$short range dependent  with $m\geq 8,\ \alpha_X > 2$.

(i) Suppose that  $d = O(T^{\beta_X})$ is a positive integer, that the set $\mathcal{H}\subset\{0,1,...,d\}$ is not empty and that  \eqref{eq.condH} and
\eqref{eq.minVARAUTO} hold true. Then,
\begin{equation}
\begin{aligned}
\max_{j_1, j_2\in\mathcal{H}}\Big| \frac{1}{T}\sum_{i_1 = j_1 + 1}^T\sum_{i_2 = j_2 + 1}^TK\left(\frac{i_1 - i_2}{k_T}\right)(X_{i_1}X_{i_1 - j_1} - \widehat{\sigma}_{j_1})(X_{i_2}X_{i_2 - j_2}  - \widehat{\sigma}_{j_2})\\
-\frac{1}{T}\sum_{i_1 = 1}^T\sum_{i_2 = 1}^T\mathbf{E} (X_{i_1}X_{i_1 - j_1} - \sigma_{j_1}) (X_{i_2}X_{i_2 - j_2} - \sigma_{j_2})
\Big|\\
= O_p\left(v_T\times T^{\alpha_X\beta_X} + k_T\times T^{\frac{8\beta_X}{m} + \alpha_X\beta_X - \frac{1}{2}}\right)
\end{aligned}
\label{eq.sigmaX}
\end{equation}

(ii) Suppose that $d = O(T^{\beta_X})$ is a positive integer,  the set $\mathcal{I}\subset\{1,...,d\}$ is not empty and that  \eqref{eq.condH} and
\eqref{eq.minVARCOEF} hold true. Recall the definition of $Z_{i,j}$ in Theorem~\ref{th.th2}(ii) and let
\begin{equation}
\widehat{Z}_{i,j} = - \frac{\widehat{\sigma}_j}{\widehat{\sigma}^2_0}(X_i^2 - \widehat{\sigma}_0) + \frac{1}{\widehat{\sigma}_0}(X_iX_{i - j} - \widehat{\sigma}_j).
\label{eq.def_autocorr}
\end{equation}
Then
\begin{equation}
\begin{aligned}
\max_{j_1, j_2\in\mathcal{I}}\Big| \frac{1}{T}\sum_{i_1 = j_1 + 1}^T\sum_{i_2 = j_2 + 1}^TK\left(\frac{i_1 - i_2}{k_T}\right)\widehat{Z}_{i_1,j_1}\widehat{Z}_{i_2, j_2}
 - \frac{1}{T}\sum_{i_1 = 1}^T\sum_{i_2 = 1}^T\mathbf{E} Z_{i_1,j_1}Z_{i_2,j_2}
\Big|\\
= O_p\left(v_T\times T^{\alpha_X\beta_X} + k_T\times T^{\frac{8\beta_X}{m} + \alpha_X\beta_X - \frac{1}{2}}\right).
\end{aligned}
\label{eq.sigmaCorr}
\end{equation}

(iii) Suppose that the conditions of Theorem~\ref{theorem.Gaussian_ARCOEF} are satisfied. Define
\begin{equation}
\widehat{Z}_{i,j} = \sum_{k = 0}^p \widehat{b}_{jk}(X_iX_{i - k} - \widehat{\sigma}_{k})
\label{eq.defZHAT}
\end{equation}
and $Z_{i,j}$ as in \eqref{eq.def_ARcoef}. Here $\widehat{B} = \{\widehat{b}_{jk}\}_{j = 1,...,p, k = 0,1,...,p} = (\widehat{\mathbf{b}}_0,...,\widehat{\mathbf{b}}_p)\in\mathbf{R}^{p\times (p + 1)}$;
$\widehat{\mathbf{b}}_0 = -\widehat{\Sigma}^{\dagger 2}\widehat{\gamma}$; $\widehat{\mathbf{b}}_i = \widehat{\Sigma}^{\dagger }\mathbf{e}_i - \widehat{\Sigma}^{\dagger }T_i\widehat{\Sigma}^{\dagger }\widehat{\gamma}$
for $i = 1,...,p-1$; and $\widehat{\mathbf{b}}_p = \widehat{\Sigma}^{\dagger }\mathbf{e}_p$. We refer to eq.\eqref{eq.def_sigma} and to Section \ref{section.AR_coef}
for the notation used. Then,
\begin{equation}
\begin{aligned}
\sup_{j_1,j_2 = 1,...,p}\Big| \frac{1}{T}\sum_{i_1 = j_1 + 1}^T\sum_{i_2 = j_2 + 1}^TK\left(\frac{i_1 - i_2}{k_T}\right)\widehat{Z}_{i_1,j_1}\widehat{Z}_{i_2, j_2}
- \frac{1}{T}\sum_{i_1 = 1}^T\sum_{i_2 = 1}^T \mathbf{E}Z_{i_1,j_1}Z_{i_2,j_2}\Big|\\
= O_p(v_T + k_T\times T^{ - \frac{1}{2}}),
\end{aligned}
\label{eq.third}
\end{equation}
\label{lemma.consistent_variance}
where $v_T$ coincides  with {Lemma \ref{lemma.covariance} and  $\alpha = \alpha_X$.}
\end{lemma}

Calculating the covariances stated above is tedious. Moreover, there is no explicit formula or table for the  quantiles for the maximum of joint
Gaussian random variables, so some simulations are required to derive the quantiles. Therefore, it is not straightforward to perform statistical inference based on  classical methods,  see for instance,  Section 4.4 in \cite{foundation_statistics}.
An alternative approach  is through using bootstrapping techniques.  One advantage of adopting such a technique  is that  complex calculations can be avoided.  In the following, we present  a
bootstrap procedure  which achieves the desired goal.

Our focus is  on  deriving  simultaneous confidence intervals or on performing hypothesis tests for  autocovariances,  autocorrelations and AR coefficients.  More specifically, we are focus on testing the hypothesis
\begin{equation}
\text{H$_0$:  } \sigma_k = \sigma_k^{(e)}\ \text{for all } k\in\mathcal{H}\ \text{versus H$_1$: } \exists k_1\in\mathcal{H} \text{ such that } \sigma_{k_1}\neq \sigma_{k_1}^{(e)}
\label{eq.test_sigma}
\end{equation}
for autocovariances, or the hypotheses
\begin{equation}
\text{H$_0$: } \rho_k = \rho_k^{(e)}\ \text{for all } k\in\mathcal{I}\ \text{versus the H$_1$: } \exists k_1\in\mathcal{I}\ \text{such that } \rho_{k_1}\neq \rho_{k_1}^{(e)}
\label{eq.test_rho}
\end{equation}
for  autocorrelations, or the hypotheses
\begin{equation}
\begin{aligned}
\text{H$_0$: } a_k = a_k^{(e)}\ \text{for all } k = 1,...,d\ \text{versus H$_1$: } \exists k_1 \in\{1,2,...,d\}\\
\text{such that } a_{k_1}\neq a_{k_1}^{(e)}
\label{eq.test_a}
\end{aligned}
\end{equation}
for  AR coefficients. Here $\sigma_k^{(e)},\ \rho_k^{(e)}$ and $a_k^{(e)}$  are given hypothesized values. An interesting  special case is that when  $\rho_{k}^{(e)} = \sigma_k^{(e)} = 0$, which is implemented in the R function `\textrm{acf}'.
The   bootstrap procedure used in this section is called `the second-order wild bootstrap' and   originates from the dependent wild bootstrap
introduced by \cite{MR2656050}, also see \cite{https://doi.org/10.48550/arxiv.2110.13498}. The procedure is considered in \cite{Bootstrap_time_series} and it coincides with the version of the  dependent wild bootstrap introduced  in \cite{MR3796524} when the order of the autoregression fitted in the aforecited paper is set equal to zero. It directly resamples the products $X_iX_{i-j}$  of  the  time series, which makes it suitable to analyze second-order statistics like  sample autocovariances and  autocorrelations.

\begin{algorithm}[Second-Order Wild Bootstrap]
$\\$
\textbf{Input: } Observations $X_1,...,X_T$, the maximum lag $d$, the lag for calculating AR coefficients $p$ ($p\leq d$), the kernel function $K(\cdot):\mathbf{R}\to [0,\infty)$, a bandwidth $k_T$,
the nominal coverage probability $1 - \alpha$ and the number of bootstrap replicates $B$. We also need the index set $\mathcal{H}\subset\{0,1,...,d\}$ for
autocovariances and the index set $\mathcal{I}\subset\{1,2,...,d\}$ for autocorrelations.

\noindent\textbf{Additional input for hypothesis testing: } The autocovariances $\sigma^{(e)}_k, k\in\mathcal{H}$, the  autocorrelations $\rho^{(e)}_k, k\in\mathcal{I}$ and the expected AR coefficients $a^{(e)}_k, k = 1,2,...,p$.

1. Calculate the  sample autocovariances $\widehat{\sigma}_j, j = 0,1,...,d$ and the sample autocorrelations $\widehat{\rho}_j, j = 1,2,...,d$ as in \eqref{eq.def_sigma}. Then calculate the sample AR coefficients $\widehat{a}_j$ as in \eqref{eq.def_SAMPLEAR}. Define the `second-order residuals' $\widehat{\epsilon}_i^{(j)} = X_iX_{i - j} - \widehat{\sigma}_j$ for $j = 0,1,...,d$ and $i = j+1, j+2,...,T$.

2. Generate joint normal random variables $\varepsilon_1,...,\varepsilon_T$ such that $\mathbf{E}\varepsilon_j = 0$ and
$\mathbf{E}\varepsilon_{j_1}\varepsilon_{j_2} = K\left((j_1 - j_2)/k_T\right)$. Then calculate
\begin{equation}
\widehat{\sigma}^*_j =\widehat{\sigma}_j +  \frac{1}{T}\sum_{i = j+1}^T \widehat{\epsilon}_i^{(j)}\times \varepsilon_i\
\ \text{for } j = 0,1,...,d
\label{eq.bootSigma}
\end{equation}

3. Define $\widehat{\Sigma}^*$ and $\widehat{\gamma}^*$ as
\begin{equation}
\widehat{\Sigma}^* = \{\widehat{\sigma}^*_{\vert j - k\vert}\}_{j,k = 1,...,p}\ \text{and } \widehat{\gamma}^* = (\widehat{\sigma}^*_1,...,\widehat{\sigma}^*_p)^T
\label{eq.defBootCov}
\end{equation}
and calculate
\begin{equation}
\widehat{\rho}^*_j = \widehat{\sigma}^*_j/\widehat{\sigma}^*_0\ \text{for } j \in\mathcal{I},\ \widehat{a}^* = (\widehat{a}^*_1,...,\widehat{a}^*_{p})^T = \widehat{\Sigma}^{*\dagger} \widehat{\gamma}^*
\end{equation}
recall that $\dagger$ represents the Moore-Penrose pseudo inverse.  Then define
\begin{equation}
\begin{aligned}
\delta^*_{b,\sigma} = \sqrt{T}\max_{j\in\mathcal{H}}\vert \widehat{\sigma}^*_j - \widehat{\sigma}_j\vert,\ \delta^*_{b, \rho} = \sqrt{T}\max_{j\in\mathcal{I}}\vert\widehat{\rho}^*_j  - \widehat{\rho}_j\vert\\
\text{and } \delta^*_{b, a} = \sqrt{T}\max_{j = 1,...,p}\vert\widehat{a}^*_j - \widehat{a}_j\vert
\end{aligned}
\end{equation}

4. Repeat step 2 and  3 for $b  =  1,2,...,B$ times. Then calculate the $1 - \alpha $ sample quantiles $C^*_{1-\alpha,\sigma}$ of  the sequence $\{\delta^*_{b,\sigma}\}_{b = 1,...,B}$, $C^*_{1-\alpha, \rho} $ of the sequence $\{\delta^*_{b,\rho}\}_{b  = 1,...,B}$, and $C^*_{1-\alpha, a}$ of the sequence $\{\delta^*_{b, a}\}_{b = 1,...,B}$. Recall the definition of a sample quantile:
For a numerical sequence $a_b, b = 1,...,B$ such that $ a_1\leq a_2\leq...\leq a_B$,
\begin{equation}
C^*_{1-\alpha} = a_{b^*},\ \text{where  } b^* = \min\left\{b = 1,...,B: \  \frac{1}{B}\sum_{c = 1}^B \mathbf{1}_{\{a_c\leq a_b\}}\geq 1-\alpha\right\}.
\end{equation}

5a. \ For constructing confidence intervals: The $1-\alpha$ simultaneous confidence intervals for $\sigma_j, j\in\mathcal{H}$, $\rho_j, j\in\mathcal{I}$ and $a_j,
j = 1,2,...,p$ will be
\begin{equation}
\begin{aligned}
\sqrt{T}\max_{j\in\mathcal{H}}\vert \widehat{\sigma}_j - \sigma_j\vert\leq C^*_{1-\alpha, \sigma},\ \sqrt{T}\max_{j\in\mathcal{I}}\vert
\widehat{\rho}_j - \rho_j\vert\leq C^*_{1-\alpha,\rho}\\
\text{and }  \sqrt{T}\max_{j = 1,...,p}\vert \widehat{a}_j - a_j\vert\leq C^*_{1-\alpha, a}
\end{aligned}
\end{equation}

5b. \ For hypothesis testing: Reject the null hypothesis in eq.\eqref{eq.test_sigma} if $\sqrt{T}\max_{j\in\mathcal{H}}\vert \widehat{\sigma}_j - \sigma^{(e)}_j\vert> C^*_{1-\alpha, \sigma}$; reject the null hypothesis in eq.\eqref{eq.test_rho} if $\sqrt{T}\max_{j\in\mathcal{I}}\vert
\widehat{\rho}_j - \rho_j^{(e)}\vert >C^*_{1-\alpha,\rho}$; and reject the null hypothesis in eq.\eqref{eq.test_a} if $\sqrt{T}\max_{j = 1,...,p}\vert \widehat{a}_j - a_j^{(e)}\vert> C^*_{1-\alpha, a}$.
\label{algorithm1}
\end{algorithm}


For
$B\to\infty$ and from Glivenko - Cantelli's  Theorem and Theorem 1.2.1 in \cite{MR1707286}, the empirical cumulative distribution function of $\delta^*_{b,\tau}, \tau \in\{\sigma, \rho, a\}$ will
converge to $Prob^*\left(\delta^*_{b,\tau}\leq x\right)$ almost surely in the bootstrap world (i.e., the conditional probability space conditioning on $X_1,...,X_T$). The bootstrap confidence interval
is then consistent as long as
\begin{equation}
\sup_{x\in\mathbf{R}}\vert Prob^*(\delta^*_{b, \tau}\leq x) - H_\tau(x)\vert = o_p(1),
\end{equation}
where $\tau \in\{\sigma,\rho, a\}$; and $H_\sigma(x)$, $H_\rho(x)$, $H_a(x)$, respectively,  represent $Prob\left(\max_{j\in\mathcal{H}}\vert\xi_j\vert\leq x\right)$ in
\eqref{eq.auto_covariance_gaussian}, $Prob\left(\max_{j\in\mathcal{I}}\vert\zeta_j\vert\leq x\right)$ in \eqref{eq.rho_central} and
$ Prob\left(\max_{j = 1,...,p}\vert\xi_j\vert\leq x\right)$ in \eqref{eq.auto_coeff}. The following theorem justifies
validity of the proposed bootstrap algorithm.

\begin{theorem}
Suppose the kernel function $K(\cdot)$ and
the observations $X_1,...,X_T$ satisfy the conditions of  Lemma \ref{lemma.consistent_variance} and define $v_T$ as in Lemma \ref{lemma.covariance} {with $\alpha = \alpha_X$}.

(i) In addition suppose $d = O(T^{\beta_X})$ is a positive integer, the bandwidth $k_T$ is such that
 $v_T\times T^{7\alpha_X\beta_X} = o(1)$ and $k_T\times T^{\frac{8\beta_X}{m} + 7\alpha_X\beta_X - \frac{1}{2}} = o(1)$, and the set $\mathcal{H}\subset\{0,1,...,d\}$ is not empty. Suppose \eqref{eq.condH} and
\eqref{eq.minVARAUTO} hold true. Then
\begin{equation}
\begin{aligned}
\sup_{x\in\mathbf{R}}\Big| Prob^*(\sqrt{T}\max_{j\in\mathcal{H}}\vert \widehat{\sigma}^*_j - \widehat{\sigma}_j\vert\leq x) - H_\sigma(x)\Big| \\
= O_p\left(\left(v_T\times T^{7\alpha_X\beta_X}\right)^{1/6} + \left(k_T\times T^{\frac{8\beta_X}{m} + 7\alpha_X\beta_X - \frac{1}{2}}\right)^{1/6}\right)
\end{aligned}
\label{eq.Boot_delta_sigma}
\end{equation}

(ii) In addition suppose $d = O(T^{\beta_X})$ is a positive integer, the bandwidth $k_T$ satisfies $v_T\times T^{7\alpha_X\beta_X} = o(1)$ and $k_T\times T^{\frac{8\beta_X}{m} + 7\alpha_X\beta_X - \frac{1}{2}} = o(1)$  and the set $\mathcal{I}\subset\{1,...,d\}$ is not empty.
Suppose \eqref{eq.condH} and \eqref{eq.minVARCOEF} hold true, then
\begin{equation}
\sup_{x\in\mathbf{R}}\Big| Prob^*(\sqrt{T}\max_{j\in\mathcal{I}} \vert\widehat{\rho}^*_j - \widehat{\rho}_j\vert\leq x) - H_\rho(x)\Big| = o_p(1).
\label{eq.Boot_delta_rho}
\end{equation}

(iii) In addition suppose $p = O(1)$ is a positive number and the bandwidth $k_T$ satisfies $v_T = o(1)$ and $k_T \times T^{-1/2}  = o(1)$. Suppose
\eqref{eq.cond_AUTOCORR} hold true and assume that the smallest eigenvalue of the covariance matrix $\Sigma$ (see Theorem \ref{theorem.Gaussian_ARCOEF}) is greater than a constant $c > 0$. Then,
\begin{equation}
\sup_{x\in\mathbf{R}}\Big| Prob^*(\sqrt{T}\max_{j = 1,...,p}\vert \widehat{a}^*_j - \widehat{a}_j\vert\leq x) - H_a(x)\Big| = o_p(1).
\label{eq.Boot_delta_a}
\end{equation}
\label{thm.bootstrap}
\end{theorem}
Similar to Theorem \ref{theorem.auto_covariance}, if $\beta_X\neq 0$, the conditions in Theorem \ref{thm.bootstrap} imply  that the maximum lag $d$ cannot be too large compared to the sample size $T$.

However, since  linear approximations are used to derive the asymptotic distribution of  sample autocorrelations and  sample AR coefficients; see  \eqref{eq.change_autocorrelation},  and \eqref{eq.prob_SEC2},  and \eqref{eq.aroundA} in the supplementary material, the
exact convergence rates of eq.\eqref{eq.Boot_delta_rho} and eq.\eqref{eq.Boot_delta_a} are difficult  to derive.

\section{Numerical Results}
\label{section.numerical_experiment}
This section performs numerical experiments to demonstrate the finite sample performance of the proposed  bootstrap (Algorithm \ref{algorithm1}). We first
investigate the performance of  the bootstrap using simulated data. Then we  use a real-life data set to demonstrate  the differences between inference based on Algorithm \ref{algorithm1} and on the
AR-sieve  bootstrap. Recall  that the AR-sieve  bootstrap leads to consistent confidence intervals only in specific situations;
we refer to    \cite{MR2893863} for details.

\textbf{Selecting the bandwidth $k_T$: } Implementation of   Bootstrap Algorithm 1 requires the selection of the bandwidth parameter $k_T$.  \cite{MR2041534} introduced
an automatic bandwidth selection algorithm which also has been used in  \cite{MR2656050} for  selecting the bandwidth of  the dependent wild bootstrap. We refer to the
R package `np'(see \cite{JSSv027i05}) for the implementation of this algorithm. Other methods include \cite{MR3511584} (also see \cite{MR3796524}) and \cite{MR2748557}. In this section we will use the  procedure proposed by  \cite{MR2041534}.

\textbf{Model selection: } We select the order $p$ of the  AR models based on AIC (as incorporated in the `ar' function in R).

\textbf{Simulated data: } Let  $\{e_i\}_{i\in\mathbf{Z}}$ be  i.i.d. normal random variables with $\mathbf{E}e_i = 0$ and $\mathbf{E}e^2_i = 1$.
We consider three kinds of white noise innovations, i.e., \textit{independent: } $\epsilon_i = e_i$;
\textit{product  of normals: } $\epsilon_i = e_ie_{i -  1}$; \textit{non-stationary: } $\epsilon_i = e_i$ for $i = 2k$ and
$e_{i}e_{i - 1}$ for $i =  2k - 1$, here $k\in\mathbf{Z}$. Observe that all innovations satisfy $\mathbf{E}\epsilon_i = 0,\ \mathbf{E}\epsilon^2_i = 1$ and
$\mathbf{E}\epsilon_i\epsilon_j = 0$ for $i\neq j$. However, the \textit{product normal} innovations are dependent and the last defined
innovations are not stationary. We consider the following time series models:

\textit{AR(1):  } $X_t = 0.9X_{t - 1} + \epsilon_t$

\textit{AR(2): } $X_t = 0.5X_{t - 1} + 0.2X_{t - 2} + \epsilon_t$

\textit{AR(4): } $X_t = 0.3X_{t - 1} + 0.2X_{t - 2} + 0.2X_{t - 3} + 0.1X_{t - 4} + \epsilon_t$

\textit{MA(3): } $X_t = \epsilon_t + 0.6\epsilon_{t - 1} + 0.4\epsilon_{t - 2} + 0.1 \epsilon_{t - 3}$

\textit{Nonlinear AR(2): } $X_t = \sin(X_{t-1}) + \cos(X_{t - 2})+\epsilon_{t}$

The results obtained  are shown  in Figure~\ref{figure_1}, Figure~\ref{fig.examp2} and Table~\ref{table.estimator}. To save computing time, we verify the performance of the bootstrap algorithms based on the warp-speed method, see \cite{MR3064050}. If the innovations in the time series are indeed i.i.d., the (point-wise or simultaneous) confidence intervals obtained  by  Algorithm ~\ref{algorithm1} and the AR-sieve bootstrap, are very similar. However, when the innovations are dependent, bootstrap Algorithm \ref{algorithm1} tends to generate wider confidence intervals. Table~\ref{table.estimator} records the coverage probability of the second-order wild bootstrap algorithm. We compare the proposed method to the AR-sieve  bootstrap. The second-order wild bootstrap algorithm has desired coverage probability for the autocovariances, the autocorrelations, and the AR coefficients, even if the time series at hand has non-i.i.d. innovations or does not stem from  an autoregressive process. On the contrary,  the AR-sieve bootstrap performs well for AR coefficients when the innovations are independent. However, if the innovations are dependent, then the AR-sieve bootstrap may fail to achieve  the correct coverage probability.

\begin{figure}[htbp]
    \subfigure[AR(4) with \textit{independent} innovations]{
    \includegraphics[width = 1.7in]{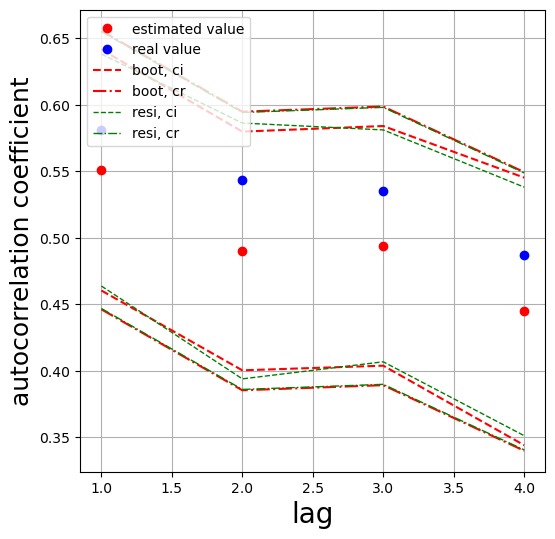}
  }
    \subfigure[AR(4) with \textit{product normal} innovations]{
    \includegraphics[width = 1.7in]{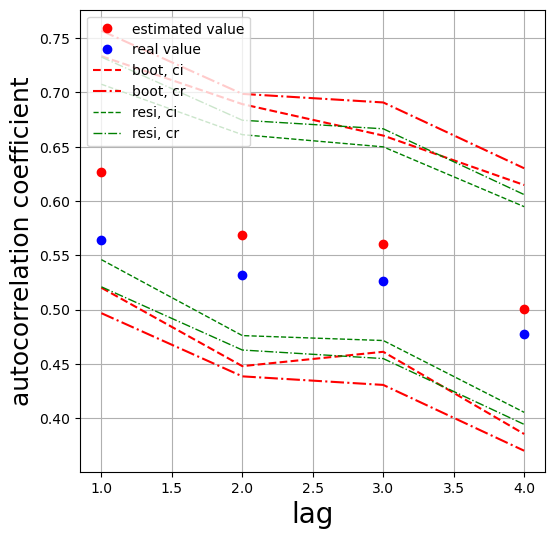}
  }
    \subfigure[AR(4) with \textit{non-stationary} innovations]{
    \includegraphics[width = 1.7in]{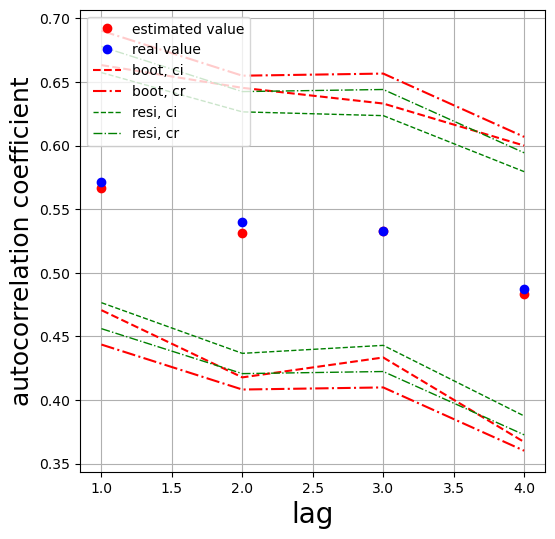}
  }
  \caption{Point-wise and simultaneous confidence interval for the autocorrelation coefficients(ACF) of the AR(4) model with different types of innovations. The blue and  red dots represent the real  and the estimated values, respectively. In the legend, `boot' represent Bootstrap Algorithm~\ref{algorithm1} and `resi' represent
  the AR-sieve bootstrap.  Furthermore, `ci' means the point-wise confidence intervals and `cr'  simultaneous confidence intervals.
  }
    \label{figure_1}
\end{figure}

\begin{figure}[htbp]
    \subfigure[AR(4) with \textit{independent} innovations]{
    \includegraphics[width = 1.7in]{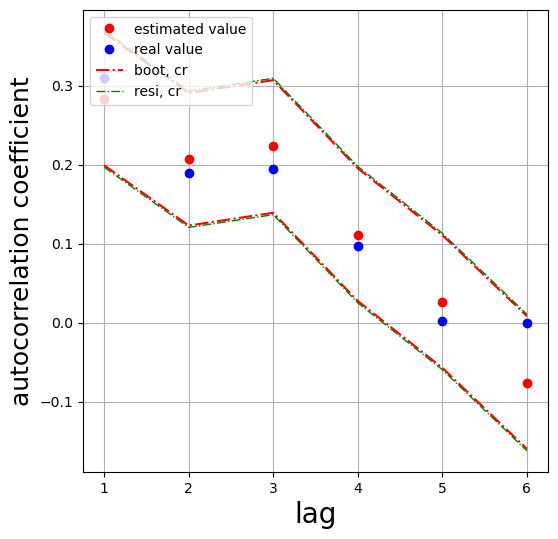}
  }
    \subfigure[AR(4) with \textit{product normal} innovations]{
    \includegraphics[width = 1.7in]{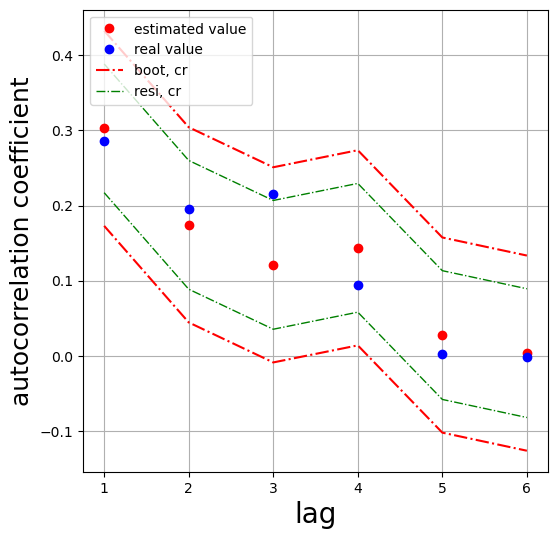}
  }
    \subfigure[AR(4) with \textit{non-stationary} innovations]{
    \includegraphics[width = 1.7in]{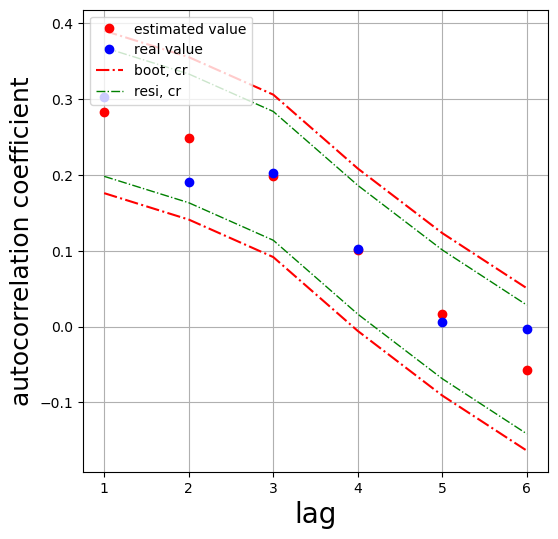}
  }
  \caption{Simultaneous confidence intervals for the AR coefficients of the AR(4) model with different types of innovations. The blue  and  red dots represent the real values and the estimated values, respectively while the   legend coincides with that of Figure~\ref{figure_1}. We select the max lag to be equal to $6$.
  }
    \label{fig.examp2}
\end{figure}

\begin{table}[htbp]
\centering
\scriptsize
\caption{Performance of the Bootstrap Algorithm~\ref{algorithm1} and of the AR bootstrap. The order of the model (i.e., $p$ in the Bootstrap Algorithm~\ref{algorithm1}) is selected by the
AIC criterion (as implemented in the `ar' function in R). We choose $\mathcal{H} = \{0,1,2,3\},\ \mathcal{I} = \{1,2,3,4\}$ and $d = 7$.  The sample size is $T = 1000$ and the nominal coverage probability is  set equal to $95\%$.
}
\label{table.estimator}
\begin{tabular}{l l l l l l l}
\hline\hline
Model & innovation                    & order   &   $k_T$   &  statistics & boot coverage & AR coverage\\
AR(1) & \textit{independent}    & 1         & $32.9$    &  autocovariance & $92.8\%$ & $100\%$\\
         &                                      &           &                & autocorrelation  & $90.3\%$ & $93.3\%$\\
         &                                      &           &                & AR coefficients   & $93.8\%$ & $94.3\%$\\
         & \textit{product of normals} & 1        & 30.9        & autocovariance  & $95.7\%$ & $100\%$\\
         &                                      &           &                 & autocorrelation  & $95.2\%$ & $95.6\%$\\
         &                                      &           &                 & AR coefficients   & $95.0\%$ & $76.7\%$\\
         & \textit{Non-stationary}  &  1       & 35.7         & autocovariance &  $95.4\%$ & $93.3\%$\\
         &                                       &          &                 & autocorrelation  & $94.9\%$ & $97.2\%$\\
         &                                       &         &                  & AR coefficients   & $92.7\%$ & $91.3\%$\\
AR(2) & \textit{independent}    & 2         &   $34.8$             & autocovariance & $93.9\%$ & $100\%$\\
          &                                    &            &                & autocorrelation & $93.7\%$ & $94.6\%$\\
          &                                    &            &                & AR coefficients  & $96.2\%$ & $95.2\%$\\
          & \textit{product of normals} & 2       &  $22.7$   & autocovariance & $93.8\%$ & $99.2\%$\\
          &                                      &          &                & autocoefficients & $95.4\%$ & $96.3\%$\\
          &                                      &          &                & AR coefficients   & $93.2\%$ & $76.4\%$\\
          & \textit{Non-stationary} & 2        & $17.4$    & autocovariance  & $94.9\%$ & $99.4\%$\\
          &                                      &           &                & autocoefficients & $97.5\%$ & $97.6\%$\\
          &                                     &            &               &  AR coefficients   & $94.1\%$ & $82.5\%$\\
AR(4) & \textit{independent}     & 4         & $35.0$  & autocovariance    & $93.1\%$ & $98.7\%$\\
          &                                     &            &               & autocorrelation   & $93.8\%$ & $95.3\%$\\
          &                                     &            &               & AR coefficients    & $95.6\%$ & $95.0\%$\\
          & \textit{product of normals} &  6       & $43.3$   & autocovariance   & $95.9\%$ & $100\%$\\
          &                                      &           &               & autocorrelation   & $96.0\%$ & $88.9\%$\\
          &                                      &           &               & AR coefficients    & $95.2\%$ & $78.2\%$\\
          &  \textit{Non-stationary} & 4        & $40.1$  & autocovariance   & $95.0\%$ & $97.2\%$\\
          &                                      &             &            & autocorrelation    & $95.4\%$ & $92.0\%$\\
          &                                      &             &             & AR coefficients    & $93.8\%$ & $89.8\%$\\
MA(3) & \textit{independent}      & 3          & 10.8    & autocovariance   & $94.8\%$ & $100\%$\\
           &                                      &             &            & autocorrelation   & $93.4\%$ & $97.7\%$\\
           &                                      &             &             & AR coefficients   & $94.0\%$ & $94.3\%$ \\
           & \textit{product of normals } & 3         & 8.7       & autocovariance &  $95.0\%$  & $100\%$\\
           &                                       &            &             & autocorrelation & $94.9\%$  & $97.2\%$\\
           &                                       &            &              & AR coefficients  & $94.9\%$ & $83.8\%$\\
           &  \textit{Non-stationary} & 5        & 8.7        & autocovariance & $93.7\%$ & $100\%$\\
           &                                        &          &               & autocorrelations & $92.7\%$ & $93.7\%$\\
           &                                        &          &                & AR coefficients   & $95.0\%$ & $88.6\%$\\
nonlinear & \textit{independent}    & 2        & 8.0         & autocovariance  & $93.4\%$ & $100\%$\\
              &                                     &          &                & autocorrelation  & $94.3\%$ & $100\%$\\
              &                                     &         &                 & AR coefficients   & $95.7\%$ & $96.9\%$\\
              & \textit{product of normals} & 2     & 4.1           & autocovariance  & $94.9\%$ & $100\%$\\
              &                                       &       &                 & autocorrelation  & $93.1\%$ & $100\%$\\
              &                                       &      &                  & AR coefficients   & $96.0\%$ & $78.9\%$\\
              & \textit{Non-stationary}  & 2   &  5.4           & autocovariance  &  $96.0\%$ & $100\%$\\
              &                                      &       &                  & autocorrelation  & $93.5\%$ &  $ 100\%$\\
              &                                      &       &                  & AR coefficients   & $92.8\%$  & $86.9\%$\\
\hline\hline
\end{tabular}
\end{table}

\section{Conclusions}
\label{section.conclusion}
Focusing on a weakly (but not necessarily strictly) stationary time series, this paper derives Gaussian approximation results for  sample autocovariances,
 sample autocorrelations and sample AR coefficients, the latter when an autoregressive process is fitted to a general statonary time series.
A consistent
bootstrap algorithm, called `the second-order wild bootstrap', is used to perform   statistical inference on  the corresponding population parameters.
While  strictly stationary assumptions form  a necessary set of assumptions in the time series literature, our work  weakness considerable these assumptions  in order to perform the  type of statistical inference considered in this paper.
\clearpage
\section*{Acknowledgement}
The authors appreciate Dr. Kejin Wu for the valuable suggestions and comments. The research is partially supported by the Chinese University of Hong Kong (Shenzhen) University development fund -research start-up fund, UDF01002774.
\clearpage
\printbibliography 
\appendix
\numberwithin{equation}{section}
\numberwithin{lemma}{section}

\section{Appendix: Preliminary Results}
This section introduces some special functions that are helpful in the following proofs, also see \cite{ZHANG}.  For any $\tau, \psi > 0$, $z\in\mathbf{R}$, define $F_\tau(x_1,...,x_s) = \frac{1}{\tau}\log(\sum_{i = 1}^s \exp(\tau x_s))$;
\begin{equation}
G_\tau(x_1,...,x_s) =\frac{1}{\tau}\log\left(\sum_{i = 1}^s \exp(\tau x_i) + \sum_{i = 1}^s \exp(-\tau x_i)\right)\\
= F_\tau(x_1,...,x_s, -x_1,...,-x_s)
\end{equation}
Define $g_0(x) = (1 - \min(1, \max(x, 0))^4)^4$ and $g_{\psi, z}(x)  = g_0(\psi(x - z))$. Then define $h_{\tau, \psi, z}(x_1,...,x_n) = g_{\psi, z}(G_\tau(x_1,...,x_n))$. From lemma A.2 and (8) in \cite{MR3161448} and (S1) to (S5) in \cite{MR3992401}, $g_* = \sup_{x\in\mathbf{R}}(\vert g^{\prime}_0(x)\vert + \vert g^{\prime\prime}_0(x)\vert + \vert g^{\prime\prime\prime}_0(x)\vert) < \infty$;  $\mathbf{1}_{x\leq z}\leq g_{\psi, z}(x)\leq \mathbf{1}_{x\leq z + 1/\psi}$; $\sup_{x,z\in\mathbf{R}}\vert g^{\prime}_{\psi, z}(x)\vert\leq g_*\psi$, $\sup_{x,z\in\mathbf{R}}\vert g^{\prime\prime}_{\psi,z}(x)\vert\leq g_*\psi^2$ and $\sup_{x,z\in\mathbf{R}}\vert g^{\prime\prime\prime}_{\psi, z}(x)\vert\leq g_*\psi^3$. Define the operator $\partial_i f = \frac{\partial f}{\partial x_i}$. Then $\partial_i F_\tau\geq 0$;
$\sum_{i = 1}^s \partial_i F_\tau = 1$; $\sum_{i = 1}^s\sum_{j = 1}^s\vert\partial_i\partial_j F_\tau\vert\leq 2 \tau$;

\noindent $\sum_{i = 1}^s\sum_{j = 1}^s\sum_{k = 1}^s\vert\partial_i\partial_j\partial_k F_\tau\vert\leq 6\tau^2$. Moreover,
\begin{equation}
\begin{aligned}
F_\tau(x_1,...,x_s) - \frac{\log(s)}{\tau}\leq \max_{i = 1,...,s}x_i\leq F_\tau(x_1,...,x_s)\\
\Rightarrow G_\tau(x_1,...,x_s) - \frac{\log(2s)}{\tau}\leq \max_{i = 1,...,s}\vert x_i\vert\leq G_\tau(x_1,...,x_s)
\end{aligned}
\end{equation}
Since $\partial_i G_\tau = \partial_i F_\tau - \partial_{s + i}F_\tau$, we get $\sum_{i = 1}^s \vert\partial_i G_\tau\vert\leq 1$. For $\partial_i\partial_j G_\tau = \partial_i \partial_j F_\tau - \partial_i\partial_{j + s}F_\tau - \partial_{i + s}\partial_j F_\tau + \partial_{i+s}\partial_{j + s}F_\tau$, we have $\sum_{i = 1}^s\sum_{j = 1}^s\vert\partial_i\partial_j G_\tau\vert\leq 2\tau$. Since $\partial_i\partial_j\partial_k G_\tau = \partial_i\partial_j\partial_k F_\tau - \partial_{i + s}\partial_j\partial_k F_\tau - \partial_i\partial_{j + s}\partial_k F_\tau - \partial_i\partial_j\partial_{k + s} F_\tau + \partial_{i + s}\partial_{j + s}\partial_k F_\tau + \partial_i\partial_{j + s}\partial_{k + s} F_\tau + \partial_{i + s}\partial_j\partial_{k + s} F_\tau - \partial_{i + s}\partial_{j + s}\partial_{k + s} F_\tau$, $\sum_{i = 1}^s\sum_{j = 1}^s\sum_{k = 1}^s\vert\partial_i\partial_j\partial_k G_\tau\vert\leq 6\tau^2$. For $\partial_i h_{\tau, \psi, z} = g^{\prime}_{\psi, z}(G_\tau(x_1,...,x_s))\times \partial_i G_\tau$, we get $\sum_{i = 1}^s \vert\partial_i h_{\tau, \psi, z} \vert\leq g_*\psi$. Moreover,
\begin{equation}
\begin{aligned}
\partial_i\partial_j h_{\tau,\psi,z} = g^{\prime\prime}_{\psi, z}(G_\tau(x_1,...,x_s))\times \partial_i G_\tau\partial_j G_\tau + g^{\prime}_{\psi, z}(G_\tau(x_1,...,x_s))\times \partial_i \partial_j G_\tau\\
\Rightarrow \sum_{i = 1}^s\sum_{j = 1}^s\vert \partial_i\partial_j h_{\tau,\psi,z}\vert\leq g_*\psi^2 + 2g_*\psi\tau\\
\partial_i\partial_j\partial_k h_{\tau, \psi, z} =  g^{\prime\prime\prime}_{\psi, z}(G_\tau(x_1,...,x_s))\times \partial_i G_\tau\partial_j G_\tau\partial_k G_\tau\\
+ g^{\prime\prime}_{\psi, z}(G_\tau(x_1,...,x_s))\times\left(\partial_i\partial_j G_\tau\times \partial_k G_\tau + \partial_i\partial_k G_\tau \times\partial_j G_\tau + \partial_j\partial_k G_\tau\times \partial_i G_\tau\right)\\
+ g^{\prime}(G_\tau(x_1,...,x_s))\times \partial_i \partial_j \partial_k G_\tau\\
\Rightarrow \sum_{i = 1}^s\sum_{j = 1}^s\sum_{k = 1}^s\vert\partial_i\partial_j\partial_k h_{\tau, \psi, z}\vert\leq g_*\psi^3 + 6g_*\tau\psi^2 + 6g_*\psi\tau^2
\end{aligned}
\end{equation}

Then we derive a lemma which is a corollary of \cite{MR3350040}. It
introduces some properties of joint Gaussian random variables.

\begin{lemma}
(i). Suppose $\xi_1,...,\xi_s$ are $s$ joint normal random variables with $\mathbf{E}\xi_i = 0$. Besides, suppose $\exists$ a positive real number(may not be a constant) $L$ and a
constant $c > 0$ such that $L>1$ and $c\leq \Vert\xi_i\Vert_2\leq L$ for $i = 1,...,s$. Then there exists a constant $C$ that is independent of $s,L$; and
\begin{equation}
\begin{aligned}
\sup_{x\in\mathbf{R}}\vert Prob\left(\max_{i = 1,...,s}\vert \xi_i\vert\leq x + \delta\right) - Prob\left(\max_{i = 1,...,s}\vert\xi_i\vert\leq x\right)\vert\\
\leq CL\delta\times(1 + \sqrt{\log(s)} + \sqrt{\vert\log(\delta)\vert} + \sqrt{\log(L)})
\label{eq.firGaussian}
\end{aligned}
\end{equation}

(ii). Define $\Sigma = \{\sigma_{ij}\}_{i, j = 1,...,s}$ such that $\sigma_{ij} = \mathbf{E}\xi_i\xi_j$. Suppose $\xi_i^\dagger, i = 1,...,s$ are joint normal random
variables with $\mathbf{E}\xi^\dagger_i = 0$. Define $\Sigma^\dagger  = \{\sigma_{ij}^\dagger\}_{i,j  = 1,...,s}$ such that
$\sigma^\dagger_{ij} = \mathbf{E}\xi_i^\dagger\xi_j^\dagger$ and $\Delta = \max_{i,j = 1,...,s}\vert\sigma_{ij} - \sigma_{ij}^\dagger\vert$. If $\Delta < 1$, then
\begin{equation}
\begin{aligned}
\sup_{x\in\mathbf{R}}\vert Prob\left(\max_{i = 1,...,s}\vert\xi_i\vert\leq x\right) - Prob\left(\max_{i = 1,...,s}\vert\xi^\dagger_i\vert\leq x\right)\vert\\
\leq \frac{CL\Delta^{1/6}}{(1 + \log(s))^{1/4}} + C\Delta^{1/3}(L + \log^3(s))+ \frac{CL\sqrt{\log(L)}\Delta^{1/3}}{(1 + \log(s))^{1/2}}
\label{eq.normal_2}
\end{aligned}
\end{equation}
\label{lemma.normal_property}
\end{lemma}

\begin{proof}[Proof of Lemma \ref{lemma.normal_property}]
For $\max_{i = 1,...,s}\vert\xi_i\vert = \max(\max_{i = 1,...,s}\xi_i, \max_{i = 1,...,s}(-\xi_i))$ and $(-\xi_1,...,-\xi_{p_1})$ has the same joint distribution as
$(\xi_1,...,\xi_{p_1})$,
\begin{equation}
\begin{aligned}
\sup_{x\in\mathbf{R}}Prob\left(x<\max_{i = 1,...,s}\vert\xi_i\vert\leq x + \delta\right)\leq \sup_{x\in\mathbf{R}} Prob\left(x < \max_{i = 1,...,s}\xi_i\leq x + \delta\right)\\
+ \sup_{x\in\mathbf{R}} Prob\left(x < \max_{i = 1,...,s}(-\xi_i)\leq x + \delta\right)\\
\leq  2\sup_{x\in\mathbf{R}} Prob\left(\vert\max_{i = 1,...,s}\xi_i - x\vert\leq \delta\right)
\end{aligned}
\end{equation}
From theorem 3 and (18), (19) in \cite{MR3350040}, define $\underline{\sigma}  = \min_{i = 1,...,s}\Vert\xi_i\Vert_2$
and $\overline{\sigma} = \max_{i = 1,...,s}\Vert\xi_i\Vert_2$,
\begin{equation}
\begin{aligned}
Prob\left(\vert\max_{i = 1,...,s}\xi_i - x\vert\leq \delta\right)\leq \frac{4\delta(1 + \sqrt{2\log(s)})}{\underline{\sigma}}\\
 + \frac{4\delta}{\underline{\sigma}}\times
\left(2 + \frac{\overline{\sigma}}{\underline{\sigma}}\sqrt{2\log(s)} + \frac{\overline{\sigma}}{\underline{\sigma}}\sqrt{2\vert\log(\underline{\sigma}/\delta)\vert}\right)
+ \frac{\delta}{\underline{\sigma}}(\sqrt{2\log(s)} + \sqrt{1\vee \log(\underline{\sigma}/\delta)})\\
\leq \frac{4\sqrt{2}\delta}{c}(1 + \sqrt{\log(s)}) + \frac{8\delta}{c}\times(1 + \frac{L}{c}\sqrt{\log(s)} + \frac{L}{c}\sqrt{\vert\log(c)\vert + \vert\log(\delta)\vert + \vert\log(L)\vert})\\
+ \frac{2\delta}{c}(\sqrt{\log(s)} + \sqrt{1 + \vert\log(c)\vert + \vert\log(\delta)\vert + \vert\log(L)\vert})\\
\leq C\delta L\times (1 + \sqrt{\log(s)} + \sqrt{\vert\log(\delta)\vert} + \sqrt{\log(L)})
\end{aligned}
\end{equation}
and we prove \eqref{eq.firGaussian}.

Without loss of generality, suppose $(\xi_1,...,\xi_s)$ are independent of $(\xi_1^\dagger,...,\xi^\dagger_s)$. Similar to \cite{MR3350040}, for
any $0\leq t\leq 1$, define the random variable $Z_i(t) = \sqrt{t}\xi_i + \sqrt{1 - t}\xi^\dagger_i$. According to theorem 2.27 in \cite{MR1681462} and lemma
2 in \cite{MR3350040},
for any $\tau,\psi > 0$,
\begin{equation}
\begin{aligned}
\mathbf{E}h_{\tau, \psi, x}(\xi_1,...,\xi_s) - \mathbf{E}h_{\tau, \psi, x}(\xi_1^\dagger,...,\xi_s^\dagger) = \mathbf{E}\int_{[0, 1]}\frac{d}{dt}h_{\tau,\psi,x}(Z_1(t),...,Z_s(t))dt\\
= \frac{1}{2}\sum_{i = 1}^s\int_{[0,1]} \mathbf{E}\partial_i h_{\tau, \psi,x}(Z_1(t),...,Z_s(t))\xi_i\times t^{-1/2}dt\\
- \frac{1}{2}\sum_{i = 1}^s\int_{[0, 1]}\mathbf{E}\partial_i h_{\tau, \psi, x}(Z_1(t),...,Z_s(t))\xi_i^\dagger\times (1 - t)^{-1/2} dt\\
= \frac{1}{2}\sum_{i = 1}^s\sum_{k = 1}^s(\sigma_{ik} - \sigma^\dagger_{ik})\int_{[0,1]}\mathbf{E}\partial_i\partial_k h_{\tau,\psi, x}(Z_1(t),...,Z_s(t))dt\\
\Rightarrow \sup_{x\in\mathbf{R}}\vert \mathbf{E}h_{\tau, \psi, x}(\xi_1,...,\xi_s) - \mathbf{E}h_{\tau, \psi, x}(\xi_1^\dagger,...,\xi_s^\dagger)\vert
\leq g_*\Delta\times (\psi^2 + \psi\tau)
\end{aligned}
\label{eq.h_property}
\end{equation}
Therefore, define $t = \frac{1}{\psi} + \frac{\log(2s)}{\tau}$, we have
\begin{equation}
\begin{aligned}
Prob\left(\max_{i = 1,...,s}\vert\xi_i\vert\leq x\right) - Prob\left(\max_{i = 1,...,s}\vert\xi^\dagger_i\vert\leq x\right)\\
\leq CLt(1 + \sqrt{\log(s)} + \sqrt{\vert\log(t)\vert} + \sqrt{\log(L)})\\
+ Prob\left(\max_{i = 1,...,s}\vert\xi_i\vert\leq x - t\right) - Prob\left(\max_{i = 1,...,s}\vert\xi^\dagger_i\vert\leq x\right)\\
\leq CLt(1 + \sqrt{\log(s)} + \sqrt{\vert\log(t)\vert} + \sqrt{\log(L)})\\
+ \mathbf{E}h_{\tau, \psi, x - \frac{1}{\psi}}(\xi_1,...,\xi_s) - \mathbf{E}h_{\tau, \psi, x - \frac{1}{\psi}}(\xi_1^\dagger,...,\xi_s^\dagger)\\
\text{and } Prob\left(\max_{i = 1,...,s}\vert\xi_i\vert\leq x\right) - Prob\left(\max_{i = 1,...,s}\vert\xi^\dagger_i\vert\leq x\right)\\
\geq -CLt(1 + \sqrt{\log(s)} + \sqrt{\vert\log(t)\vert} + \sqrt{\log(L)}) \\
+ Prob\left(\max_{i = 1,...,s}\vert\xi_i\vert\leq x + t\right) - Prob\left(\max_{i = 1,...,s}\vert\xi^\dagger_i\vert\leq x\right)\\
\geq -CLt(1 + \sqrt{\log(s)} + \sqrt{\vert\log(t)\vert} + \sqrt{\log(L)})\\
 + \mathbf{E}h_{\tau,\psi, x + \frac{\log(2s)}{\tau}}(\xi_1,...,\xi_s) - \mathbf{E}h_{\tau, \psi, x + \frac{\log(2s)}{\tau}}(\xi_1^\dagger,...\xi_s^\dagger)
\end{aligned}
\label{eq.prob_to_h}
\end{equation}
In particular,
\begin{equation}
\begin{aligned}
\sup_{x\in\mathbf{R}}\vert Prob\left(\max_{i = 1,...,s}\vert\xi_i\vert\leq x\right) - Prob\left(\max_{i = 1,...,s}\vert\xi^\dagger_i\vert\leq x\right)\vert\\
\leq CLt(1 + \sqrt{\log(s)} + \sqrt{\vert\log(t)\vert} + \sqrt{\log(L)})\\
+ \sup_{x\in\mathbf{R}}\vert \mathbf{E}h_{\tau,\psi, x}(\xi_1,...,\xi_s) - \mathbf{E}h_{\tau, \psi, x}(\xi_1^\dagger,...,\xi_s^\dagger)\vert\\
\leq CLt(1 + \sqrt{\log(s)} + \sqrt{\vert\log(t)\vert} + \sqrt{\log(L)}) + g_*\Delta(\psi^2 + \psi\tau)
\end{aligned}
\end{equation}
Choose $\tau = \psi = \left(1 + \log(2) +\log(s)\right)^{3/2}/ \Delta^{1/3}$, then $t = \frac{\Delta^{1/3}}{(1 + \log(2) + \log(s))^{1/2}} < 1$; and
\begin{equation}
\begin{aligned}
Lt(1 + \sqrt{\log(s)} + \sqrt{\vert\log(t)\vert} + \sqrt{\log(L)})\\
\leq 2L\Delta^{1/3} + CL\frac{\Delta^{1/6}}{(1 + \log(s))^{1/4}} + \frac{CL\sqrt{\log(L)}\Delta^{1/3}}{(1 + \log(s))^{1/2}}\\
\text{and } \Delta\psi^2 = \Delta^{1/3}\times (1 + \log(2) + \log(s))^3\leq C\Delta^{1/3}(1 + \log^3(s))
\end{aligned}
\end{equation}
and we prove \eqref{eq.normal_2}.

\end{proof}

\section{Proof of results of  section \ref{section.medium_range}}
\begin{proof}[Proof of Lemma \ref{lemma.linear_comb}]
Define the random variables $W_{i,j}^{(s)} = \mathbf{E}X_{i,j}|\mathcal{F}_{i,s}$ and $M_{k, j}^{(s)} = \sum_{i = T - k + 1}^T a_i(W_{i,j}^{(s)} - W_{i,j}^{(s - 1)})$, here $s\geq 1$. From corollary C.9 in \cite{MR2001996},
\begin{equation}
\lim_{s\to\infty} W_{i,j}^{(s)} = X_{i,j}\ \text{almost surely and } \lim_{s\to\infty}\Vert X_{i,j} - W_{i,j}^{(s)}\Vert_m =0
\end{equation}
We also know that $M_{k,j}^{(s)}$ is $\mathcal{F}_{T, k + s  -  1}$ measurable and
\begin{equation}
M_{k+1, j}^{(s)}  - M_{k, j}^{(s)} = a_{T - k}(W^{(s)}_{T-k, j} - W^{(s - 1)}_{T-k,j})
\end{equation}
Apply $\pi-\lambda$ system theorem on the $\lambda$(Dynkin) system
\begin{equation}
\left\{A\in\mathcal{F}_{T, k + s - 1}: \mathbf{E}W_{T - k, j}^{(s)}\times \mathbf{1}_A = \mathbf{E}W_{T - k, j}^{(s  - 1)}\times \mathbf{1}_A \right\}
\end{equation}
and the $\pi$ system $\{A_T\times A_{T - 1}\times ... \times A_{T - k - s  + 1}\}$, here $A_i$ is generated by $e_i$(see \eqref{eq.defEps}), we know that
$\mathbf{E}W_{T - k, j}^{(s)}|\mathcal{F}_{T, k + s  - 1} =  W_{T-k, j}^{(s - 1)}$ almost surely; and $M_{k, j}^{(s)}, k = 1,2,...,T$ form a martingale. Therefore, from the Burkholder's inequality(theorem 1.1 in \cite{MR0400380}) and theorem 2 in \cite{MR0133849},
\begin{equation}
\begin{aligned}
\Vert M_{T, j}^{(s)}\Vert_m\leq C\sqrt{\Vert\sum_{i = 1}^T a^2_i(W_{i,j}^{(s)} - W_{i,j}^{(s - 1)})^2\Vert_{m/2}}\leq C\sqrt{\sum_{i =  1}^T a^2_i}\times \delta_m(s)\\
\text{and } \Vert \sum_{i = 1}^T a_iX_{i, j }\Vert_m\leq \Vert\sum_{i = 1}^T a_i W_{i,j}^{(0)}\Vert + \sum_{s = 1}^\infty \Vert M_{T, j}^{(s)}\Vert_m\\
\leq C^\prime\sqrt{\sum_{i = 1}^T a^2_i\Vert X_{i,j}\Vert^2_m} + C\sqrt{\sum_{i = 1}^T a^2_i}\sum_{s = 1}^\infty\delta_m(s)
\end{aligned}
\label{eq.truncate}
\end{equation}
here $C,C^\prime$ are constants depending only on $m$. Similarly,
\begin{equation}
\begin{aligned}
\Vert\sum_{i = 1}^T a_i (X_{i,j} - \mathbf{E}X_{i, j}|\mathcal{F}_{i, s})\Vert_m\leq \sum_{k = s + 1}^\infty \Vert M_{T, j}^{(k)}\Vert_m\leq C\sqrt{\sum_{i = 1}^T a^2_i}\sum_{k = s + 1}^\infty \delta_m(k)
\end{aligned}
\end{equation}
From definition \ref{def.medium_range}, we prove \eqref{eq.linear_comb}.
\end{proof}

\begin{proof}[Proof of Lemma \ref{lemma.recognize}]
Notice that $\sum_{j = 0}^d a_{kj}(X_iX_{i - j} - \mathbf{E}X_iX_{i  - j})$ is a function of $..., e_{i - 1}, e_i$ and $\mathbf{E}\sum_{j = 0}^d a_{kj}(X_iX_{i - j} - \mathbf{E}X_iX_{i  - j}) = 0$. From lemma \ref{lemma.linear_comb} and \eqref{eq.cov}
\begin{equation}
\begin{aligned}
\Vert \sum_{j = 0}^d a_{kj}(X_iX_{i - j} - \mathbf{E}X_iX_{i  - j})\Vert_{m/2}\leq \Vert\sum_{j = 0}^d a_{kj}X_iX_{i - j}\Vert_{m/2} + \vert \sum_{j = 0}^da_{kj}\mathbf{E}X_iX_{i  - j}\vert\\
\leq \Vert X_i\Vert_m\times \Vert\sum_{j = 0}^d a_{kj}X_{i - j}\Vert_m +  \vert \sum_{j = 0}^da_{kj}\mathbf{E}X_iX_{i  - j}\vert\\
\leq C\sqrt{\sum_{j = 0}^d a^2_{kj}} + \sqrt{\sum_{j = 0}^d a^2_{kj}}\times \sqrt{\sum_{j = 0}^d (\mathbf{E}X_iX_{i - j})^2}\\
\leq \sqrt{\sum_{j = 0}^d a^2_{kj}}\times \left(C + C^\prime\sqrt{\sum_{j = 0}^d \frac{1}{(1 + j)^{2\alpha}}}\right)\\
\Rightarrow \max_{i\in\mathbf{Z}, k = 1,...,p_1}\Vert\sum_{j = 0}^d a_{kj}(X_iX_{i - j} - \mathbf{E}X_iX_{i  - j})\Vert_{m/2} = O(1)
\end{aligned}
\end{equation}
For any integer $s\geq 0$,
\begin{equation}
\begin{aligned}
\Vert Z_{i,k} - Z_{i,k}(s)\Vert_{m / 2}\leq\Vert X_i  - X_i(s)\Vert_m\times \Vert\sum_{j = 0}^d a_{kj}X_{i - j}\Vert_m\\
 + \Vert X_i(s)\Vert_m\times \Vert\sum_{j =0}^d a_{kj}(X_{i - j} - X_{i - j}(s - j))\Vert_m\\
 \leq C\delta_m(s) + C\sum_{j = 0}^d \vert a_{kj}\vert\times \delta_m(s - j)
\end{aligned}
\end{equation}
Then
\begin{equation}
\begin{aligned}
\sum_{s = s_0}^\infty \sup_{i\in\mathbf{Z}, k = 1,...,p_1}\Vert Z_{i,k} - Z_{i,k}(s)\Vert_{m / 2}\\
 \leq C\sum_{s = s_0}^\infty\delta_m(s) +  C\max_{k = 1,...,p_1, j = 0,...,d}\vert a_{kj}\vert\sum_{s = s_0}^\infty\sum_{j = 0}^d\delta_m(s - j)\\
 \leq \frac{C^\prime}{(1 + s_0)^\alpha} + C^\prime\sum_{j = 0}^d \sum_{s = s_0\vee j}^\infty\delta_m(s - j)\\
 \leq \frac{C^\prime}{(1 + s_0)^\alpha} + C^{\prime\prime}\sum_{j = 0}^d \frac{1}{(1 + 0 \vee (s_0 - j))^\alpha}\\
\end{aligned}
\end{equation}
If $s_0\geq d$, then
\begin{equation}
\begin{aligned}
\sum_{j = 0}^d \frac{1}{(1 + 0 \vee (s_0 - j))^{\alpha}} = \sum_{j = 0}^d \frac{1}{(1 + s_0 - j)^{\alpha}}\leq \frac{1}{(1 + s_0 - d)^{\alpha}} +
\int_{1 + s_0 - d}^{1 + s_0}\frac{1}{x^{\alpha}}dx\\
\leq \frac{C}{(1 + s_0 - d)^{\alpha - 1}}\\
\text{and } (1 + s_0)^{\alpha - 1}\frac{1}{(1 + s_0 - d)^{\alpha - 1}} = \left(1 + \frac{d}{1 + s_0 - d}\right)^{\alpha - 1}\leq (1 + d)^{\alpha - 1}
\end{aligned}
\end{equation}
On the other hand, if $s_0 < d$, then
\begin{equation}
\begin{aligned}
\sum_{j = 0}^d \frac{1}{(1 + 0 \vee (s_0 - j))^{\alpha}}  = (d - s_0) + \sum_{j = 0}^{s_0}\frac{1}{(1 + s_0 - j)^{\alpha}}\\
\text{and } (1 + s_0)^{\alpha - 1}\times (d - s_0)\leq (1 + d)^{\alpha}
\end{aligned}
\end{equation}
so $Z_{i, k}$ is $(\frac{m}{2}, \alpha - 1, \alpha\mathcal{B})$ - medium range dependent random variables and we prove (i).

From Theorem 3.1.1 in \cite{MR1093459}, there exists  $\psi_j, j = 0,1,...$ such that $\vert\psi_j\vert\leq c^{-j}$ with a constant $c > 1$ and $Y_i = \sum_{j = 0 }^\infty \psi_jX_{i - j}$. Therefore
\begin{equation}
\begin{aligned}
\mathbf{E}Y_i = \sum_{j = 0}^\infty \psi_j\mathbf{E}X_{i - j} = 0,\
\Vert Y_i\Vert_m\leq \sum_{j = 0}^\infty \vert\psi_j\vert\times \Vert X_{i - j}\Vert_m\leq C\sum_{j = 0}^\infty \vert\psi_j\vert\\
\text{and }\Vert Y_i  -Y_i(s)\Vert_m\leq\sum_{j = 0}^s\vert \psi_j\vert\times \Vert X_{i - j} - X_{i - j}(s - j)\Vert_m\\
\Rightarrow \sum_{s = s_0}^\infty \sup_{i\in\mathbf{Z}}\Vert Y_i  -Y_i(s)\Vert_m\leq \sum_{j = 0}^\infty c^{-j}\sum_{s =  s_0\vee j}^\infty \delta_m(s - j)
\leq C\sum_{j = 0}^\infty \frac{c^{-j}}{(1 + 0\vee (s_0 - j))^\alpha}
\end{aligned}
\label{eq.linear_process_1}
\end{equation}
Define $s_1 = \lfloor\frac{s_0}{2}\rfloor$, i.e., the largest integer that is smaller than or equal to $s_0 / 2$,
\begin{equation}
\begin{aligned}
\sum_{j = 0}^\infty \frac{c^{-j}}{(1 + 0\vee (s_0 - j))^\alpha}\leq \sum_{j = 0}^{s_1} \frac{c^{-j}}{(1 + \frac{s_0}{2})^\alpha} + \sum_{j = s_1 + 1}^\infty c^{-j}
\end{aligned}
\label{eq.linear_process_2}
\end{equation}
From \eqref{eq.linear_process_1} and \eqref{eq.linear_process_2}, we prove (ii).
\end{proof}

\begin{proof}[Proof of Theorem \ref{lemma.Gaussian}]
First from lemma \ref{lemma.recognize},  $Z_{i,k}, i\in\mathbf{Z}, k = 1,...,p_1$ are $(\frac{m}{2}, \alpha - 1, \alpha\mathcal{B})$-medium range
dependent random variables. Notice that $\xi_1,...,\xi_{p_1}$ are joint normal random variables and according to lemma \ref{lemma.linear_comb}

\begin{equation}
\begin{aligned}
\Vert \xi_k \Vert_2 = \Vert\frac{1}{\sqrt{T}}\sum_{i = 1}^T Z_{i, k}\Vert_2\leq \Vert\frac{1}{\sqrt{T}}\sum_{i = 1}^T Z_{i, k}\Vert_{m/2}
\leq CT^{\alpha\mathcal{B}}
\end{aligned}
\label{eq.xi}
\end{equation}
for a constant $C$.

For sufficiently large $\tau > 0$, define $t = \frac{1 + \log(2p_1)}{\tau}$. From lemma \ref{lemma.normal_property} and \eqref{eq.prob_to_h},
\begin{equation}
\begin{aligned}
\sup_{x\in\mathbf{R}}\vert Prob\left(\max_{k = 1,...,p_1}\vert\frac{1}{\sqrt{T}}\sum_{i = 1}^T Z_{i, k}\vert\leq x\right)
 -  Prob\left(\max_{k = 1,...,p_1}\vert\xi_k\vert\leq x\right)\vert\\
\leq \sup_{x\in\mathbf{R}}\vert \mathbf{E}h_{\tau,\tau, x}\left(\frac{1}{\sqrt{T}}\sum_{i = 1}^T Z_{i, 1},...,\frac{1}{\sqrt{T}}\sum_{i = 1}^T Z_{i, p_1}\right) - \mathbf{E}h_{\tau,\tau,x}(\xi_1,...,\xi_{p_1})\vert\\
+ CT^{\alpha\mathcal{B}}t\times \left(1 + \sqrt{\log(T)} + \sqrt{\vert\log(t)\vert}\right)
\end{aligned}
\label{eq.prob_to_hs}
\end{equation}
Define $Z_{i, k}^{(s)} = \mathbf{E}Z_{i,k}|\mathcal{F}_{i, s}$ for the integer $s\geq 0$. Then
\begin{equation}
\begin{aligned}
\Vert\frac{1}{\sqrt{T}}\sum_{i = 1}^T Z_{i, k} - \frac{1}{\sqrt{T}}\sum_{i = 1}^TZ_{i, k}^{(s)}\Vert_{m/2}\leq \frac{CT^{\alpha\mathcal{B}}}{(1 + s)^{\alpha - 1}}\\
\Rightarrow  \vert \mathbf{E}h_{\tau,\tau, x}\left(\frac{1}{\sqrt{T}}\sum_{i = 1}^T Z_{i, 1},...,\frac{1}{\sqrt{T}}\sum_{i = 1}^T Z_{i, p_1}\right)
- \mathbf{E}h_{\tau,\tau, x}\left(\frac{1}{\sqrt{T}}\sum_{i = 1}^T Z_{i, 1}^{(s)},...,\frac{1}{\sqrt{T}}\sum_{i = 1}^T Z_{i, p_1}^{(s)}\right)\vert\\
\leq g_*\tau\mathbf{E}\max_{k = 1,...,p_1}\vert\frac{1}{\sqrt{T}}\sum_{i = 1}^T (Z_{i, k} - Z_{i,k}^{(s)})\vert\\
\leq C\tau p_1^{2/m}\max_{k = 1,...,p_1}\Vert \frac{1}{\sqrt{T}}\sum_{i = 1}^T Z_{i, k} - \frac{1}{\sqrt{T}}\sum_{i = 1}^TZ_{i, k}^{(s)}\Vert_{m/2}
\leq C^\prime\tau\times\frac{T^{\alpha\mathcal{B} + \frac{2\alpha_{p_1}}{m}}}{(1 + s)^{\alpha - 1}}
\end{aligned}
\label{eq.sec_part}
\end{equation}

For any integer $l > s$, define the `big block' $\mathcal{S}_{jk}$ and the `small block' $\mathcal{U}_{jk}$ as
\begin{equation}
\mathcal{S}_{jk} = \frac{1}{\sqrt{T}}\sum_{v = (j -  1)\times (s + l) + 1}^{[(j - 1)\times (s + l) + l]\wedge T}Z_{v, k}^{(s)}\
\text{and } \mathcal{U}_{jk} = \frac{1}{\sqrt{T}}\sum_{v = (j - 1)\times(s + l) + l + 1}^{[j\times (s + l)]\wedge T}Z_{v, k}^{(s)}
\end{equation}
with $k = 0, 1,...,p_1$ and $j  = 1,2,...,\lceil\frac{T}{s + l}\rceil$, here $\lceil x\rceil$ represents the smallest integer that is larger than or equal to $x$.
Define $V_0 = \lceil\frac{T}{s + l}\rceil$, then
$\sum_{j = 1}^{V_0}\mathcal{S}_{jk} + \sum_{j = 1}^{V_0}\mathcal{U}_{jk} = \frac{1}{\sqrt{T}}\sum_{i  =1}^T Z_{i, k}^{(s)}$. Besides,
$(\mathcal{S}_{j1},...,\mathcal{S}_{jp_1})$ are mutually independent; $(\mathcal{U}_{j1},...,\mathcal{U}_{jp_1})$ are mutually
independent. Then from \eqref{eq.truncate}
\begin{equation}
\begin{aligned}
\vert\mathbf{E}h_{\tau,\tau,x}\left(\frac{1}{\sqrt{T}}\sum_{i = 1}^T Z_{i, 1},...,\frac{1}{\sqrt{T}}\sum_{i = 1}^T Z_{i, p_1}\right)
- \mathbf{E}h_{\tau, \tau, x}\left(\sum_{j = 1}^{V_0}\mathcal{S}_{j1},...,\sum_{j = 1}^{V_0}\mathcal{S}_{jp_1}\right)\vert\\
\leq g_*\tau\mathbf{E}\max_{k = 1,...,p_1}\vert\sum_{j = 1}^{V_0}\mathcal{U}_{jk}\vert\leq
g_*\tau p_1^{2/m}\times \max_{k = 1,...,p_1}\Vert\sum_{j = 1}^{V_0}\mathcal{U}_{jk}\Vert_{m/2}\\
\leq C\tau\times \frac{T^{\frac{2\alpha_{p_1}}{m}+\alpha\mathcal{B}}\sqrt{sV_0}}{\sqrt{T}}
\end{aligned}
\label{eq.third_part}
\end{equation}
Define $(\mathcal{S}^*_{j1},...\mathcal{S}^*_{jp_1}), j = 1,2,...V_0$ such that $(\mathcal{S}^*_{j1},...\mathcal{S}^*_{jp_1})$ are joint normal random
variables, $\mathbf{E}\mathcal{S}^*_{jk} = 0$, $\mathbf{E}\mathcal{S}^*_{jk_1}\mathcal{S}^*_{jk_2} = \mathbf{E}\mathcal{S}_{jk_1}\mathcal{S}_{jk_2}$;
$(\mathcal{S}^*_{j1},...\mathcal{S}^*_{jp_1})$ are mutually independent, and $\mathcal{S}^*_{j_1k_1}$ is independent of $\mathcal{S}_{j_2k_2}$ for
arbitrary $j_1,j_2,k_1,k_2$. Define $\mathcal{H}_{ik} = \sum_{j  = 1}^{i - 1}\mathcal{S}_{jk} + \sum_{j = i + 1}^{V_0}\mathcal{S}^*_{jk}$, then
we have $\mathcal{H}_{ik} + \mathcal{S}_{ik} = \mathcal{H}_{i+1k} + \mathcal{S}^*_{i+1k}$. Besides,
\begin{equation}
\begin{aligned}
\vert \mathbf{E}h_{\tau, \tau, x}\left(\sum_{j = 1}^{V_0}\mathcal{S}_{j1},...,\sum_{j = 1}^{V_0}\mathcal{S}_{jp_1}\right)
- \mathbf{E}h_{\tau, \tau, x}\left(\sum_{j = 1}^{V_0}\mathcal{S}^*_{j1},...,\sum_{j = 1}^{V_0}\mathcal{S}^*_{jp_1}\right)\vert\\
= \vert\mathbf{E}h_{\tau,\tau,x}\left(\mathcal{H}_{V_01} + \mathcal{S}_{V_01},...,\mathcal{H}_{V_0p_1}+\mathcal{S}_{V_0p_1}\right)
-\mathbf{E}h_{\tau,\tau,x}\left(\mathcal{H}_{11} + \mathcal{S}^*_{11},...,\mathcal{H}_{1p_1} + \mathcal{S}^*_{1p_1}\right)\vert\\
\leq \sum_{l = 1}^{V_0}\vert \mathbf{E}h_{\tau,\tau,x}\left(\mathcal{H}_{l1} + \mathcal{S}_{l1},...,\mathcal{H}_{lp_1}+\mathcal{S}_{lp_1}\right)
-\mathbf{E}h_{\tau,\tau,x}\left(\mathcal{H}_{l1} + \mathcal{S}^*_{l1},...,\mathcal{H}_{lp_1} + \mathcal{S}^*_{lp_1}\right)\vert
\end{aligned}
\end{equation}
According to \eqref{eq.truncate},
\begin{equation}
\begin{aligned}
\vert\mathbf{E}\left(h_{\tau,\tau,x}\left(\mathcal{H}_{l1} + \mathcal{S}_{l1},...,\mathcal{H}_{lp_1}+\mathcal{S}_{lp_1}\right) -h_{\tau,\tau,x}\left(\mathcal{H}_{l1} + \mathcal{S}^*_{l1},...,\mathcal{H}_{lp_1} + \mathcal{S}^*_{lp_1}\right) \right)
|\mathcal{H}_{l1},...,\mathcal{H}_{lp_1}\vert\\
\leq \vert\sum_{k = 1}^{p_1}\partial_k h_{\tau,\tau,x}(\mathcal{H}_{l1},...,\mathcal{H}_{lp_1})\times (\mathbf{E}\mathcal{S}_{lk} - \mathbf{E}\mathcal{S}^*_{lk})\vert\\
+ \frac{1}{2}\vert\sum_{k_1 = 1}^{p_1}\sum_{k_2 = 1}^{p_1}\partial_{k_1}\partial_{k_2}h_{\tau,\tau,x}(\mathcal{H}_{l1},...,\mathcal{H}_{lp_1})
\times(\mathbf{E}\mathcal{S}_{lk_1}\mathcal{S}_{lk_2} - \mathbf{E}\mathcal{S}_{lk_1}^*\mathcal{S}_{lk_2}^*)\vert\\
+ 3g_*\tau^3\mathbf{E}\max_{k = 1,...,p_1}\vert\mathcal{S}_{lk}\vert^3 + 3g_*\tau^3\mathbf{E}\max_{k = 1,...,p_1}\vert\mathcal{S}^*_{lk}\vert^3\\
\leq C\tau^3 p_1^{\frac{6}{m}}\times\max_{k = 1,...,p_1}\Vert\mathcal{S}_{lk}\Vert^3_{m/2}\\
\leq C^\prime \tau^3\times \frac{T^{\frac{6\alpha_{p_1}}{m} + 3\alpha\mathcal{B}}\times l^{3/2}}{T^{3/2}}\\
\text{which implies } \sup_{x\in\mathbf{R}}\vert \mathbf{E}h_{\tau, \tau, x}\left(\sum_{j = 1}^{V_0}\mathcal{S}_{j1},...,\sum_{j = 1}^{V_0}\mathcal{S}_{jp_1}\right)
- \mathbf{E}h_{\tau, \tau, x}\left(\sum_{j = 1}^{V_0}\mathcal{S}^*_{j1},...,\sum_{j = 1}^{V_0}\mathcal{S}^*_{jp_1}\right)\vert\\
\leq C^{\prime\prime}\tau^3\times T^{\frac{6\alpha_{p_1}}{m} + 3\alpha\mathcal{B} - \frac{1}{2}}\times \sqrt{l}
\end{aligned}
\label{eq.fourth_part}
\end{equation}
Notice that
\begin{equation}
\begin{aligned}
\vert\mathbf{E}\xi_{k_1}\xi_{k_2}  - \mathbf{E}\sum_{j_1 = 1}^{V_0}\sum_{j_2 = 1}^{V_0}\mathcal{S}^*_{j_1k_1}\mathcal{S}^*_{j_2k_2}\vert\\
= \vert\frac{1}{T}\sum_{i_1 = 1}^T\sum_{i_2 = 1}^T\mathbf{E}Z_{i_1, k_1}Z_{i_2, k_2}
- \sum_{j_1 = 1}^{V_0}\sum_{j_2 = 1}^{V_0}\mathbf{E}\mathcal{S}_{j_1k_1}\mathcal{S}_{j_2k_2}\vert\\
\leq \vert\frac{1}{T}\sum_{i_1 = 1}^T\sum_{i_2 = 1}^T\mathbf{E}Z_{i_1, k_1}Z_{i_2, k_2}
- \frac{1}{T}\sum_{i_1 = 1}^T\sum_{i_2 = 1}^T\mathbf{E}Z_{i_1, k_1}^{(s)}Z_{i_2, k_2}^{(s)}\vert
+ \vert\sum_{j_1 = 1}^{V_0}\sum_{j_2 = 1}^{V_0}\mathbf{E}\mathcal{S}_{j_1k_1}\mathcal{U}_{j_2k_2}\vert\\
+ \vert\sum_{j_1 = 1}^{V_0}\sum_{j_2 = 1}^{V_0}\mathbf{E}\mathcal{U}_{j_1k_1}\mathcal{S}_{j_2k_2}\vert
+ \vert\sum_{j_1 = 1}^{V_0}\sum_{j_2 = 1}^{V_0}\mathbf{E}\mathcal{U}_{j_1k_1}\mathcal{U}_{j_2k_2}\vert
\end{aligned}
\end{equation}
From lemma \ref{lemma.recognize} and \eqref{eq.truncate},
\begin{equation}
\begin{aligned}
\vert\sum_{j_1 = 1}^{V_0}\sum_{j_2 = 1}^{V_0}\mathbf{E}\mathcal{S}_{j_1k_1}\mathcal{U}_{j_2k_2}\vert
\leq \Vert\sum_{j_1 = 1}^{V_0}\mathcal{S}_{j_1k_1}\Vert_{m/2}\times \Vert\sum_{j_2 = 1}^{V_0}\mathcal{U}_{j_2k_2}\Vert_{m/2}
\leq \frac{CT^{2\alpha\mathcal{B}}\times V_0\sqrt{sl}}{T}\\
\text{and }\vert\sum_{j_1 = 1}^{V_0}\sum_{j_2 = 1}^{V_0}\mathbf{E}\mathcal{U}_{j_1k_1}\mathcal{U}_{j_2k_2}\vert
\leq \Vert \sum_{j_1 = 1}^{V_0}\mathcal{U}_{j_1k_1}\Vert_{m/2}\times \Vert \sum_{j_2 = 1}^{V_0}\mathcal{U}_{j_2k_2}\Vert_{m/2}
\leq \frac{CT^{2\alpha\mathcal{B}}\times V_0s}{T}
\end{aligned}
\end{equation}
On the other hand,  from lemma \ref{lemma.linear_comb} and \eqref{eq.truncate}
\begin{equation}
\begin{aligned}
\vert\frac{1}{T}\sum_{i_1 = 1}^T\sum_{i_2 = 1}^T Z_{i_1, k_1}Z_{i_2, k_2} - \frac{1}{T}\sum_{i_1 = 1}^T\sum_{i_2 = 1}^T
\mathbf{E}Z_{i_1, k_1}^{(s)}Z_{i_2, k_2}^{(s)}\vert\\
\leq \vert \frac{1}{T}\sum_{i_1 = 1}^T\sum_{i_2 = 1}^T Z_{i_1, k_1}Z_{i_2, k_2} - \frac{1}{T}\sum_{i_1 = 1}^T\sum_{i_2 = 1}^T Z_{i_1, k_1}^{(s)}Z_{i_2, k_2}\vert\\
+ \vert \frac{1}{T}\sum_{i_1 = 1}^T\sum_{i_2 = 1}^T Z_{i_1, k_1}^{(s)}Z_{i_2, k_2} - \frac{1}{T}\sum_{i_1 = 1}^T\sum_{i_2 = 1}^T Z_{i_1, k_1}^{(s)}Z_{i_2, k_2}^{(s)}\vert\\
\leq \frac{1}{T}\Vert\sum_{i = 1}^T (Z_{i, k_1} - Z_{i, k_1}^{(s)})\Vert_{m/2}\times \Vert\sum_{i = 1}^T Z_{i, k_2}\Vert_{m/2}
+ \frac{1}{T}\Vert\sum_{i = 1}^T Z_{i, k_1}^{(s)}\Vert_{m/2}\times \Vert \sum_{i = 1}^T (Z_{i, k_2} - Z_{i, k_2}^{(s)})\Vert_{m/2}\\
\leq \frac{CT^{2\alpha\mathcal{B}}}{(1 + s)^{\alpha - 1}}
\end{aligned}
\end{equation}
From \eqref{eq.h_property},
\begin{equation}
\begin{aligned}
\sup_{x\in\mathbf{R}}\vert\mathbf{E}h_{\tau,\tau,x}\left(\sum_{j = 1}^{V_0}\mathcal{S}^*_{j1},...,\sum_{j = 1}^{V_0}\mathcal{S}^*_{jp_1}\right)
- \mathbf{E}h_{\tau,\tau,x}\left(\xi_1,...\xi_{p_1}\right)\vert\\
\leq C\tau^2\times \left(\frac{T^{2\alpha\mathcal{B}}}{(1 + s)^{\alpha - 1}} + \frac{T^{2\alpha\mathcal{B}}V_0\sqrt{sl}}{T}\right)
\end{aligned}
\label{eq.fifth_part}
\end{equation}
From \eqref{eq.prob_to_hs}, \eqref{eq.sec_part}, \eqref{eq.third_part}, \eqref{eq.fourth_part} and \eqref{eq.fifth_part},
\begin{equation}
\begin{aligned}
\sup_{x\in\mathbf{R}}\vert Prob\left(\max_{k = 1,...,p_1}\vert\frac{1}{\sqrt{T}}\sum_{i = 1}^T Z_{i, k}\vert\leq x\right)
 -  Prob\left(\max_{k = 1,...,p_1}\vert\xi_k\vert\leq x\right)\vert\\
 \leq CT^{\alpha\mathcal{B}}t\times \left(1 + \sqrt{\log(T)} + \sqrt{\vert\log(t)\vert}\right)
 +  C\tau\times\frac{T^{\alpha\mathcal{B} + \frac{2\alpha_{p_1}}{m}}}{(1 + s)^{\alpha - 1 }}\\
 + C\tau\times \frac{T^{\frac{2\alpha_{p_1}}{m}+\alpha\mathcal{B}}\sqrt{sV_0}}{\sqrt{T}}
 + C\tau^3\times T^{\frac{6\alpha_{p_1}}{m} + 3\alpha\mathcal{B} - \frac{1}{2}}\times \sqrt{l}
 + C\tau^2\times \left(\frac{T^{2\alpha\mathcal{B}}}{(1 + s)^{\alpha - 1}} + \frac{T^{2\alpha\mathcal{B}}V_0\sqrt{sl}}{T}\right)
\end{aligned}
\end{equation}
Choose $s = \lfloor T^{\alpha_s}\rfloor$ and $l = \lfloor T^{\alpha_l}\rfloor$, here $\alpha_s,\ \alpha_l$ satisfy \eqref{eq.cond_alpha} and $\lfloor x\rfloor$ denotes the largest integer that is smaller than or equal to
$x$.
Then select $\tau = T^{\alpha\beta + \lambda}$ such that $\lambda>0$ and
\begin{equation}
\begin{aligned}
2\alpha\mathcal{B} + \frac{2\alpha_{p_1}}{m} - (\alpha - 1)\alpha_s + \lambda<0, \ 2\alpha\mathcal{B} + \frac{2\alpha_{p_1}}{m} + \frac{1}{2}\alpha_s
- \frac{1}{2}\alpha_l + \lambda < 0\\
6\alpha\mathcal{B} + \frac{6\alpha_{p_1}}{m} + \frac{1}{2}\alpha_l - \frac{1}{2} + 3\lambda < 0,\ 4\alpha\mathcal{B} + 2\lambda - (\alpha - 1)\alpha_s < 0\\
4\alpha\mathcal{B} + \frac{1}{2}\alpha_s - \frac{1}{2}\alpha_l + 2\lambda < 0\\
\end{aligned}
\end{equation}
According to \eqref{eq.cond_alpha}, this $\lambda $ exists. We have $t  = \frac{1 + \log(2p_1)}{\tau}\leq CT^{-\alpha\mathcal{B} -\lambda}\times \log(T) < 1$ for
sufficiently large $T$; and $t\geq \frac{1}{\tau} = T^{-\alpha\mathcal{B} - \lambda}$. For sufficiently large $T$
\begin{equation}
\begin{aligned}
T^{\alpha\mathcal{B}}t\times \left(1 + \sqrt{\log(T)} + \sqrt{\vert\log(t)\vert}\right)\leq CT^{-\lambda}(\log(T))^{3/2}\\
\tau\times\frac{T^{\alpha\mathcal{B} + \frac{2\alpha_{p_1}}{m}}}{(1 + s)^{\alpha - 1}}\leq CT^{2\alpha\mathcal{B} + \lambda + \frac{2\alpha_{p_1}}{m} - (\alpha - 1)\alpha_s}\\
\tau\times \frac{T^{\frac{2\alpha_{p_1}}{m}+\alpha\mathcal{B}}\sqrt{sV_0}}{\sqrt{T}}\leq CT^{2\alpha\mathcal{B} + \frac{2\alpha_{p_1}}{m} + \lambda + \frac{\alpha_s}{2} - \frac{\alpha_l}{2}}\\
\tau^3\times T^{\frac{6\alpha_{p_1}}{m} + 3\alpha\mathcal{B} - \frac{1}{2}}\times \sqrt{l}\leq CT^{6\alpha\mathcal{B} + \frac{6\alpha_{p_1}}{m} + 3\lambda + \frac{\alpha_l}{2} - \frac{1}{2}}\\
\frac{\tau^2\times T^{2\alpha\mathcal{B}}}{(1 + s)^{\alpha - 1}}\leq CT^{4\alpha\mathcal{B} + 2\lambda - (\alpha - 1)\alpha_s},\ \frac{\tau^2T^{2\alpha\mathcal{B}}V_0\sqrt{sl}}{T}
\leq CT^{4\alpha\mathcal{B} + 2\lambda  + \frac{\alpha_s}{2} - \frac{\alpha_l}{2}}
\end{aligned}
\end{equation}
and we prove \eqref{eq.delta_prob}.
\end{proof}

\begin{proof}[Proof of Lemma \ref{lemma.covariance}]
First notice that
\begin{equation}
\begin{aligned}
\max_{j_1,j_2 = 1,...,p_1}\vert\frac{1}{T}\sum_{i_1 = 1}^T\sum_{i_2 = 1}^T Z_{i_1, j_1}Z_{i_2, j_2}K\left(\frac{i_1 - i_2}{k_T}\right) - \frac{1}{T}\sum_{i_1 = 1}^T\sum_{i_2 = 1}^T\mathbf{E}Z_{i_1, j_1}Z_{i_2, j_2}\vert\\
\leq \max_{j_1,j_2 = 1,...,p_1}\vert \frac{1}{T}\sum_{i_1 = 1}^T\sum_{i_2 = 1}^T (Z_{i_1, j_1}Z_{i_2, j_2} - \mathbf{E}Z_{i_1, j_1}Z_{i_2, j_2})K\left(\frac{i_1 - i_2}{k_T}\right)\vert\\
+ \max_{j_1,j_2 =1,...,p_1}\vert\frac{1}{T}\sum_{i_1 = 1}^T\sum_{i_2 = 1}^T \mathbf{E}Z_{i_1, j_1}Z_{i_2, j_2}\times \left(K\left(\frac{i_1 - i_2}{k_T}\right) - 1\right) \vert\\
\end{aligned}
\end{equation}
From lemma \ref{lemma.recognize}, $Z_{i, k}$  are $(\frac{m}{2}, \alpha - 1, \alpha\mathcal{B})$-medium range dependent random variables.
From \eqref{eq.cov}, there exists a constant $C$ such that
$\vert \mathbf{E}Z_{i_1, j_1}Z_{i_2, j_2}\vert\leq \frac{CT^{\alpha\mathcal{B}}}{(1 + \vert i_1 - i_2\vert)^{\alpha - 1}}$.
From section 0.9.7 in \cite{MR2978290},
\begin{equation}
\begin{aligned}
\vert\frac{1}{T}\sum_{i_1 = 1}^T\sum_{i_2 = 1}^T \mathbf{E}Z_{i_1, j_1}Z_{i_2, j_2}\times \left(K\left(\frac{i_1 - i_2}{k_T}\right) - 1\right) \vert\\
\leq \frac{CT^{\alpha\mathcal{B}}}{T}\sum_{i_1 = 1}^T\sum_{i_2  = 1}^T\left(1 - K\left(\frac{i_1 - i_2}{k_T}\right)\right)\times \frac{1}{(1 + \vert i_1 - i_2\vert)^{\alpha - 1}}\\
\leq 2CT^{\alpha\mathcal{B}}\sum_{l = 0}^\infty \left(1 - K\left(\frac{l}{k_T}\right)\right)\times \frac{1}{(1 + l)^{\alpha - 1}}
\end{aligned}
\end{equation}
Set $L = \lfloor \frac{k_T}{2}\rfloor > 0$ for sufficiently large $T$, then
\begin{equation}
\begin{aligned}
\sum_{l = 0}^\infty \left(1 - K\left(\frac{l}{k_T}\right)\right)\times \frac{1}{(1 + l)^{\alpha - 1}}\\
\leq \frac{\sup_{x\in[0,1]}\vert K^\prime(x)\vert}{k_T}\sum_{l = 0}^L \frac{1}{(1 + l)^{\alpha - 2}}
+ \sum_{l = L+1}^\infty \frac{1}{(1 + l)^{\alpha - 1}}
\end{aligned}
\end{equation}
For $\alpha > 2$,
\begin{equation}
\begin{aligned}
\sum_{l = 0}^L \frac{1}{(1 + l)^{\alpha - 2}}\leq 1 + \int_{[1, L+1]}\frac{\mathrm{d}x}{x^{\alpha - 2}}\\
\text{and } \sum_{l = L+1}^\infty \frac{1}{(1 + l)^{\alpha - 1}}\leq \int_{[L+1,\infty)}\frac{\mathrm{d}x}{x^{\alpha - 1}}
= \frac{1}{\alpha - 2}(L+1)^{2-\alpha}
\end{aligned}
\end{equation}
For $\int_{[1, L+1]}\frac{\mathrm{d}x}{x^{\alpha - 2}}\leq C(L + 1)^{3-\alpha}$ if $2<\alpha <3$;
$\int_{[1, L+1]}\frac{\mathrm{d}x}{x^{\alpha - 2}}\leq C\log (L + 1)$ if $\alpha = 3$; and
$\int_{[1, L+1]}\frac{\mathrm{d}x}{x^{\alpha - 2}}\leq C$ if $\alpha > 3$, we have
\begin{equation}
\begin{aligned}
\max_{j_1,j_2 = 1,...,p_1}\vert\frac{1}{T}\sum_{i_1 = 1}^T\sum_{i_2 = 1}^T \mathbf{E}Z_{i_1, j_1}Z_{i_2, j_2}\times \left(K\left(\frac{i_1 - i_2}{k_T}\right) - 1\right) \vert
= O\left(v_T\times T^{\alpha\mathcal{B}}\right)
\end{aligned}
\end{equation}
On the other hand,
\begin{equation}
\begin{aligned}
\Vert \frac{1}{T}\sum_{i_1 = 1}^T\sum_{i_2 = 1}^T (Z_{i_1, j_1}Z_{i_2, j_2} - \mathbf{E}Z_{i_1, j_1}Z_{i_2, j_2})K\left(\frac{i_1 - i_2}{k_T}\right)\Vert_{m/4}\\
\leq \frac{1}{T}\sum_{s = 0}^{T  - 1}K\left(\frac{s}{k_T}\right)\Vert \sum_{i_2 = 1}^{T  - s}(Z_{i_2, j_2}Z_{i_2 + s, j_1} - \mathbf{E}Z_{i_2, j_2}Z_{i_2 + s, j_1})\Vert_{m/4}\\
+ \frac{1}{T}\sum_{s = 1}^{T -  1}K\left(\frac{s}{k_T}\right)\Vert \sum_{i_1 = 1}^{T  - s}(Z_{i_1, j_1}Z_{i_1 + s, j_2} - \mathbf{E}Z_{i_1, j_1}Z_{i_1 + s, j_2})\Vert_{m/4}
\end{aligned}
\end{equation}
Form Theorem 2 in \cite{MR0133849} and \eqref{eq.truncate}
\begin{equation}
\begin{aligned}
\Vert \sum_{i_2 = 1}^{T  - s}(Z_{i_2, j_2}Z_{i_2 + s, j_1} - \mathbf{E}Z_{i_2, j_2}Z_{i_2 + s, j_1})\Vert_{m/4}\\
\leq \Vert \sum_{i_2 = 1}^{T  - s}(\mathbf{E}(Z_{i_2, j_2}Z_{i_2 + s, j_1})|\mathcal{F}_{i_2 + s, 0} - \mathbf{E}Z_{i_2, j_2}Z_{i_2 + s, j_1})\Vert_{m/4}\\
+ \sum_{t = 1}^\infty\Vert \sum_{i_2 = 1}^{T - s}(\mathbf{E}(Z_{i_2, j_2}Z_{i_2 + s, j_1})|\mathcal{F}_{i_2 + s, t} - \mathbf{E}Z_{i_2, j_2}Z_{i_2 + s, j_1}|\mathcal{F}_{i_2 + s, t - 1})\Vert_{m/4}\\
\leq C\sqrt{T}\max_{i_2 = 1,...,T-s}\Vert Z_{i_2, j_2}Z_{i_2 + s, j_1} - \mathbf{E}Z_{i_2, j_2}Z_{i_2 + s, j_1}\Vert_{m/4}\\
+ C\sqrt{T}\sum_{t = 1}^\infty \max_{i_2 = 1,...,T-s}\Vert Z_{i_2, j_2}(t - s)Z_{i_2 + s , j_1}(t) - Z_{i_2, j_2}Z_{i_2 + s , j_1}\Vert_{m/4}
\end{aligned}
\end{equation}
For
\begin{equation}
\begin{aligned}
\Vert Z_{i_2, j_2}Z_{i_2 + s, j_1} - \mathbf{E}Z_{i_2, j_2}Z_{i_2 + s, j_1}\Vert_{m/4}\leq
\Vert Z_{i_2, j_2}\Vert_{m/2}\Vert Z_{i_2 + s, j_1}\Vert_{m/2} + \vert\mathbf{E}Z_{i_2, j_2}Z_{i_2 + s, j_1}\vert\\
\leq 2\Vert Z_{i_2, j_2}\Vert_{m/2}\Vert Z_{i_2 + s, j_1}\Vert_{m/2}\leq C\\
\text{and  } \sum_{t = 1}^\infty \max_{i_2 = 1,...,T-s}\Vert Z_{i_2, j_2}(t - s)Z_{i_2 + s , j_1}(t) - Z_{i_2, j_2}Z_{i_2 + s , j_1}\Vert_{m/4}\\
\leq \sum_{t = 1}^\infty \max_{i_2 = 1,...,T-s}\Vert Z_{i_2, j_2}(t - s)\Vert_{m/2}\times \Vert Z_{i_2 +s, j_1}(t) - Z_{i_2 + s, j_1}\Vert_{m/2}\\
+ \sum_{t = 1}^\infty \max_{i_2 = 1,...,T - s}\Vert Z_{i_2 + s, j_1}\Vert_{m/2}\times \Vert Z_{i_2, j_2}(t - s)  - Z_{i_2, j_2}\Vert_{m/2}\leq CT^{\alpha\mathcal{B}}
\end{aligned}
\end{equation}
Therefore,
\begin{equation}
\begin{aligned}
\Vert \frac{1}{T}\sum_{i_1 = 1}^T\sum_{i_2 = 1}^T (Z_{i_1, j_1}Z_{i_2, j_2} - \mathbf{E}Z_{i_1, j_1}Z_{i_2, j_2})K\left(\frac{i_1 - i_2}{k_T}\right)\Vert_{m/4}
\leq \frac{CT^{\alpha\mathcal{B}}}{\sqrt{T}}\sum_{s = 0}^\infty K\left(\frac{s}{k_T}\right)
\end{aligned}
\end{equation}
Since $\sum_{s = 0}^\infty K\left(\frac{s}{k_T}\right) \leq K(0) + \int_{[0,\infty)}K\left(\frac{x}{k_T}\right)dx\leq Ck_T$,
\begin{equation}
\begin{aligned}
\Vert\ \max_{j_1,j_2 = 1,...,p_1}\vert\frac{1}{T}\sum_{i_1 = 1}^T\sum_{i_2 = 1}^T Z_{i_1, j_1}Z_{i_2, j_2}K\left(\frac{i_1 - i_2}{k_T}\right) - \frac{1}{T}\sum_{i_1 = 1}^T\sum_{i_2 = 1}^T\mathbf{E}Z_{i_1, j_1}Z_{i_2, j_2}\vert\ \Vert_{m/4}\\
\leq \max_{j_1,j_2 =1,...,p_1}\vert\frac{1}{T}\sum_{i_1 = 1}^T\sum_{i_2 = 1}^T \mathbf{E}Z_{i_1, j_1}Z_{i_2, j_2}\times \left(K\left(\frac{i_1 - i_2}{k_T}\right) - 1\right) \vert\\
+p_1^{\frac{8}{m}}\max_{j_1, j_2 = 1,...,p_1} \Vert \frac{1}{T}\sum_{i_1 = 1}^T\sum_{i_2 = 1}^T (Z_{i_1, j_1}Z_{i_2, j_2} - \mathbf{E}Z_{i_1, j_1}Z_{i_2, j_2})K\left(\frac{i_1 - i_2}{k_T}\right)\Vert_{m/4}\\
= O\left(v_T\times T^{\alpha\mathcal{B}} + k_T\times T^{\frac{8\alpha_{p_1}}{m} + \alpha\mathcal{B} - \frac{1}{2}}\right)
\end{aligned}
\end{equation}
and we prove \eqref{eq.cov_converge}.
\end{proof}

\section{Proof of Theorem \ref{theorem.auto_covariance}}
\begin{proof}[Proof of Theorem \ref{theorem.auto_covariance}]
Suppose $\mathcal{H}  = \{h_1, h_2,...,h_z\}$ with $z  = \vert\mathcal{H}\vert\leq (d + 1) = O(T^{\beta_X})$,
then define the real numbers $b_{ij}, i = 1,2,...,z, j = 0,1,...d$ such that $b_{ih_i} = 1$ and $b_{ij} = 0$ for $j\neq h_i$.
We have $\sum_{j = 0}^d b_{ij}^2 =  1$ and
\begin{equation}
\frac{1}{\sqrt{T}}\sum_{i = 1}^T \sum_{j = 0}^d b_{kj}(X_iX_{i  - j} - \sigma_{j}) =
\frac{1}{\sqrt{T}}\sum_{i = 1}^T(X_iX_{i - h_k} - \sigma_{h_k})
\end{equation}
From Theorem \ref{lemma.Gaussian},
\begin{equation}
\sup_{x\in\mathbf{R}}\vert Prob\left(\max_{k \in\mathcal{H}}\vert \frac{1}{\sqrt{T}}\sum_{i = 1}^T X_iX_{i - k} - \sqrt{T}\sigma_k\vert\leq x\right)
 -  Prob\left(\max_{k\in\mathcal{H}}\vert\xi_k\vert\leq x\right)\vert = o(1)
\end{equation}
From Lemma \ref{lemma.recognize}, the random variables $X_iX_{i - k} - \sigma_k, i\in\mathbf{Z}, k = 0,1,...d$ are $(\frac{m}{2}, \alpha_X - 1, \alpha_X\beta_X)$
medium range dependent random variables. According to Lemma \ref{lemma.linear_comb} and  expression \eqref{eq.cov},
\begin{equation}
\begin{aligned}
\Vert \sqrt{T}(\widehat{\sigma}_k - \sigma_k) - \left( \frac{1}{\sqrt{T}}\sum_{i = 1}^T X_iX_{i - k} - \sqrt{T}\sigma_k\right)\Vert_{m/2}\\
\leq \frac{k\vert\sigma_k\vert}{\sqrt{T}}
+ \frac{1}{\sqrt{T}}\Vert\sum_{i = 1}^k (X_iX_{i  - k} - \sigma_k)\Vert_{m/2}\\
\leq \frac{Ck}{\sqrt{T}}\times \frac{1}{(1 + k)^{\alpha_X}} + \frac{C\sqrt{k}\times T^{\alpha_X\beta_X}}{\sqrt{T}}\leq C^\prime T^{\alpha_X\beta_X - \frac{1}{2} + \frac{\beta_X}{2}}\\
\Rightarrow \Vert\  \max_{k\in\mathcal{H}}\vert\sqrt{T}(\widehat{\sigma}_k - \sigma_k) - \left( \frac{1}{\sqrt{T}}\sum_{i = 1}^T X_iX_{i - k} - \sqrt{T}\sigma_k\right)\vert\ \Vert_{m/2}\\
\leq CT^{\alpha_X\beta_X - \frac{1}{2} + \frac{\beta_X}{2} + \frac{2\beta_X}{m}}\
\end{aligned}
\label{eq.deviation}
\end{equation}
From Lemma \ref{lemma.Gaussian} and \eqref{eq.xi}, for any given $\xi>0$, $\exists$ a constant $C_\xi>0$ such that for sufficiently large $T$
\begin{equation}
\begin{aligned}
Prob\left(\max_{j\in\mathcal{H}}\sqrt{T}\vert\widehat{\sigma}_j - \sigma_j\vert\leq x\right)\\
\leq
\xi + Prob\left(\max_{k \in\mathcal{H}}\vert \frac{1}{\sqrt{T}}\sum_{i = 1}^T X_iX_{i - k} - \sqrt{T}\sigma_k\vert\leq x + C_\xi T^{\frac{2\beta_X}{m} + \alpha_X\beta_X - \frac{1}{2} + \frac{\beta_X}{2}}\right)\\
\leq 2\xi + Prob\left(\max_{k\in\mathcal{H}}\vert\xi_k\vert\leq x\right) + CT^{\alpha_X\beta_X}\delta\times (1 + \sqrt{\log(T)} + \sqrt{\vert\log(\delta)\vert})\\
\text{and }
Prob\left(\max_{j\in\mathcal{H}}\sqrt{T}\vert\widehat{\sigma}_j - \sigma_j\vert\leq x\right)\\
\geq
-\xi + Prob\left(\max_{k \in\mathcal{H}}\vert \frac{1}{\sqrt{T}}\sum_{i = 1}^T X_iX_{i - k} - \sqrt{T}\sigma_k\vert\leq x - C_\xi T^{\frac{2\beta_X}{m} + \alpha_X\beta_X - \frac{1}{2} + \frac{\beta_X}{2}}\right)\\
\geq -2\xi + Prob\left(\max_{k\in\mathcal{H}}\vert\xi_k\vert\leq x\right) - CT^{\alpha_X\beta_X}\delta\times (1 + \sqrt{\log(T)} + \sqrt{\vert\log(\delta)\vert})\\
\end{aligned}
\end{equation}
with $\delta = C_\xi T^{\frac{2\beta_X}{m} + \alpha_X\beta_X - \frac{1}{2} + \frac{\beta_X}{2}}$. From \eqref{eq.condH},
$\frac{2\beta_X}{m} + 2\alpha_X\beta_X - \frac{1}{2} + \frac{\beta_X}{2} < 0$, so we prove \eqref{eq.auto_covariance_gaussian}.

Define the real numbers $c_{ij}, i = 1,...,z, j = 0,1,...,d $ such that $c_{i0} = -\frac{\sigma_{h_i}}{\sigma^2_0}$, $\sigma_{ih_i} = \frac{1}{\sigma_0}$, and
$c_{ij} = 0$ for $j\neq 0,h_i$. From \eqref{eq.cov}, we have $Z_{i,h_i} = \sum_{k = 0}^dc_{ik}(X_iX_{i - k} - \sigma_k)$,
\begin{equation}
\begin{aligned}
\vert\sigma_j\vert\leq \frac{C}{(1 + j)^{\alpha_X}}\ \text{for a constant $C$.}\\
\Rightarrow \sum_{j = 0}^d c^2_{ij} = \frac{\sigma^2_{h_i}}{\sigma^4_0} + \frac{1}{\sigma^2_0}\leq C^\prime\ \text{for a constant $C^\prime$}
\label{eq.size_sigma}
\end{aligned}
\end{equation}
From theorem \ref{lemma.Gaussian},
\begin{equation}
\begin{aligned}
\sup_{x\in\mathbf{R}}\vert Prob\left(\max_{k\in\mathcal{I}}\vert\frac{1}{\sqrt{T}}\sum_{i = 1}^T Z_{i, k}\vert\leq x\right)
 -  Prob\left(\max_{k\in\mathcal{I}}\vert\zeta_k\vert\leq x\right)\vert = o(1)
\end{aligned}
\end{equation}
On the other hand, from Lemma \ref{lemma.linear_comb} and \eqref{eq.deviation}
\begin{equation}
\begin{aligned}
\Vert \widehat{\sigma}_j - \sigma_j \Vert_{m/2}\leq \frac{j\vert\sigma_j\vert}{T}
+ \frac{1}{T}\Vert\sum_{i = j + 1}^T(X_iX_{i - j} - \sigma_j)\Vert_{m/2}
\leq \frac{C}{(1 + j)^{\alpha_X}}\times \frac{j}{T} + CT^{\alpha_X\beta_X - \frac{1}{2}}\\
\Rightarrow \Vert\ \max_{j = 0,1,...,d}\vert\widehat{\sigma}_j - \sigma_j\vert \ \Vert_{m / 2}\leq (d + 1)^{2/m}\max_{j = 0,1,...,d}\Vert \widehat{\sigma}_j - \sigma_j\Vert_{m/2}
= O\left(T^{\alpha_X\beta_X + \frac{2\beta_X}{m} - \frac{1}{2}}\right)
\label{eq.maxAUTO}
\end{aligned}
\end{equation}
Therefore, with probability tending to $1$ we have $\vert \frac{\widehat{\sigma}_0 - \sigma_0}{\sigma_0}\vert < 1/2$. Correspondingly,
\begin{equation}
\begin{aligned}
\vert\sqrt{T}(\widehat{\rho}_j - \rho_j)  - \frac{1}{\sqrt{T}}\sum_{i = 1}^T Z_{i, j}\vert
\leq \vert\frac{\sqrt{T}(\widehat{\sigma}_j - \sigma_j)}{\sigma_0} - \frac{1}{\sigma_0\sqrt{T}}\sum_{i = 1}^T(X_iX_{i - j} - \sigma_j)\vert\\
+\vert-\frac{\sqrt{T}\widehat{\sigma}_j(\widehat{\sigma}_0 - \sigma_0)}{\sigma^2_0} + \frac{\sigma_j}{\sigma^2_0\sqrt{T}}\sum_{i = 1}^T(X^2_i - \sigma_0)\vert
+ \sqrt{T}\vert\frac{\widehat{\sigma}_j}{\sigma_0}\vert\times \sum_{j = 2}^\infty\vert \frac{\widehat{\sigma}_0 - \sigma_0}{\sigma_0}\vert^j\\
\leq  \frac{1}{\sigma_0\sqrt{T}}\vert\sum_{i = 1}^j X_iX_{i - j }\vert
+ \vert\frac{\sqrt{T}(\widehat{\sigma}_0 - \sigma_0)}{\sigma^2_0}\vert\times \vert\widehat{\sigma}_j - \sigma_j\vert
+ 2\vert\frac{\sqrt{T}\widehat{\sigma}_j}{\sigma_0}\vert\times \vert \frac{\widehat{\sigma}_0 - \sigma_0}{\sigma_0}\vert^2\\
\Rightarrow \max_{j\in\mathcal{I}}\vert\sqrt{T}(\widehat{\rho}_j - \rho_j)  - \frac{1}{\sqrt{T}}\sum_{i = 1}^T Z_{i, j}\vert\leq
\frac{1}{\sigma_0}\times \max_{j\in\mathcal{I}}\frac{\vert\sum_{i = 1}^j X_iX_{i - j}\vert}{\sqrt{T}}\\
+ \frac{\sqrt{T}}{\sigma^2_0}(\max_{j\in\mathcal{I}}\vert\widehat{\sigma}_j - \sigma_j\vert)^2 + \frac{2\sqrt{T}}{\sigma^3_0}(\max_{j\in\mathcal{I}}\vert\widehat{\sigma}_j - \sigma_j\vert)^3
+ \frac{2\sqrt{T}\max_{j\in\mathcal{I}}\vert\sigma_j\vert}{\sigma^3_0}(\max_{j\in\mathcal{I}}\vert\widehat{\sigma}_j - \sigma_j\vert)^2\\
= O_p\left(T^{2\alpha_X\beta_X + \frac{4\beta_X}{m} - \frac{1}{2}}\right)
\end{aligned}
\end{equation}
Therefore, for any given $\xi>0$, choose sufficiently large constant $C_\xi>0$ and $T$ such that
\begin{equation}
\begin{aligned}
Prob\left(\max_{j\in\mathcal{I}}\sqrt{T}\vert\widehat{\rho}_j - \rho_j\vert\leq x\right)
\leq \xi + Prob\left(\max_{j\in\mathcal{I}}\vert\frac{1}{\sqrt{T}}\sum_{i = 1}^T Z_{i,j}\vert\leq x + C_\xi T^{2\alpha_X\beta_X + \frac{4\beta_X}{m} - \frac{1}{2}}\right)\\
\leq 2\xi + Prob\left(\max_{j\in\mathcal{I}}\vert\zeta_k\vert\leq x\right) + CT^{\alpha_X\beta_X}\delta\times (1 + \sqrt{\log(T)} + \sqrt{\vert\log(\delta)\vert})\\
\text{and } Prob\left(\max_{j\in\mathcal{I}}\sqrt{T}\vert\widehat{\rho}_j - \rho_j\vert\leq x\right)\geq - \xi
+ Prob\left(\max_{j\in\mathcal{I}}\vert\frac{1}{\sqrt{T}}\sum_{i = 1}^T Z_{i,j}\vert\leq x - C_\xi T^{2\alpha_X\beta_X + \frac{4\beta_X}{m} - \frac{1}{2}}\right)\\
\geq -2\xi  + Prob\left(\max_{j\in\mathcal{I}}\vert\zeta_k\vert\leq x\right) - CT^{\alpha_X\beta_X}\delta\times (1 + \sqrt{\log(T)} + \sqrt{\vert\log(\delta)\vert})\\
\end{aligned}
\end{equation}
with $\delta = C_\xi T^{2\alpha_X\beta_X + \frac{4\beta_X}{m} - \frac{1}{2}}$. Form \eqref{eq.condH},
$3\alpha_X\beta_X + \frac{4\beta_X}{m} - \frac{1}{2} < 0$,  and we prove \eqref{eq.rho_central}.
\end{proof}

\section{Proof of Theorem \ref{theorem.Gaussian_ARCOEF}}
\begin{proof}[Proof of Theorem \ref{theorem.Gaussian_ARCOEF}]
According to \eqref{eq.cov}, $\vert\sigma_{j}\vert\leq C$ for a constant $C$. For $p = O(1)$ and the smallest eigenvalue of $\Sigma$ is greater than $c > 0$,
\begin{equation}
\begin{aligned}
\sum_{j = 0}^p b^2_{ij}\leq (p+1)\max_{j = 0,...,p} (\mathbf{e}_i^T\mathbf{b}_j)^2
\leq 2\max_{j = 1,...,p}(\mathbf{e}_i^T\Sigma^{-1}\mathbf{e}_j)^2 + 2\max_{j = 0,...,p - 1}(\mathbf{e}_i^T\Sigma^{-1}T_j\Sigma^{-1}\gamma)^2\\
\leq C\ \text{for a constant $C$}
\end{aligned}
\label{eq.conditionB}
\end{equation}
From theorem \ref{lemma.Gaussian}, we have
\begin{equation}
\sup_{x\in\mathbf{R}}\vert Prob\left(\max_{j = 1,...,p}\vert\frac{1}{\sqrt{T}}\sum_{i = 1}^T Z_{i, j }\vert\leq x\right) - Prob\left(\max_{j = 1,...,p}\vert\xi_j\vert\leq x\right)\vert
= o(1)
\end{equation}
From section 0.9.7 in \cite{MR2978290} and Lemma \ref{lemma.linear_comb}
\begin{equation}
\begin{aligned}
\vert\widehat{\Sigma} - \Sigma\vert_2\leq 2\sum_{k  = 0}^{p - 1}\vert\widehat{\sigma}_k - \sigma_k\vert
\leq \frac{2}{T}\sum_{k = 0}^{p - 1} k\vert\sigma_k\vert + \frac{2}{T}\sum_{k = 0}^{p - 1}\vert\sum_{i = k + 1}^T(X_iX_{i - k} - \sigma_k)\vert
= O_p(1/\sqrt{T})\\
\text{Similarly}\  \vert\widehat{\gamma} - \gamma\vert_2\leq \sqrt{p}\times \max_{k = 1,...,p}\vert\widehat{\sigma}_k - \sigma_k\vert = O_p(1/\sqrt{T})
\end{aligned}
\label{eq.diff_Covariance_Matrix}
\end{equation}
so with probability tending to $1$ $\widehat{\Sigma}$ is non-singular and $\vert\Sigma^{-1}\vert_2\times \vert\widehat{\Sigma} - \Sigma\vert_2 < 1/2$. From corollary 5.6.16 in \cite{MR2978290},
\begin{equation}
\begin{aligned}
\vert\widehat{\Sigma}^{\dagger} - \Sigma^{-1}  + \Sigma^{-1}(\widehat{\Sigma} - \Sigma)\Sigma^{-1}\vert_2
\leq \left(\sum_{k = 2}^\infty \vert\Sigma^{-1}\vert_2^k\times \vert\widehat{\Sigma} - \Sigma\vert_2^k \right)\times \vert\Sigma^{-1}\vert_2\leq C\vert\widehat{\Sigma} - \Sigma\vert^2_2\\
\text{and } \vert \widehat{\Sigma}^{\dagger} - \Sigma^{-1}  \vert_2\leq \vert \Sigma^{-1}(\widehat{\Sigma} - \Sigma)\Sigma^{-1}\vert_2 +C \vert \widehat{\Sigma} - \Sigma\vert_2^2 = O_p(1/\sqrt{T})
\end{aligned}
\label{eq.diff_Covariance_Inverse}
\end{equation}
Recall $a = (a_1,...,a_p)^T = \Sigma^{-1}\gamma$, then
\begin{equation}
\begin{aligned}
\vert (\widehat{a} - a) -  \left(\Sigma^{-1}(\widehat{\gamma} - \gamma) - \Sigma^{-1}(\widehat{\Sigma} - \Sigma) \Sigma^{-1}\gamma\right)\vert_\infty\\
\leq \vert\widehat{\Sigma}^{-1} - \Sigma^{-1}\vert_2\times \vert\widehat{\gamma} - \gamma\vert_2
 + \vert\widehat{\Sigma}^{\dagger} - \Sigma^{-1}  + \Sigma^{-1}(\widehat{\Sigma} - \Sigma)\Sigma^{-1}\vert_2\times \vert\gamma\vert_2 = O_p(1/T)
\end{aligned}
\end{equation}
Since
\begin{equation}
\begin{aligned}
\Sigma^{-1}(\widehat{\gamma} - \gamma) - \Sigma^{-1}(\widehat{\Sigma} - \Sigma)\Sigma^{-1}\gamma\\
= \sum_{k = 1}^p (\widehat{\sigma}_k - \sigma_k)\Sigma^{-1}\mathbf{e}_k
- \sum_{k = 0}^{p - 1}(\widehat{\sigma}_k - \sigma_k)\Sigma^{-1}T_k\Sigma^{-1}\gamma = \sum_{k = 0}^p (\widehat{\sigma}_k - \sigma_k)\mathbf{b}_k\\
\text{and } \vert\frac{1}{T}\sum_{i = 1}^T Z_{i,j} - \sum_{k = 0}^p b_{jk}(\widehat{\sigma}_k - \sigma_k)\vert
= \vert\frac{1}{T}\sum_{k = 0}^p b_{jk}\left(\sum_{i = 1}^k X_{i}X_{i - k}\right)\vert  =O_p(1/T)
\end{aligned}
\end{equation}
From Lemma \ref{lemma.normal_property}, for any given $\delta>0$ and sufficiently large $T$,
\begin{equation}
\begin{aligned}
Prob\left(\max_{j = 1,...,p}\vert\sqrt{T}(\widehat{a}_j - a_j)\vert\leq x\right)
\leq \delta + Prob\left(\max_{j = 1,...,p}\vert\frac{1}{\sqrt{T}}\sum_{i = 1}^TZ_{i, j}\vert\leq x + \delta\right)\\
\leq 2\delta + Prob\left(\max_{j = 1,...,p}\vert\xi_j\vert\leq x\right) + C\delta\times(1 + \sqrt{\log(p)} + \sqrt{\vert\log(\delta)\vert})\\
\text{and } Prob\left(\max_{j = 1,...,p}\vert\sqrt{T}(\widehat{a}_j - a_j)\vert\leq x\right)\geq -\delta
+ Prob\left(\max_{j = 1,...,p}\vert\frac{1}{\sqrt{T}}\sum_{i = 1}^TZ_{i, j}\vert\leq x - \delta\right)\\
\geq -2\delta + Prob\left(\max_{j = 1,...,p}\vert\xi_j\vert\leq x\right)  - C\delta\times(1 + \sqrt{\log(p)} + \sqrt{\vert\log(\delta)\vert})
\end{aligned}
\end{equation}
and we prove \eqref{eq.auto_coeff}.
\end{proof}

\section{Proof of results of Section \ref{section.bootstrap}}
\begin{proof}[Proof of Lemma \ref{lemma.consistent_variance}]
(i). From Lemma \ref{lemma.covariance}
\begin{equation}
\begin{aligned}
\max_{j_1, j_2 \in\mathcal{H}}\vert \frac{1}{T}\sum_{i_1 = 1}^T\sum_{i_2 = 1}^TK\left(\frac{i_1 - i_2}{k_T}\right)(X_{i_1}X_{i_1 - j_1} - \sigma_{j_1})(X_{i_2}X_{i_2 - j_2}  - \sigma_{j_2})\\
- \frac{1}{T}\sum_{i_1 = 1}^T\sum_{i_2 = 1}^T\mathbf{E} (X_{i_1}X_{i_1 - j_1} - \sigma_{j_1}) (X_{i_2}X_{i_2 - j_2} - \sigma_{j_2})
\vert\\
= O_p\left(v_T\times T^{\alpha_X\beta_X} + k_T\times T^{\frac{8\beta_X}{m} + \alpha_X\beta_X - \frac{1}{2}}\right)
\end{aligned}
\end{equation}
On the other hand,
\begin{equation}
\begin{aligned}
\vert \frac{1}{T}\sum_{i_1 = j_1 + 1}^T\sum_{i_2 = j_2 + 1}^TK\left(\frac{i_1 - i_2}{k_T}\right)(X_{i_1}X_{i_1 - j_1} - \widehat{\sigma}_{j_1})(X_{i_2}X_{i_2 - j_2}  - \widehat{\sigma}_{j_2})\\
-  \frac{1}{T}\sum_{i_1 = 1}^T\sum_{i_2 = 1}^TK\left(\frac{i_1 - i_2}{k_T}\right)(X_{i_1}X_{i_1 - j_1} - \sigma_{j_1})(X_{i_2}X_{i_2 - j_2}  - \sigma_{j_2})
\vert\\
\leq \frac{\vert \widehat{\sigma}_{j_2} - \sigma_{j_2}\vert}{T}\vert\sum_{i_1 = j_1 + 1}^T\sum_{i_2 = j_2 + 1}^T K\left(\frac{i_1 - i_2}{k_T}\right)(X_{i_1}X_{i_1 - j_1} - \sigma_{j_1})\vert\\
+ \frac{\vert\widehat{\sigma}_{j_1} - \sigma_{j_1}\vert}{T}\vert\sum_{i_1 = j_1 + 1}^T\sum_{i_2 = j_2 + 1}^T K\left(\frac{i_1 - i_2}{k_T}\right)(X_{i_2}X_{i_2 - j_2} - \sigma_{j_2})\vert \\
+ \frac{\vert(\widehat{\sigma}_{j_1} - \sigma_{j_1})\times (\widehat{\sigma}_{j_2} - \sigma_{j_2})\vert}{T}\sum_{i_1 = j_1 + 1}^T\sum_{i_2 = j_2 + 1}^T K\left(\frac{i_1 - i_2}{k_T}\right)\\
+ \frac{1}{T}\vert\sum_{i_1 = 1}^{j_1} \sum_{i_2 = 1}^{j_2} K\left(\frac{i_1 - i_2}{k_T}\right)
(X_{i_1}X_{i_1 - j_1} - \sigma_{j_1})(X_{i_2}X_{i_2 - j_2}  - \sigma_{j_2})\vert\\
+ \frac{1}{T}\vert\sum_{i_1 = 1}^{j_1} \sum_{i_2 = j_2 + 1}^T K\left(\frac{i_1 - i_2}{k_T}\right)
(X_{i_1}X_{i_1 - j_1} - \sigma_{j_1})(X_{i_2}X_{i_2 - j_2}  - \sigma_{j_2})\vert\\
+ \frac{1}{T}\vert\sum_{i_1 = j_1 + 1}^T\sum_{i_2 = 1}^{j_2} K\left(\frac{i_1 - i_2}{k_T}\right)
(X_{i_1}X_{i_1 - j_1} - \sigma_{j_1})(X_{i_2}X_{i_2 - j_2}  - \sigma_{j_2})\vert
\end{aligned}
\end{equation}
For $X_iX_{i - j} - \sigma_j$ are $(\frac{m}{2}, \alpha_X - 1, \alpha_X\beta_X)$ - medium range dependent random variables,
from Lemma \ref{lemma.linear_comb},
\begin{equation}
\begin{aligned}
\Vert\ \sum_{i_1 = j_1 + 1}^T\left(\sum_{i_2 = j_2 + 1}^T K\left(\frac{i_1 - i_2}{k_T}\right)\right)(X_{i_1}X_{i_1 - j_1} - \sigma_{j_1})\ \Vert_{m/2}\\
\leq CT^{\alpha_X\beta_X}\times \sqrt{\sum_{i_1 = j_1 + 1}^T\left(\sum_{i_2 = j_2 + 1}^T K\left(\frac{i_1 - i_2}{k_T}\right)\right)^2}
\end{aligned}
\end{equation}
For
\begin{equation}
\begin{aligned}
\sum_{i_2 = 1}^T K\left(\frac{i_1 - i_2}{k_T}\right)\leq 2\sum_{s = 0}^\infty K\left(\frac{s}{k_T}\right)\leq 1 + \int_{[0,\infty)} K(\frac{x}{k_T}) dx = O(k_T)
\end{aligned}
\end{equation}
we have
\begin{equation}
\begin{aligned}
\Vert\ \sum_{i_1 = j_1 + 1}^T\left(\sum_{i_2 = j_2 + 1}^T K\left(\frac{i_1 - i_2}{k_T}\right)\right)(X_{i_1}X_{i_1 - j_1} - \sigma_{j_1})\ \Vert_{m/2}
\leq Ck_T\times T^{\alpha_X\beta_X+\frac{1}{2}}
\end{aligned}
\end{equation}
so from \eqref{eq.maxAUTO}
\begin{equation}
\begin{aligned}
\max_{j_1,j_2\in\mathcal{H}}
\frac{\vert \widehat{\sigma}_{j_2} - \sigma_{j_2}\vert}{T}\vert\sum_{i_1 = j_1 + 1}^T
\sum_{i_2 = j_2 + 1}^T K\left(\frac{i_1 - i_2}{k_T}\right)(X_{i_1}X_{i_1 - j_1} - \sigma_{j_1})\vert\\
\leq \frac{\max_{j\in\mathcal{H}}\vert \widehat{\sigma}_j - \sigma_j\vert}{T}
\times \max_{j_1, j_2\in\mathcal{H}}\vert \sum_{i_1 = j_1 + 1}^T\left(\sum_{i_2 = j_2 + 1}^T K\left(\frac{i_1 - i_2}{k_T}\right)\right)(X_{i_1}X_{i_1 - j_1} - \sigma_{j_1})\vert\\
= O_p\left(k_T\times T^{2\alpha_X\beta_X + \frac{6\beta_X}{m} - 1}   \right)
\end{aligned}
\end{equation}
Also from Section 0.9.7 in \cite{MR2978290}
\begin{equation}
\begin{aligned}
\sum_{i_1 = j_1 + 1}^T\sum_{i_2 = j_2 + 1}^T K\left(\frac{i_1 - i_2}{k_T}\right)\leq Ck_T\times T\\
\Rightarrow \max_{j_1,j_2\in\mathcal{H}}\frac{\vert(\widehat{\sigma}_{j_1} - \sigma_{j_1})\times (\widehat{\sigma}_{j_2} - \sigma_{j_2})\vert}{T}
\sum_{i_1 = j_1 + 1}^T\sum_{i_2 = j_2 + 1}^T K\left(\frac{i_1 - i_2}{k_T}\right)\\
= O_p\left(k_T\times T^{2\alpha_X\beta_X + \frac{4\beta_X}{m} - 1}\right)
\end{aligned}
\end{equation}
For
\begin{equation}
\begin{aligned}
\Vert\ \sum_{i_1 = 1}^{j_1} \sum_{i_2 = j_2 + 1}^T K\left(\frac{i_1 - i_2}{k_T}\right)
(X_{i_1}X_{i_1 - j_1} - \sigma_{j_1})(X_{i_2}X_{i_2 - j_2}  - \sigma_{j_2})\ \Vert_{m/4}\\
\leq \sum_{i_1 = 1}^{j_1}\sum_{i_2 = j_2 + 1}^T K\left(\frac{i_1 - i_2}{k_T}\right)\times \Vert X_{i_1}X_{i_1 - j_1} - \sigma_{j_1}\Vert_{m/2}\times \Vert X_{i_2}X_{i_2 - j_2}  - \sigma_{j_2}\Vert_{m/2}\\
\leq Ck_T\times \sqrt{T}\times \sqrt{j_1}\leq C^\prime k_T\times T^{\frac{1}{2} + \frac{\beta_X}{2}}\\
\Rightarrow \max_{j_1, j_2 \in\mathcal{H}}\frac{1}{T}\vert\sum_{i_1 = 1}^{j_1} \sum_{i_2 = j_2 + 1}^T K\left(\frac{i_1 - i_2}{k_T}\right)
(X_{i_1}X_{i_1 - j_1} - \sigma_{j_1})(X_{i_2}X_{i_2 - j_2}  - \sigma_{j_2})\vert\\
= O_p\left(k_T\times T^{\frac{8\beta_X}{m} + \frac{\beta_X}{2} - \frac{1}{2}}\right)
\end{aligned}
\end{equation}
and
\begin{equation}
\begin{aligned}
\Vert\ \sum_{i_1 = 1}^{j_1} \sum_{i_2 = 1}^{j_2} K\left(\frac{i_1 - i_2}{k_T}\right)
(X_{i_1}X_{i_1 - j_1} - \sigma_{j_1})(X_{i_2}X_{i_2 - j_2}  - \sigma_{j_2})\ \Vert_{m/4}\\
\leq \sum_{i_1 = 1}^{j_1}\sum_{i_2 = 1}^{j_2} K\left(\frac{i_1 - i_2}{k_T}\right)\times \Vert X_{i_1}X_{i_1 - j_1} - \sigma_{j_1}\Vert_{m/2}
\times\Vert X_{i_2}X_{i_2 - j_2}  - \sigma_{j_2}\Vert_{m/2}\\
\leq Ck_T\sqrt{j_1\times j_2}\leq C^\prime k_T\times T^{\beta_X}\\
\Rightarrow \max_{j_1, j_2 \in\mathcal{H}}\frac{1}{T}\vert\sum_{i_1 = 1}^{j_1} \sum_{i_2 = j_2 + 1}^T K\left(\frac{i_1 - i_2}{k_T}\right)
(X_{i_1}X_{i_1 - j_1} - \sigma_{j_1})(X_{i_2}X_{i_2 - j_2}  - \sigma_{j_2})\vert\\
= O_p\left(k_T\times T^{\frac{8\beta_X}{m} + \beta_X - 1}\right)
\end{aligned}
\end{equation}
From \eqref{eq.condH}, we prove \eqref{eq.sigmaX}.

(ii). From Lemma \ref{lemma.covariance},
\begin{equation}
\begin{aligned}
\max_{j_1, j_2\in\mathcal{I}}\vert \frac{1}{T}\sum_{i_1 = 1}^T\sum_{i_2 = 1}^TK\left(\frac{i_1 - i_2}{k_T}\right)Z_{i_1,j_1}Z_{i_2, j_2}
 - \frac{1}{T}\sum_{i_1 = 1}^T\sum_{i_2 = 1}^T\mathbf{E} Z_{i_1,j_1}Z_{i_2,j_2}
\vert\\
= O_p\left(v_T\times T^{\alpha_X\beta_X} + k_T\times T^{\frac{8\beta_X}{m} + \alpha_X\beta_X - \frac{1}{2}}\right)
\end{aligned}
\end{equation}
On the other hand,
\begin{equation}
\begin{aligned}
\vert \sum_{i_1 = j_1 + 1}^T\sum_{i_2 = j_2 + 1}^TK\left(\frac{i_1 - i_2}{k_T}\right)\widehat{Z}_{i_1, j_1}\widehat{Z}_{i_2, j_2} - \sum_{i_1 = 1}^T\sum_{i_2 = 1}^TK\left(\frac{i_1 - i_2}{k_T}\right)Z_{i_1,j_1}Z_{i_2, j_2}
\vert\\
\leq \vert\sum_{i_1 = 1}^{j_1}\sum_{i_2 = 1}^{j_2} K\left(\frac{i_1 - i_2}{k_T}\right)Z_{i_1, j_1}Z_{i_2, j_2}\vert
+ \vert\sum_{i_1 = j_1 + 1}^T\sum_{i_2 = 1}^{j_2}K\left(\frac{i_1 - i_2}{k_T}\right)Z_{i_1, j_1}Z_{i_2, j_2}\vert\\
+ \vert\sum_{i_1 = 1}^{j_1}\sum_{i_2 = j_2 + 1}^T K\left(\frac{i_1 - i_2}{k_T}\right)Z_{i_1, j_1}Z_{i_2, j_2}\vert\\
+ \vert\sum_{i_1 = j_1 + 1}^T\sum_{i_2 = j_2 + 1}^T K\left(\frac{i_1 - i_2}{k_T}\right)Z_{i_1, j_1}(\widehat{Z}_{i_2, j_2} - Z_{i_2, j_2})\vert\\
+ \vert\sum_{i_1 = j_1 + 1}^T\sum_{i_2 = j_2 + 1}^T K\left(\frac{i_1 - i_2}{k_T}\right)Z_{i_2, j_2}(\widehat{Z}_{i_1, j_1} - Z_{i_1, j_1})\vert\\
+ \vert\sum_{i_1 = j_1 + 1}^T\sum_{i_2 = j_2 + 1}^T K\left(\frac{i_1 - i_2}{k_T}\right)(\widehat{Z}_{i_2, j_2} - Z_{i_2, j_2})(\widehat{Z}_{i_1, j_1} - Z_{i_1, j_1})\vert
\end{aligned}
\label{eq.decomposition_K}
\end{equation}
For
\begin{equation}
\begin{aligned}
\vert\widehat{Z}_{i, j} - Z_{i, j}\vert\leq \frac{\vert\widehat{\sigma}_j\vert}{\widehat{\sigma}_0^2}\times \vert\widehat{\sigma}_0 - \sigma_0\vert
+ \frac{1}{\widehat{\sigma}_0}\vert\widehat{\sigma}_j - \sigma_j\vert
+ \vert \frac{\widehat{\sigma}_j}{\widehat{\sigma}^2_0} - \frac{\sigma_j}{\sigma^2_0}\vert \times \vert X^2_i - \sigma_0\vert\\
+ \vert \frac{1}{\widehat{\sigma}_0} - \frac{1}{\sigma_0}\vert \times \vert X_iX_{i - j} - \sigma_j\vert
\end{aligned}
\end{equation}
From Section 0.9.7 in \cite{MR2978290}
\begin{equation}
\begin{aligned}
\max_{j_1, j_2\in\mathcal{I}}\vert\sum_{i_1 = j_1 + 1}^T\sum_{i_2 = j_2 + 1}^T K\left(\frac{i_1 - i_2}{k_T}\right)Z_{i_1, j_1}(\widehat{Z}_{i_2, j_2} - Z_{i_2, j_2})\vert\\
\leq 2\sum_{s = 0}^\infty K\left(\frac{s}{k_T}\right)\times \max_{j\in\mathcal{I}}\sqrt{\sum_{i = j + 1}^T Z^2_{i, j}}
\times \max_{j\in\mathcal{I}}\sqrt{\sum_{i = j + 1}^T (\widehat{Z}_{i,j} - Z_{i, j})^2}\\
\leq Ck_T\times \max_{j\in\mathcal{I}}\sqrt{\sum_{i = 1}^T Z^2_{i, j}}
\times\max_{j\in\mathcal{I}}[\sqrt{T} \frac{\vert\widehat{\sigma}_j\vert\times \vert\widehat{\sigma}_0 - \sigma_0\vert}{\widehat{\sigma}^2_0}
+ \sqrt{T}\frac{\vert\widehat{\sigma}_j - \sigma_j\vert}{\widehat{\sigma}_0}\\
 + \vert \frac{\widehat{\sigma}_j}{\widehat{\sigma}^2_0} - \frac{\sigma_j}{\sigma^2_0}\vert \times \sqrt{\sum_{i = 1}^T (X^2_i - \sigma_0)^2}
+  \vert\frac{1}{\widehat{\sigma}_0} - \frac{1}{\sigma_0}\vert\times \sqrt{\sum_{i = 1}^T (X_iX_{i - j} - \sigma_j)^2}]
\end{aligned}
\label{eq.decomposite_K}
\end{equation}
For $\Vert\sum_{i = 1}^T Z^2_{i,j}\Vert_{m/4}\leq \sum_{i = 1}^T \Vert Z_{i,j}\Vert_{m/2}^2\leq CT$ for a constant $C$, we have
$\max_{j\in\mathcal{I}}\sqrt{\sum_{i = 1}^TZ^2_{i, j}} = O_p\left(T^{\frac{4\beta_X}{m} + \frac{1}{2}}\right)$. Similarly,
$\Vert\sum_{i = 1}^T (X_i^2 - \sigma_0)^2\Vert_{m/4}\leq \sum_{i = 1}^T \Vert X^2_i - \sigma_0\Vert_{m/2}^2\leq CT$ implies
$ \sqrt{\sum_{i = 1}^T (X^2_i - \sigma_0)^2} = O_p(\sqrt{T})$;
$\Vert\sum_{i  = 1}^T (X_iX_{i - j} - \sigma_j)^2\Vert_{m/4}\leq \sum_{i = 1}^T \Vert X_iX_{i  -j} - \sigma_j\Vert_{m/2}^2\leq CT$
implies $\max_{j\in\mathcal{I}}\sqrt{\sum_{i = 1}^T (X_iX_{i - j} - \sigma_j)^2} = O_p\left(T^{\frac{4\beta_X}{m} + \frac{1}{2}}\right)$.
From \eqref{eq.minVARCOEF}, \eqref{eq.size_sigma} and \eqref{eq.maxAUTO},  $\widehat{\sigma}_0 > c/2$ with probability tending to $1$  , and
\begin{equation}
\begin{aligned}
\max_{j\in\mathcal{I}}\vert\widehat{\sigma}_j\vert\times \frac{\vert\widehat{\sigma}_0 - \sigma_0\vert}{\widehat{\sigma}^2_0}
\leq \frac{4\vert \widehat{\sigma}_0 - \sigma_0\vert}{c^2}\times \left(\max_{j\in\mathcal{I}}\vert\sigma_j\vert
 + \max_{j\in\mathcal{I}}\vert\widehat{\sigma}_j - \sigma_j\vert\right) = O_p\left(T^{\alpha_X\beta_X + \frac{2\beta_X}{m} - \frac{1}{2}}\right)\\
\max_{j\in\mathcal{I}}\frac{\vert\widehat{\sigma}_j - \sigma_j\vert}{\widehat{\sigma}_0}\leq \frac{2}{c}\max_{j\in\mathcal{I}}
\vert\widehat{\sigma}_j - \sigma_j\vert = O_p\left(T^{\alpha_X\beta_X + \frac{2\beta_X}{m} - \frac{1}{2}}\right)\\
\max_{j\in\mathcal{I}}\vert \frac{\widehat{\sigma}_j}{\widehat{\sigma}^2_0} - \frac{\sigma_j}{\sigma^2_0}\vert
\leq \frac{4}{c^2}\max_{j\in\mathcal{I}}\vert\widehat{\sigma}_j - \sigma_j\vert + \max_{j\in\mathcal{I}}\vert\sigma_j\vert
\times \frac{4\vert\sigma^2_0 - \widehat{\sigma}^2_0\vert}{\sigma^2_0\times c^2}
= O_p\left(T^{\alpha_X\beta_X + \frac{2\beta_X}{m} - \frac{1}{2}}\right)\\
\vert\frac{1}{\widehat{\sigma}_0} - \frac{1}{\sigma_0}\vert\leq \frac{2\vert \widehat{\sigma}_0 - \sigma_0\vert}{c\sigma_0}
 = O_p\left(T^{\alpha_X\beta_X  - \frac{1}{2}}\right)\\
\end{aligned}
\end{equation}
we have
\begin{equation}
\begin{aligned}
\max_{j\in\mathcal{I}}\sqrt{\sum_{i = j + 1}^T (\widehat{Z}_{i,j} - Z_{i, j})^2} = O_p\left(T^{\alpha_X\beta_X + \frac{
4\beta_X}{m}}\right)
\end{aligned}
\end{equation}
and
\begin{equation}
\begin{aligned}
\max_{j_1, j_2\in\mathcal{I}}\vert\sum_{i_1 = j_1 + 1}^T\sum_{i_2 = j_2 + 1}^T K\left(\frac{i_1 - i_2}{k_T}\right)Z_{i_1, j_1}(\widehat{Z}_{i_2, j_2} - Z_{i_2, j_2})\vert\\
=O_p\left(k_T\times T^{\alpha_X\beta_X + \frac{8\beta_X}{m} + \frac{1}{2}}\right)
\end{aligned}
\end{equation}
Similar to \eqref{eq.decomposite_K},
\begin{equation}
\begin{aligned}
\max_{j_1,j_2\in\mathcal{I}}\vert\sum_{i_1 = j_1 + 1}^T\sum_{i_2 = j_2 + 1}^T K\left(\frac{i_1 - i_2}{k_T}\right)(\widehat{Z}_{i_2, j_2} - Z_{i_2, j_2})(\widehat{Z}_{i_1, j_1} - Z_{i_1, j_1})\vert\\
\leq C\sum_{s = 0}^\infty K\left(\frac{s}{k_T}\right)\times \max_{j\in\mathcal{I}}\left(\sum_{i = 1}^T (\widehat{Z}_{i, j} - Z_{i,j})^2\right)
= O_p\left(k_T\times T^{2\alpha_X\beta_X + \frac{8\beta_X}{m}}\right)
\end{aligned}
\label{eq.eqsquare}
\end{equation}
For
\begin{equation}
\begin{aligned}
\Vert\ \sum_{i_1 = 1}^{j_1}\sum_{i_2 = j_2 + 1}^T K\left(\frac{i_1 - i_2}{k_T}\right)Z_{i_1, j_1}Z_{i_2, j_2}\ \Vert_{m/4}\\
\leq \sum_{i_1 = 1}^{j_1}\sum_{i_2 = j_2 + 1}^T K\left(\frac{i_1 - i_2}{k_T}\right)\Vert Z_{i_1, j_1}\Vert_{m/2}\times \Vert Z_{i_2, j_2}\Vert_{m/2}
\leq Ck_T\times \sqrt{j_1}\times \sqrt{T}\\
\Rightarrow \max_{j_1,j_2\in\mathcal{I}}\vert\sum_{i_1 = 1}^{j_1}\sum_{i_2 = j_2 + 1}^T K\left(\frac{i_1 - i_2}{k_T}\right)
Z_{i_1, j_1}Z_{i_2, j_2}\vert = O_p\left(k_T\times T^{\frac{8\beta_X}{m} + \frac{1}{2} + \frac{\beta_X}{2}}\right)
\end{aligned}
\label{eq.upp}
\end{equation}
and
\begin{equation}
\begin{aligned}
\Vert\ \sum_{i_1 = 1}^{j_1}\sum_{i_2 = 1}^{j_2} K\left(\frac{i_1 - i_2}{k_T}\right)Z_{i_1, j_1}Z_{i_2, j_2}\ \Vert_{m/4}\\
\leq \sum_{i_1 =  1}^{j_1}\sum_{i_2  = 1}^{j_2}K\left(\frac{i_1 - i_2}{k_T}\right)\Vert Z_{i_1, j_1}\Vert_{m/2}\times \Vert Z_{i_2, j_2}\Vert_{m/2}
\leq Ck_T\times \sqrt{j_1j_2}\\
\Rightarrow \max_{j_1,j_2\in\mathcal{I}}\vert \sum_{i_1 = 1}^{j_1}\sum_{i_2 = 1}^{j_2} K\left(\frac{i_1 - i_2}{k_T}\right)Z_{i_1, j_1}Z_{i_2, j_2}\vert
= O_p\left(k_T\times T^{\frac{8\beta_X}{m} + \beta_X}\right)
\end{aligned}
\label{eq.upp2}
\end{equation}
so
\begin{equation}
\begin{aligned}
\max_{j_1,j_2\in\mathcal{I}}\vert \frac{1}{T}\sum_{i_1 = j_1 + 1}^T\sum_{i_2 = j_2 + 1}^TK\left(\frac{i_1 - i_2}{k_T}\right)\widehat{Z}_{i_1, j_1}\widehat{Z}_{i_2, j_2} - \frac{1}{T}\sum_{i_1 = 1}^T\sum_{i_2 = 1}^TK\left(\frac{i_1 - i_2}{k_T}\right)Z_{i_1,j_1}Z_{i_2, j_2}
\vert\\
= O_p\left(k_T\times T^{\frac{8\beta_X}{m} + \alpha_X\beta_X - \frac{1}{2}}\right)
\end{aligned}
\end{equation}
and we prove \eqref{eq.sigmaCorr}.

(iii). Define $Z_{i,j}$ as in \eqref{eq.def_ARcoef}. From Lemma \ref{lemma.covariance} and \eqref{eq.conditionB},
\begin{equation}
\begin{aligned}
\max_{j_1, j_2 = 1,...,d}\vert \frac{1}{T}\sum_{i_1 = 1}^T\sum_{i_2 = 1}^T Z_{i_1, j_1}Z_{i_2, j_2}K\left(\frac{i_1 - i_2}{k_T}\right)
 - \frac{1}{T}\sum_{i_1 = 1}^T\sum_{i_2 = 1}^T\mathbf{E}Z_{i_1, j_1}Z_{i_2, j_2}\vert\\
= O_p(v_T + k_T\times T^{ - \frac{1}{2}})
\end{aligned}
\end{equation}
From \eqref{eq.size_sigma} and \eqref{eq.maxAUTO}, we have $\max_{j = 0,...,d}\vert\widehat{\sigma}_j - \sigma_j\vert = O_p(1/\sqrt{T})$. From
\eqref{eq.diff_Covariance_Matrix} and \eqref{eq.diff_Covariance_Inverse}, we have  $\vert\widehat{\Sigma} - \Sigma\vert_2 = O_p(1/\sqrt{T})$,
$\vert\widehat{\gamma} - \gamma\vert_2 = O_p(1/\sqrt{T})$ and $\vert \widehat{\Sigma}^\dagger - \Sigma^{-1}\vert_2 = O_p(1/\sqrt{T})$.
Correspondingly for $j = 1,...,d - 1$, set $\Delta = \widehat{\Sigma}^\dagger - \Sigma^{-1}$ and $\delta = \widehat{\gamma} - \gamma$, we have
\begin{equation}
\begin{aligned}
\vert\widehat{\mathbf{b}}_j - \mathbf{b}_j\vert_2\leq \vert\Delta\vert_2 + \vert \Delta T_j \Sigma^{-1}\gamma\vert_2 + \vert \Sigma^{-1}T_j\Delta \gamma\vert_2 + \vert \Delta T_j \Delta \gamma \vert_2\\
+ \vert \Sigma^{-1} T_j \Sigma^{-1}\delta\vert_2 + \vert \Delta T_j \Sigma^{-1}\delta\vert_2 + \vert \Sigma^{-1} T_j \Delta \delta\vert_2
+ \vert \Delta T_j \Delta \delta\vert_2 = O_p(1/\sqrt{T})\\
\text{and } \vert\widehat{\mathbf{b}}_0 - \mathbf{b}_0\vert_2\leq \vert\Delta \Sigma^{-1}\gamma\vert_2 + \vert\Sigma^{-1}\Delta\gamma\vert_2 + \vert \Delta^2\gamma\vert_2 + \vert\Sigma^{-2}\delta\vert_2 + \vert\Delta\Sigma^{-1}\delta\vert_2\\
+ \vert\Sigma^{-1}\Delta\delta\vert_2 + \vert \Delta^2\delta\vert_2
=O_p(1/\sqrt{T})\\
\vert\widehat{\mathbf{b}}_d - \mathbf{b}_d \vert_2\leq \vert\Delta\vert_2 = O_p(1/\sqrt{T})
\end{aligned}
\label{eq.delta_beta}
\end{equation}
Since
\begin{equation}
\begin{aligned}
\vert\widehat{Z}_{i,j} - Z_{i,j}\vert\leq \sum_{k = 0}^d\vert \widehat{b}_{jk}\vert\times \vert \widehat{\sigma}_k - \sigma_k\vert
+\sum_{k = 0}^d \vert\widehat{b}_{jk} - b_{jk}\vert\times \vert X_iX_{i - k} - \sigma_k\vert\\
\leq \max_{k = 0,...,d}\vert \widehat{b}_{jk}\vert\times \sum_{k = 0}^d\vert \widehat{\sigma}_k - \sigma_k\vert + \max_{k = 0,...,d}\vert \widehat{b}_{jk} - b_{jk}\vert\sum_{k = 0}^d \vert X_iX_{i - k} - \sigma_k\vert\\
\leq \max_{k = 0,...,d}\vert\widehat{\mathbf{b}}_k\vert_2\times \sum_{k = 0}^d\vert \widehat{\sigma}_k - \sigma_k\vert
+ \max_{k = 0,...,d}\vert \widehat{\mathbf{b}}_{k} - \mathbf{b}_{k}\vert_2\sum_{k = 0}^d \vert X_iX_{i - k} - \sigma_k\vert
\end{aligned}
\end{equation}
From \eqref{eq.decomposition_K}, \eqref{eq.decomposite_K} and \eqref{eq.eqsquare}, notice that
$\Vert\sum_{i = 1}^T Z_{i,j}^2\Vert_{m/4}\leq \sum_{i = 1}^T \Vert Z_{i,j}\Vert_{m/2}^2\leq CT$, $\Vert \sum_{i = 1}^T(X_iX_{i - k} - \sigma_k) ^2\Vert_{m/4}\leq \sum_{i = 1}^T \Vert X_iX_{i - k} - \sigma_k\Vert_{m/2}^2\leq CT$, and
\begin{equation}
\begin{aligned}
\sum_{i = 1}^T \vert\widehat{Z}_{i,j} - Z_{i,j}\vert^2\leq
2\max_{k = 0,...,d}\vert\widehat{\mathbf{b}}_k\vert_2^2\times \sum_{i = 1}^T \left(\sum_{k = 0}^d\vert \widehat{\sigma}_k - \sigma_k\vert\right)^2\\
+ 2\max_{k = 0,...,d}\vert \widehat{\mathbf{b}}_{k} - \mathbf{b}_{k}\vert_2^2\sum_{i = 1}^T\left( \sum_{k = 0}^d \vert X_iX_{i - k} - \sigma_k\vert\right)^2\\
\leq 2(d + 1)\max_{k = 0,...,d}\vert\widehat{\mathbf{b}}_k\vert_2^2\times T\sum_{k = 0}^d \vert\widehat{\sigma}_k - \sigma_k\vert^2\\
 + 2(d + 1)\max_{k = 0,...,d}\vert \widehat{\mathbf{b}}_{k} - \mathbf{b}_{k}\vert_2^2\sum_{i = 1}^T\sum_{k = 0}^d  (X_iX_{i - k} - \sigma_k) ^2 = O_p(1)
\end{aligned}
\label{eq.zeta_minus_zeta}
\end{equation}
we have
\begin{equation}
\begin{aligned}
\vert\sum_{i_1 = j_1 + 1}^T\sum_{i_2 = j_2 + 1}^T K\left(\frac{i_1 - i_2}{k_T}\right)Z_{i_1, j_1}(\widehat{Z}_{i_2, j_2} - Z_{i_2, j_2})\vert = O_p(k_T\times \sqrt{T})\\
\text{and } \vert\sum_{i_1 = j_1 + 1}^T\sum_{i_2 = j_2 + 1}^T K\left(\frac{i_1 - i_2}{k_T}\right)(\widehat{Z}_{i_2, j_2} - Z_{i_2, j_2})(\widehat{Z}_{i_1, j_1} - Z_{i_1, j_1})\vert = O_p(k_T)
\end{aligned}
\end{equation}
For $d = O(1)$, from \eqref{eq.upp} and \eqref{eq.upp2}, we have $\vert\sum_{i_1 = 1}^{j_1}\sum_{i_2 = j_2 + 1}^T K\left(\frac{i_1 - i_2}{k_T}\right)Z_{i_1, j_1}Z_{i_2, j_2}\vert = O_p(k_T\times \sqrt{T})$ and $\vert\sum_{i_1 = 1}^{j_1}\sum_{i_2 = 1}^{j_2} K\left(\frac{i_1 - i_2}{k_T}\right)Z_{i_1, j_1}Z_{i_2, j_2}\vert = O_p(k_T)$. This implies \eqref{eq.third}.
\end{proof}

\begin{proof}[Proof of Theorem \ref{thm.bootstrap}]
(i). According to \eqref{eq.bootSigma},
\begin{equation}
\sqrt{T}\left(\widehat{\sigma}^*_j - \widehat{\sigma}_j\right) = \frac{1}{\sqrt{T}}\sum_{i = j + 1}^T \widehat{\epsilon}_i^{(j)}\varepsilon_i\
\text{for } j\in\mathcal{H}
\end{equation}
are joint normal random variables with
\begin{equation}
\mathbf{E}^*\frac{1}{\sqrt{T}}\sum_{i = j + 1}^T \widehat{\epsilon}_i^{(j)}\varepsilon_i = \frac{1}{\sqrt{T}}\sum_{i = j + 1}^T\widehat{\epsilon}_i^{(j)}\mathbf{E}\varepsilon_i = 0
\end{equation}
and
\begin{equation}
\begin{aligned}
\mathbf{E}^*\left(\frac{1}{\sqrt{T}}\sum_{i = j_1 + 1}^T \widehat{\epsilon}_i^{(j_1)}\varepsilon_i\right)\times
\left(\frac{1}{\sqrt{T}}\sum_{i = j_2 + 1}^T \widehat{\epsilon}_i^{(j_2)}\varepsilon_i\right)
= \frac{1}{T}\sum_{i_1 = j_1 + 1}^T\sum_{i_2 = j_2 + 1}^T\widehat{\epsilon}_i^{(j_1)} \widehat{\epsilon}_i^{(j_2)}\mathbf{E}^*\varepsilon_{i_1}
\varepsilon_{i_2}\\
= \frac{1}{T}\sum_{i_1 = j_1 + 1}^T\sum_{i_2 = j_2 + 1}^T(X_{i_1}X_{i_1 - j_1} - \widehat{\sigma}_{j_1}) (X_{i_2}X_{i_2 - j_2} - \widehat{\sigma}_{j_2}) K\left(\frac{i_1 - i_2}{k_T}\right)
\end{aligned}
\end{equation}
From Lemma \ref{lemma.consistent_variance}, lemma \ref{lemma.normal_property} and \eqref{eq.xi}, define
\begin{equation}
\begin{aligned}
\Delta = \max_{j_1, j_2\in\mathcal{H}}\vert
\frac{1}{T}\sum_{i_1 = j_1 + 1}^T\sum_{i_2 = j_2 + 1}^T(X_{i_1}X_{i_1 - j_1} - \widehat{\sigma}_{j_1}) (X_{i_2}X_{i_2 - j_2} - \widehat{\sigma}_{j_2}) K\left(\frac{i_1 - i_2}{k_T}\right)\\
- \frac{1}{T}\sum_{i_1 = 1}^T\sum_{i_2 = 1}^T\mathbf{E} (X_{i_1}X_{i_1 - j_1} - \sigma_{j_1}) (X_{i_2}X_{i_2 - j_2} - \sigma_{j_2})
\vert\\
= O_p\left(v_T\times T^{\alpha_X\beta_X} + k_T\times T^{\frac{8\beta_X}{m} + \alpha_X\beta_X - \frac{1}{2}}\right)
\end{aligned}
\end{equation}

 we have
\begin{equation}
\begin{aligned}
\sup_{x\in\mathbf{R}}\vert Prob^*(\sqrt{T}\max_{j\in\mathcal{H}}\vert \widehat{\sigma}^*_j - \widehat{\sigma}_j\vert\leq x) - H_\sigma(x)\vert\\
= O_p\left(T^{\alpha_X\beta_X}\Delta^{1/6} + \Delta^{1/3}T^{\alpha_X\beta_X} + \Delta^{1/3}\log^3(T) + \Delta^{1/3}T^{\alpha_X\beta_X}\sqrt{\log(T)}\right)
\end{aligned}
\end{equation}
and we prove \eqref{eq.Boot_delta_sigma}.

(ii). According to \eqref{eq.bootSigma},
\begin{equation}
\begin{aligned}
\sqrt{T}\left(\widehat{\rho}^*_j - \widehat{\rho}_j\right) = \frac{\widehat{\sigma}_0\sum_{i = j + 1}^T\widehat{\epsilon}_i^{(j)}\varepsilon_i - \widehat{\sigma}_j\sum_{i = 1}^T \widehat{\epsilon}_i^{(0)}\varepsilon_i}{\sqrt{T}\widehat{\sigma}_0\times\left(\widehat{\sigma}_0 + \frac{1}{T}\sum_{i = 1}^T
\widehat{\epsilon}_i^{(0)}\varepsilon_i\right)}
\end{aligned}
\end{equation}

Define $\Vert\Vert_m^* = \left(\mathbf{E}\vert \cdot \vert^m|X_1,...,X_T\right)^{1/m}$ as the `$m$ - norm' in the bootstrap world,
recall that $\varepsilon_i$ have joint normal distribution,
\begin{equation}
\begin{aligned}
\Vert \sum_{i = j + 1}^T \widehat{\epsilon}_i^{(j)}\varepsilon_i\Vert_{m/4}^*\leq C\Vert \sum_{i = j + 1}^T \widehat{\epsilon}_i^{(j)}\varepsilon_i\Vert_{2}^* = C\sqrt{\sum_{i_1 = j + 1}^T\sum_{i_2 = j + 1}^T \widehat{\epsilon}_{i_1}^{(j)}
\widehat{\epsilon}_{i_2}^{(j)}K\left(\frac{i_1 - i_2}{k_T}\right)}\\
\leq C^\prime\sqrt{k_T}\sqrt{\sum_{i = 1}^T \widehat{\epsilon}_i^{(j)2}}\\
\Rightarrow \Vert\  \Vert \sum_{i = j + 1}^T \widehat{\epsilon}_i^{(j)}\varepsilon_i\Vert_{m/4}^*\ \Vert_{m/2}
\leq C^\prime\sqrt{k_T}\times\sqrt{\Vert \sum_{i = 1}^T \widehat{\epsilon}_i^{(j)2}\Vert_{m/4}}
\end{aligned}
\label{eq.half_Cov}
\end{equation}
Form Lemma \ref{lemma.recognize}, $\{X_iX_{i - j}- \sigma_j\}_{i\in\mathbf{Z}}$ are $(\frac{m}{2}, \alpha - 1, \alpha_X\beta_X)$ - medium range dependent random
variables. So  from \eqref{eq.maxAUTO}
\begin{equation}
\begin{aligned}
\Vert\sum_{i = 1}^T \widehat{\epsilon}_{i}^{(j)2}\Vert_{m/4}\leq 2\Vert\sum_{i = 1}^T (X_iX_{i - j} - \sigma_j)^2\Vert_{m/4} + 2T\Vert (\sigma_j - \widehat{\sigma}_j)^2 \Vert_{m/4} \leq CT\\
\Rightarrow \max_{j\in\mathcal{I}}\Vert\sum_{i = j+1}^T
\widehat{\epsilon}_i^{(j)}\varepsilon_i\Vert^*_{m/4} = O_p(\sqrt{k_T}\times T^{\frac{1}{2} + \frac{2\beta_X}{m}})\\
\text{and }  \Vert \sum_{i = 1}^T \widehat{\epsilon}_i^{(0)}\varepsilon_i\Vert_{m/4}^* = O_p(\sqrt{k_T}\times T^{1/2})
\end{aligned}
\label{eq.epsilonNorm}
\end{equation}
For any given $\xi > 0$, we choose $\delta = C_\xi\times \sqrt{k_T}\times T^{\frac{6\beta_X}{m} - \frac{1}{2}}$ and $\delta^\prime  = C_\xi\sqrt{k_T}\times T^{-\frac{1}{2}   }$ with sufficiently large $C_\xi$, then
\begin{equation}
\begin{aligned}
Prob^*\left(\max_{j\in\mathcal{I}}\vert\frac{1}{T}\sum_{i = j + 1}^T \widehat{\epsilon}_i^{(j)}\varepsilon_i\vert > \delta\right)
\leq \frac{\sum_{j\in\mathcal{I}}\Vert  \frac{1}{T}\sum_{i = j + 1}^T \widehat{\epsilon}_i^{(j)}\varepsilon_i\Vert_{m/4}^{*m/4}}{\delta^{m/4}}
\leq \left(\frac{C^\prime}{C_\xi}\right)^{m/4}\\
\text{and   } Prob^*\left(\vert\frac{1}{T}\sum_{i =  1}^T \widehat{\epsilon}_i^{(0)}\varepsilon_i\vert > \delta^\prime\right)\leq \frac{\Vert \frac{1}{T}\sum_{i =  1}^T \widehat{\epsilon}_i^{(0)}\varepsilon_i\Vert_{m/4}^{*m/4}}{\delta^{\prime \frac{m}{4}}}\leq \left(\frac{C^\prime}{C_\xi}\right)^{m/4}
\end{aligned}
\end{equation}
with probability at least $1 - \xi$.  On the other hand, for $j\leq d$,
\begin{equation}
\begin{aligned}
\Vert \sum_{i = 1}^j \widehat{\epsilon}_i^{(0)}\varepsilon_i\Vert_{m/4}^*\leq C\Vert \sum_{i = 1}^j \widehat{\epsilon}_i^{(0)}\varepsilon_i\Vert_{2}^*
\leq C^\prime\sqrt{k_T}\times \sqrt{\sum_{i = 1}^d \widehat{\epsilon}_i^{(0)2}}\\
\Rightarrow \Vert\ \Vert \sum_{i = 1}^j \widehat{\epsilon}_i^{(0)}\varepsilon_i\Vert_{m/4}^*\ \Vert_{m/2}\leq C^\prime\sqrt{k_T}\sqrt{\Vert\sum_{i = 1}^d \widehat{\epsilon}_i^{(0)2}\Vert_{m/4}}
\end{aligned}
\end{equation}
From \eqref{eq.epsilonNorm}, define $\delta^{\prime\prime} = C_\xi\sqrt{k_T}\times T^{\frac{6\beta_X}{m} + \frac{\beta_X}{2} - \frac{1}{2}}$ with sufficiently
large $C_\xi$, we have
\begin{equation}
\begin{aligned}
\max_{j\in\mathcal{I}}\Vert\sum_{i = 1}^j
\widehat{\epsilon}_i^{(0)}\varepsilon_i\Vert^*_{m/4} = O_p(\sqrt{k_T}\times T^{\frac{2\beta_X}{m} + \frac{\beta_X}{2}})\\
\text{and  }  Prob^*\left( \max_{j\in\mathcal{I}}\vert\frac{1}{\sqrt{T}}\sum_{i = 1}^j
\widehat{\epsilon}_i^{(0)}\varepsilon_i\vert > \delta^{\prime\prime} \right)\leq \left(\frac{C^\prime}{C_\xi}\right)^{m/4}
\end{aligned}
\end{equation}
with probability at least $1 - \xi$.

In particular, define
\begin{equation}
\widehat{Z}^*_{i,j} = -\frac{\widehat{\sigma}_j}{\widehat{\sigma}^2_0}\widehat{\epsilon}^{(0)}_i\varepsilon_i
+ \frac{1}{\widehat{\sigma}_0} \widehat{\epsilon}_i^{(j)}\varepsilon_i\ \text{and } \widehat{s}^* = \frac{1}{T}\sum_{i = 1}^T
\widehat{\epsilon}_i^{(0)}\varepsilon_i
\end{equation}
and set $\delta_0 = \sqrt{T}\delta\delta^\prime + \sqrt{T}\delta^{\prime2} + \delta^{\prime\prime}$, we have for sufficiently large $T$, with probability at least $1 - \xi$
\begin{equation}
\begin{aligned}
\max_{j\in\mathcal{I}}\vert \sqrt{T}(\widehat{\rho}^*_j - \widehat{\rho}_j) - \frac{1}{\sqrt{T}}\sum_{i = j + 1}^T Z^*_{i, j}\vert
\leq \max_{j\in\mathcal{I}}\vert \frac{1}{\sqrt{T}}\sum_{i = j + 1}^T \widehat{\epsilon}_i^{(j)}\varepsilon_i\vert\times
\frac{\vert\widehat{s}^*\vert}{\widehat{\sigma}_0\times (\widehat{\sigma}_0 + \widehat{s}^*)}\\
+ \max_{j\in\mathcal{I}}\frac{1}{\sqrt{T}}\vert\sum_{i = 1}^T\widehat{\epsilon}_i^{(0)}\varepsilon_i\vert
\times \frac{\vert\widehat{\sigma}_j\vert\times \vert\widehat{s}^*\vert}{\widehat{\sigma}^2_0(\widehat{\sigma}_0 + \widehat{s}^*)}
+ \max_{j\in\mathcal{I}}\frac{1}{\sqrt{T}}\frac{\vert\widehat{\sigma}_j\vert}{\widehat{\sigma}_0^2}\vert\sum_{i = 1}^j\widehat{\epsilon}_i^{(0)}\varepsilon_i\vert\\
\Rightarrow Prob^*\left(\max_{j\in\mathcal{I}}\vert \sqrt{T}(\widehat{\rho}^*_j - \widehat{\rho}_j) - \frac{1}{\sqrt{T}}\sum_{i = j + 1}^T Z^*_{i, j}\vert > \delta_0\right)\\
\leq Prob^*\left(\frac{2\vert\widehat{s}^*\vert}{\widehat{\sigma}^2_0}\max_{j\in\mathcal{I}}\vert\frac{1}{\sqrt{T}}\sum_{i =  j + 1}^T \widehat{\epsilon}^{(j)}_i\varepsilon_i\vert > \sqrt{T}\delta\delta^\prime\right)\\
+ Prob^*\left(2\sqrt{T}\vert\widehat{s}^*\vert^2\times \max_{j\in\mathcal{I}}\frac{\vert\widehat{\sigma}_j\vert}{\widehat{\sigma}_0^3} > \sqrt{T}\delta^{\prime2}\right)\\
+ Prob^*\left(\max_{j\in\mathcal{I}}\frac{\vert\widehat{\sigma}_j\vert}{\widehat{\sigma}_0^2}\times \max_{j\in\mathcal{I}}
\frac{1}{\sqrt{T}}\vert\sum_{i = 1}^j \widehat{\epsilon}_i^{(0)}\varepsilon_i\vert > \delta^{\prime\prime}\right)
\end{aligned}
\label{eq.change_autocorrelation}
\end{equation}
For $k_T \times T^{\frac{6\beta_X}{m} - \frac{1}{2}} = o(T^{-\frac{2\beta_X}{m} - 7\alpha_X\beta_X})$  and $\sqrt{k_T}\times T^{\frac{6\beta_X}{m} + \frac{\beta_X}{2} - \frac{1}{2}} = o(T^{-5\alpha_X\beta_X})$ as $T\to\infty$, we have

 \noindent $ Prob^*\left(\max_{j\in\mathcal{I}}\vert \sqrt{T}(\widehat{\rho}^*_j - \widehat{\rho}_j) - \frac{1}{\sqrt{T}}\sum_{i = j + 1}^T Z^*_{i, j}\vert > T^{-5\alpha_X\beta_X}\right) = o_p(1)$. From Lemma \ref{lemma.consistent_variance} and equation \ref{lemma.normal_property},
\begin{equation}
\begin{aligned}
\sup_{x\in\mathbf{R}}\vert Prob^*\left(\max_{j\in\mathcal{I}}\vert\frac{1}{\sqrt{T}}\sum_{i = j + 1}^T Z^*_{i, j}\vert\leq x\right) - H_\rho(x)\vert\\
= O_p\left((v_T\times T^{7\alpha_X\beta_X})^{1/6} + (k_T\times T^{\frac{8\beta_X}{m} + 7\alpha_X\beta_X - \frac{1}{2}})^{1/6}\right)\\
\end{aligned}
\end{equation}
Therefore
\begin{equation}
\begin{aligned}
\vert Prob^*\left(\max_{j\in\mathcal{I}}\vert\sqrt{T}(\widehat{\rho}^*_j - \widehat{\rho}_j)\vert\leq x\right) - H_\rho(x)\vert\\
\leq Prob^*\left(\max_{j\in\mathcal{I}}\vert \sqrt{T}(\widehat{\rho}^*_j - \widehat{\rho}_j) - \frac{1}{\sqrt{T}}\sum_{i = j + 1}^T Z^*_{i, j}\vert > T^{-5\alpha_X\beta_X}\right)\\
+ \sup_{x\in\mathbf{R}}\vert Prob^*\left(\max_{j\in\mathcal{I}}\vert \frac{1}{\sqrt{T}}\sum_{i = j + 1}^T Z^*_{i, j}\vert\leq x\right) - H_\rho(x)\vert
+ \sup_{x\in\mathbf{R}}\vert H_\rho(x + T^{-5\alpha_X\beta_X}) - H_\rho(x)\vert\\
\Rightarrow \sup_{x\in\mathbf{R}}\vert Prob^*\left(\max_{j\in\mathcal{I}}\vert\sqrt{T}(\widehat{\rho}^*_j - \widehat{\rho}_j)\vert\leq x\right) - H_\rho(x)\vert = o_p(1)
\end{aligned}
\end{equation}
and we prove \eqref{eq.Boot_delta_rho}.

(iii). For $p = O(1)$, $\{X_iX_{i - j} - \sigma_j\}_{i\in\mathbf{Z}}$ are $(m, \alpha - 1)$ - short range dependent random variables.
Therefore,  from \eqref{eq.half_Cov} and \eqref{eq.epsilonNorm}, for any given $\xi>0$, choose sufficiently large constant $C_\xi$ and set
$\delta = C_\xi\sqrt{k_T}\times T^{-1/2}$,
\begin{equation}
\begin{aligned}
 \max_{j = 1,...,p}\Vert\sum_{i = j+1}^T
\widehat{\epsilon}_i^{(j)}\varepsilon_i\Vert^*_{m/4} = O_p(\sqrt{k_T}\times T^{\frac{1}{2}})\\
Prob^*\left(\max_{j = 1,..., p}\vert\frac{1}{T}\sum_{i = j + 1}^T \widehat{\epsilon}_i^{(j)}\varepsilon_i\vert > \delta\right)\leq \left(\frac{C^\prime}{C_\xi}\right)^{m/4}
\end{aligned}
\end{equation}
with probability at least $1 - \xi$. From \eqref{eq.bootSigma}, section 0.9.7 in \cite{MR2978290} and \eqref{eq.defBootCov},
\begin{equation}
\begin{aligned}
\vert \widehat{\Sigma}^* - \widehat{\Sigma}\vert_2\leq 2\sum_{l = 0}^{p}\vert \frac{1}{T}\sum_{i = l + 1}^T \widehat{\epsilon}_i^{(l)}\varepsilon_i\vert
\leq 2(p + 1)C_\xi\sqrt{k_T}T^{-1/2}\\
\text{and } \vert\widehat{\gamma}^* - \widehat{\gamma}\vert_2\leq \sum_{l = 0}^{p}\vert \frac{1}{T}\sum_{i = l + 1}^T \widehat{\epsilon}_i^{(l)}\varepsilon_i\vert\leq (p + 1)C_\xi\sqrt{k_T}T^{-1/2}\\
\Rightarrow Prob^*\left(\vert \widehat{\Sigma}^* - \widehat{\Sigma}\vert_2\leq 2(p + 1)C_\xi\sqrt{k_T}T^{-1/2}\right)\geq 1 - \xi\\
\text{and } Prob^*\left(\vert\widehat{\gamma}^* - \widehat{\gamma}\vert_2\leq (p + 1)C_\xi\sqrt{k_T}T^{-1/2}\right)\geq 1 - \xi
\end{aligned}
\label{eq.prob_SEC2}
\end{equation}
with probability at least $1 - \xi$. From \eqref{eq.diff_Covariance_Matrix} and corollary 5.6.16 in \cite{MR2978290}, for sufficiently large $T$,
$\widehat{\Sigma}$'s smallest eigenvalue is greater than $c/2 > 0$ with probability at least $1 - \xi$; and
\begin{equation}
\begin{aligned}
\sqrt{T}\vert\widehat{\Sigma}^{*\dagger} - \widehat{\Sigma}^{-1}  + \widehat{\Sigma}^{-1}(\widehat{\Sigma}^* - \widehat{\Sigma})\widehat{\Sigma}^{-1}\vert_2\leq \sqrt{T}\vert\widehat{\Sigma}^{-1}\vert_2\sum_{k = 2}^\infty\vert \widehat{\Sigma}^{-1}(\widehat{\Sigma}^* - \widehat{\Sigma})\vert_2^k\\
\Rightarrow Prob^*\left(\sqrt{T}\vert\widehat{\Sigma}^{*\dagger} - \widehat{\Sigma}^{-1}  + \widehat{\Sigma}^{-1}(\widehat{\Sigma}^* - \widehat{\Sigma})\widehat{\Sigma}^{-1}\vert_2\leq Ck_T\times T^{-1/2}\right)\geq 1 - \xi
\end{aligned}
\end{equation}
with probability at least $1 - \xi$. In particular,
\begin{equation}
\begin{aligned}
\vert\sqrt{T}(\widehat{a}^* - \widehat{a}) - \sqrt{T}\left(\widehat{\Sigma}^{-1}(\widehat{\gamma}^* - \widehat{\gamma}) - \widehat{\Sigma}^{-1}(\widehat{\Sigma}^* - \widehat{\Sigma})\widehat{\Sigma}^{-1}\widehat{\gamma})\right)\vert_2\\
\leq \sqrt{T}\vert\widehat{\Sigma}^{*\dagger} - \widehat{\Sigma}^{-1}\vert_2\times \vert\widehat{\gamma}^* - \widehat{\gamma}\vert_2
+ \vert\widehat{\gamma}\vert_2\times \sqrt{T}\vert\widehat{\Sigma}^{*\dagger} - \widehat{\Sigma}^{-1} + \widehat{\Sigma}^{-1}(\widehat{\Sigma}^* - \widehat{\Sigma})\widehat{\Sigma}^{-1}\vert_2\\
\Rightarrow Prob^*\left(\vert\sqrt{T}(\widehat{a}^* - \widehat{a}) - \sqrt{T}\left(\widehat{\Sigma}^{-1}(\widehat{\gamma}^* - \widehat{\gamma}) - \widehat{\Sigma}^{-1}(\widehat{\Sigma}^* - \widehat{\Sigma})\widehat{\Sigma}^{-1}\widehat{\gamma}\right)\vert_2\leq Ck_T\times T^{-1/2}\right)\\
\geq 1 - \xi
\end{aligned}
\label{eq.aroundA}
\end{equation}
with probability at least $1 - \xi$. On the  other hand, define $\widehat{b}_{jk}$ and $\widehat{\mathbf{b}}_0,...,\widehat{\mathbf{b}}_{p}$ as in Lemma \ref{lemma.consistent_variance},  we have
\begin{equation}
\begin{aligned}
\sqrt{T}\widehat{\Sigma}^{-1}(\widehat{\gamma}^* - \widehat{\gamma}) - \sqrt{T}\widehat{\Sigma}^{-1}(\widehat{\Sigma}^* - \widehat{\Sigma})
\widehat{\Sigma}^{-1}\widehat{\gamma} = \sqrt{T}\sum_{l = 0}^{p}(\widehat{\sigma}^*_l - \widehat{\sigma}_l)\widehat{\mathbf{b}}_l\\
= \frac{1}{\sqrt{T}}\sum_{i = 1}^T\sum_{l = 0}^{\min(i - 1, p)}\widehat{\mathbf{b}}_l\widehat{\epsilon}^{(l)}_i\varepsilon_i
\end{aligned}
\end{equation}
Define $\widehat{Z}^*_{i, j} = \sum_{l = 0}^{\min(i - 1, p)}\widehat{b}_{jl}\widehat{\epsilon}^{(l)}_i\varepsilon_i$, we have
\begin{equation}
\begin{aligned}
\mathbf{E}^*\left(\frac{1}{\sqrt{T}}\sum_{i  = 1}^T \widehat{Z}^*_{i, j_1}\right)\times \left(\frac{1}{\sqrt{T}}\sum_{i = 1}^T \widehat{Z}^*_{i, j_2}\right)\\
= \frac{1}{T}\sum_{i_1 = 1}^T\sum_{i_2 = 1}^T K\left(\frac{i_1 - i_2}{k_T}\right)\sum_{l_1 = 0}^{\min(i_1 - 1, p)}
\sum_{l_2 = 0}^{\min(i_2 -  1, p)}\widehat{b}_{j_1l_1}\widehat{b}_{j_2l_2}\widehat{\epsilon}_{i_1}^{(l_1)}\widehat{\epsilon}_{i_2}^{(l_2)}
\end{aligned}
\label{eq.crossX}
\end{equation}
Define $\widehat{Z}_{i,j}$ as in \eqref{eq.defZHAT}, we have
\begin{equation}
\begin{aligned}
\eqref{eq.crossX} = \frac{1}{T}\sum_{i_1 = p + 1}^T\sum_{i_2 = p + 1}^TK\left(\frac{i_1 - i_2}{k_T}\right)\widehat{Z}_{i_1, j_1}\widehat{Z}_{i_2, j_2}\\
+  \frac{1}{T}\sum_{i_1 = 1}^{p}\sum_{i_2 = p + 1}^T K\left(\frac{i_1 - i_2}{k_T}\right)\left(\sum_{l_1 = 0}^{i_1 - 1}\widehat{b}_{j_1l_1}\widehat{\epsilon}_{i_1}^{(l_1)}\right)\times \widehat{Z}_{i_2, j_2}\\
+  \frac{1}{T}\sum_{i_1 = p + 1}^T \sum_{i_2 = 1}^{p} K\left(\frac{i_1 - i_2}{k_T}\right)\widehat{Z}_{i_1, j_1}
\times\left(\sum_{l_2 = 0}^{i_2 - 1}\widehat{b}_{j_2l_2}\widehat{\epsilon}_{i_2}^{(l_2)}\right)\\
+ \frac{1}{T}\sum_{i_1 = 1}^{p}\sum_{i_2 = 1}^{p} K\left(\frac{i_1 - i_2}{k_T}\right)\left(\sum_{l_1 = 0}^{i_1 - 1}\widehat{b}_{j_1l_1}\widehat{\epsilon}_{i_1}^{(l_1)}\right)\times\left(\sum_{l_2 = 0}^{i_2 - 1}\widehat{b}_{j_2l_2}\widehat{\epsilon}_{i_2}^{(l_2)}\right)
\end{aligned}
\end{equation}
Therefore,
\begin{equation}
\begin{aligned}
\vert \mathbf{E}^*\left(\frac{1}{\sqrt{T}}\sum_{i  = 1}^T \widehat{Z}^*_{i, j_1}\right)\times \left(\frac{1}{\sqrt{T}}\sum_{i = 1}^T \widehat{Z}^*_{i, j_2}\right) - \frac{1}{T}\sum_{i_1 = j_1 + 1}^T\sum_{i_2 = j_2 + 1}^TK\left(\frac{i_1 - i_2}{k_T}\right)\widehat{Z}_{i_1,j_1}\widehat{Z}_{i_2, j_2}\vert\\
\leq \vert  \frac{1}{T}\sum_{i_1 = 1}^{p}\sum_{i_2 = p + 1}^T K\left(\frac{i_1 - i_2}{k_T}\right)\left(\sum_{l_1 = 0}^{i_1 - 1}\widehat{b}_{j_1l_1}\widehat{\epsilon}_{i_1}^{(l_1)}\right)\times \widehat{Z}_{i_2, j_2}\vert\\
+ \vert \frac{1}{T}\sum_{i_1 = p + 1}^T \sum_{i_2 = 1}^{p} K\left(\frac{i_1 - i_2}{k_T}\right)\widehat{Z}_{i_1, j_1}
\times\left(\sum_{l_2 = 0}^{i_2 - 1}\widehat{b}_{j_2l_2}\widehat{\epsilon}_{i_2}^{(l_2)}\right)\vert\\
+ \vert \frac{1}{T}\sum_{i_1 = 1}^{p}\sum_{i_2 = 1}^{p} K\left(\frac{i_1 - i_2}{k_T}\right)\left(\sum_{l_1 = 0}^{i_1 - 1}\widehat{b}_{j_1l_1}\widehat{\epsilon}_{i_1}^{(l_1)}\right)\times\left(\sum_{l_2 = 0}^{i_2 - 1}\widehat{b}_{j_2l_2}\widehat{\epsilon}_{i_2}^{(l_2)}\right)\vert\\
+ \vert\frac{1}{T}\sum_{i_1 = j_1 + 1}^{p}\sum_{i_2 = p + 1}^T K\left(\frac{i_1 - i_2}{k_T}\right)\widehat{Z}_{i_1, j_1}\widehat{Z}_{i_2, j_2}\vert\\
+ \vert\frac{1}{T}\sum_{i_1 = p + 1}^T\sum_{i_2 = j_2 + 1}^{p} K\left(\frac{i_1 - i_2}{k_T}\right)\widehat{Z}_{i_1, j_1}\widehat{Z}_{i_2, j_2}\vert\\
+ \vert\frac{1}{T}\sum_{i_1 = j_1 + 1}^{p}\sum_{i_2 = j_2 + 1}^{p} K\left(\frac{i_1 - i_2}{k_T}\right)\widehat{Z}_{i_1, j_1}\widehat{Z}_{i_2, j_2}\vert
\end{aligned}
\end{equation}
From \eqref{eq.zeta_minus_zeta}, \eqref{eq.delta_beta} and \eqref{eq.maxAUTO},
\begin{equation}
\begin{aligned}
\sqrt{\sum_{i =  1}^T \widehat{Z}_{i, j}^2}\leq 2\sqrt{\sum_{i = 1}^T Z_{i, j}^2 + \sum_{i = 1}^T (\widehat{Z}_{i, j} - Z_{i,j})^2} = O_p(\sqrt{T})\\
\sqrt{\sum_{i = 1}^{p} \widehat{Z}_{i, j}^2}\leq 2\sqrt{\sum_{i = 1}^{p} Z_{i, j}^2 + \sum_{i = 1}^{p} (\widehat{Z}_{i, j} - Z_{i,j})^2} = O_p(1)\\
\text{and  for $i = 1,...,p$, }  \left(\sum_{l = 0}^{i - 1}\widehat{b}_{jl}\widehat{\epsilon}_{i}^{(l)}\right)^2\leq 4(\sum_{l = 0}^{i - 1}b_{jl}\epsilon_i^{(l)})^2
+ 4\left(\sum_{l = 0}^{i - 1}b_{jl}(\widehat{\sigma}_l - \sigma_l)\right)^2\\
+ 4\left(\sum_{l = 0}^{i - 1}(\widehat{b}_{jl} - b_{jl})(\widehat{\sigma}_l - \sigma_l)\right)^2 + 4\left(\sum_{l = 0}^{i - 1}
(\widehat{b}_{jl} - b_{jl})\epsilon_i^{(l)}\right)^2 = O_p(1)
\end{aligned}
\end{equation}
Therefore,

\begin{equation}
\begin{aligned}
\vert \sum_{i_1 = 1}^{p}\sum_{i_2 = p + 1}^T K\left(\frac{i_1 - i_2}{k_T}\right)\left(\sum_{l_1 = 0}^{i_1 - 1}\widehat{b}_{j_1l_1}\widehat{\epsilon}_{i_1}^{(l_1)}\right)\times \widehat{Z}_{i_2, j_2}\vert\\
\leq 2\left(\sum_{j = 0}^\infty K\left(\frac{j}{k_T}\right)\right)\times \sqrt{\sum_{i_1 = 1}^{p} \left(\sum_{l_1 = 0}^{i_1 - 1}\widehat{b}_{j_1l_1}\widehat{\epsilon}_{i_1}^{(l_1)}\right)^2}\times \sqrt{\sum_{i_2 =  1}^T \widehat{Z}_{i_2, j_2}^2} = O_p\left(k_T\sqrt{T}\right)\\
\vert\sum_{i_1 = 1}^{p}\sum_{i_2 = 1}^{p} K\left(\frac{i_1 - i_2}{k_T}\right)\left(\sum_{l_1 = 0}^{i_1 - 1}\widehat{b}_{j_1l_1}\widehat{\epsilon}_{i_1}^{(l_1)}\right)\times\left(\sum_{l_2 = 0}^{i_2 - 1}\widehat{b}_{j_2l_2}\widehat{\epsilon}_{i_2}^{(l_2)}\right)\vert\\
\leq 2\left(\sum_{j = 0}^\infty K\left(\frac{j}{k_T}\right)\right)\times \sqrt{\sum_{i_1 = 1}^{p} \left(\sum_{l_1 = 0}^{i_1 - 1}\widehat{b}_{j_1l_1}\widehat{\epsilon}_{i_1}^{(l_1)}\right)^2}\times \sqrt{\sum_{i_2 = 1}^{p} \left(\sum_{l_2 = 0}^{i_2 - 1}\widehat{b}_{j_2l_2}\widehat{\epsilon}_{i_2}^{(l_2)}\right)^2} = O_p(k_T)\\
\vert\sum_{i_1 = j_1 + 1}^{p}\sum_{i_2 = p + 1}^T K\left(\frac{i_1 - i_2}{k_T}\right)\widehat{Z}_{i_1, j_1}\widehat{Z}_{i_2, j_2}\vert\\
\leq 2\left(\sum_{j = 0}^\infty K\left(\frac{j}{k_T}\right)\right)\sqrt{\sum_{i_1 = j_1 + 1}^{p} \widehat{Z}_{i_1, j_1}^2}\times
\sqrt{\sum_{i_2 = 1}^T\widehat{Z}_{i_2, j_2}^2} = O_p(k_T\sqrt{T})\\
\vert\sum_{i_1 = j_1 + 1}^{p}\sum_{i_2 = j_2 + 1}^{p} K\left(\frac{i_1 - i_2}{k_T}\right)\widehat{Z}_{i_1, j_1}\widehat{Z}_{i_2, j_2}\vert\\
\leq 2\left(\sum_{j = 0}^\infty K\left(\frac{j}{k_T}\right)\right)\sqrt{\sum_{i_1 = j_1 + 1}^{p} \widehat{Z}_{i_1, j_1}^2 }
\times \sqrt{\sum_{i_2 = j_2 + 1}^{p} \widehat{Z}_{i_2, j_2}^2} = O_p(k_T)
\end{aligned}
\end{equation}
Therefore,
\begin{equation}
\begin{aligned}
\vert \mathbf{E}^*\left(\frac{1}{\sqrt{T}}\sum_{i  = 1}^T \widehat{Z}^*_{i, j_1}\right)\times \left(\frac{1}{\sqrt{T}}\sum_{i = 1}^T \widehat{Z}^*_{i, j_2}\right) - \frac{1}{T}\sum_{i_1 = j_1 + 1}^T\sum_{i_2 = j_2 + 1}^TK\left(\frac{i_1 - i_2}{k_T}\right)\widehat{Z}_{i_1,j_1}\widehat{Z}_{i_2, j_2}\vert\\
= O_p(\frac{k_T}{\sqrt{T}})
\end{aligned}
\end{equation}
From lemma \ref{lemma.consistent_variance} and lemma \ref{lemma.normal_property}, we have
\begin{equation}
\begin{aligned}
\sup_{x\in\mathbf{R}}\vert Prob^*\left(\max_{j = 1,...,p}\vert \frac{1}{\sqrt{T}}\sum_{i  = 1}^T \widehat{Z}^*_{i, j}\vert\leq x\right) - H_a(x)\vert = o_p(1)
\end{aligned}
\end{equation}
From \eqref{eq.aroundA}, for any given $\xi > 0$; with probability at least $1 - \xi$
\begin{equation}
\begin{aligned}
Prob^*\left(\max_{j = 1,...,p}\sqrt{T}\vert\widehat{a}^*_j - \widehat{a}_j\vert\leq x\right)\leq \xi\\
 + Prob^*\left(\max_{j = 1,...,p}\vert \frac{1}{\sqrt{T}}\sum_{i  = 1}^T \widehat{Z}^*_{i, j}\vert \leq x + Ck_T\times T^{-1/2}\right)\\
 \text{and } Prob^*\left(\max_{j = 1,...,p}\sqrt{T}\vert\widehat{a}^*_j - \widehat{a}_j\vert\leq x\right)\geq -\xi\\
 + Prob^*\left(\max_{j = 1,...,p}\vert \frac{1}{\sqrt{T}}\sum_{i  = 1}^T \widehat{Z}^*_{i, j}\vert \leq x - Ck_T\times T^{-1/2}\right)
\end{aligned}
\end{equation}
From Lemma \ref{lemma.normal_property}, we prove \eqref{eq.Boot_delta_a}.

\end{proof}
\end{document}